\documentclass[a4paper]{article}
\usepackage[top=1in, bottom=1in, left=1.1in, right=1.8in]{geometry}
\usepackage{amsfonts}
\usepackage{amsmath}
\usepackage{amsthm,amssymb}
\usepackage{CJK,graphicx}
\usepackage{amscd}
\usepackage{amssymb}
\usepackage{mathrsfs}
\usepackage[all,cmtip]{xy}
\usepackage{lmodern}
\usepackage[Symbol]{upgreek}
\usepackage{bm}
\usepackage{eucal}
\usepackage[nopar]{lipsum}
\usepackage{rotating}
\usepackage{tikz}
\usepackage{calc}
\usepackage{tikz-cd}
\usetikzlibrary{calc}
\usetikzlibrary{decorations.markings}
\newsavebox\CBox
\newcommand\hcancel[2][0.5pt]{%
	\ifmmode\sbox\CBox{$#2$}\else\sbox\CBox{#2}\fi%
	\makebox[0pt][l]{\usebox\CBox}%
	\rule[0.5\ht\CBox-#1/2]{\wd\CBox}{#1}}

\newtheorem{theorem}{Theorem}[section]
\newtheorem{corollary}{Corollary}[section]
\newtheorem{definition}{Definition}[section]
\newtheorem{lemma}{Lemma}[section]
\newtheorem{proposition}{Proposition}[section]

\AtEndDocument{\bigskip{\footnotesize
		\textsc{King's College London}\\
		\textit{E-mail address}: \texttt{yin.li@kcl.ac.uk} \par
}}

\begin{document}

\title{\textbf{Koszul duality via suspending Lefschetz fibrations}}\author{Yin Li}\date{}\maketitle
\begin{abstract}
Let $M$ be a Liouville 6-manifold which is the smooth fiber of a Lefschetz fibration on $\mathbb{C}^4$ constructed by suspending a Lefschetz fibration on $\mathbb{C}^3$. We prove that for many examples including stabilizations of Milnor fibers of hypersurface cusp singularities, the compact Fukaya category $\mathcal{F}(M)$ and the wrapped Fukaya category $\mathcal{W}(M)$ are related through $A_\infty$-Koszul duality, by identifying them with cyclic and Calabi-Yau completions of the same quiver algebra. This implies the split-generation of the compact Fukaya category $\mathcal{F}(M)$ by vanishing cycles. Moreover, new examples of Liouville manifolds which admit quasi-dilations in the sense of Seidel-Solomon are obtained.
\end{abstract}

\section{Introduction}

\subsection{Background}

Let $M$ be a Liouville manifold obtained by attaching a cylindrical end $\mathbb{R}_+\times V$ to a Liouville domain $M^\mathrm{in}$, where $V=\partial M^\mathrm{in}$ is the contact boundary. Fix some coefficient field $\mathbb{K}$ to work with. Associated to $M$ there are two versions of Fukaya categories: the compact Fukaya category $\mathcal{F}(M)$ and the wrapped Fukaya category $\mathcal{W}(M)$. Under the assumption that $c_1(M)=0$, they are $\mathbb{Z}$-graded $A_\infty$-categories.

Recall that the objects of $\mathcal{F}(M)$ are oriented, $\mathit{Spin}$, closed exact Lagrangian submanifolds $L\subset M$ with vanishing Maslov class (when $\mathrm{char}(\mathbb{K})=2$, the orientable and $\mathit{Spin}$ assumptions on $L$ can be removed). Let $L_1,\dots,L_r$ be a finite collection of objects in $\mathcal{F}(M)$, using compactly supported Hamiltonian perturbations, one can always achieve that these Lagrangian submanifolds are intersecting transversally, so that the Floer cochain complexes $\mathit{CF}^\ast(L_i,L_j)$ are well-defined. The $A_\infty$-relations on the chain level have been established rigorously by Seidel in $\cite{ps1}$, which makes $\mathcal{F}(M)$ a well-defined $A_\infty$-category.

Since $M$ is non-compact, it is also natural to take into account certain non-compact Lagrangian submanifolds of $M$, and define an $A_\infty$-category $\mathcal{W}(M)$ which has possibly infinite dimensional morphism spaces given by the wrapped Floer cochain complexes $\mathit{CW}^\ast(L_i,L_j)$. To be concrete, the non-compact Lagrangian submanifolds which are allowed as objects of $\mathcal{W}(M)$ are those which are modelled on a cone $\mathbb{R}_+\times\Lambda$ over the cylindrical end $\mathbb{R}_+\times V$, where $\Lambda\subset V$ is a closed Legendrian submanifold. The $A_\infty$-structure on $\mathcal{W}(M)$ can be defined either by using a linear Hamiltonian function together with the telescope construction $\cite{as}$ or a quadratic Hamiltonian function with some appropriate rescalings of the Liouville flow $\cite{ma1}$.\bigskip

When $M=T^\ast Q$ is the cotangent bundle of a compact smooth manifold $Q$, the $A_\infty$-categories $\mathcal{F}(M)$ and $\mathcal{W}(M)$ have topological interpretations, whose most refined versions are due to Abouzaid ($\cite{ma2,ma3}$). Denote by $\mathcal{F}_M$ and $\mathcal{W}_M$ the $A_\infty$-algebras $\mathit{CF}^\ast(Q,Q)$ and $\mathit{CW}^\ast(T_q^\ast Q,T_q^\ast Q)$ respectively, we have the following quasi-isomorphisms:
\begin{equation}
\mathcal{F}_M\cong C^\ast(Q;\mathbb{K});\mathcal{W}_M\cong C_{-\ast}(\Omega_q Q;\mathbb{K}),
\end{equation}
where $C^\ast(Q;\mathbb{K})$ is the dg algebra of singular cochains on $Q$ and $C_{-\ast}(\Omega_q Q;\mathbb{K})$ is the dg algebra of chains on the based loop space $\Omega_q Q$.

A useful topological tool of studying the dg algebras $C^\ast(Q;\mathbb{K})$ and $C_{-\ast}(\Omega_q Q;\mathbb{K})$ is Adams' cobar construction $\cite{ja,ah}$. Recall that there are natural augmentations
\begin{equation}
\varepsilon_\mathcal{F}:C^\ast(Q;\mathbb{K})\rightarrow\mathbb{K},\varepsilon_\mathcal{W}:C_{-\ast}(\Omega_q Q;\mathbb{K})\rightarrow\mathbb{K},
\end{equation}
which make $C^\ast(Q;\mathbb{K})$ and $C_{-\ast}(\Omega_q Q;\mathbb{K})$ into augmented dg algebras, where $\varepsilon_\mathcal{F}$ is induced by the inclusion $\mathit{pt}\hookrightarrow Q$ of a point, and $\varepsilon_\mathcal{W}$ comes from the trivial local system $\pi_1(Q,\mathit{pt})\rightarrow\mathbb{K}$. It follows from the Eilenberg-Moore equivalence that
\begin{equation}\label{eq:em}
R\mathrm{Hom}_{C_{-\ast}(\Omega_q Q;\mathbb{K})}(\mathbb{K},\mathbb{K})\cong C^\ast(Q;\mathbb{K}).
\end{equation}
If we further assume that $Q$ is \textit{simply-connected}, it follows from Adams' cobar construction that there is another quasi-isomorphism
\begin{equation}
R\mathrm{Hom}_{C^\ast(Q;\mathbb{K})}(\mathbb{K},\mathbb{K})\cong C_{-\ast}(\Omega_q Q;\mathbb{K}),
\end{equation}
namely $C^\ast(Q;\mathbb{K})$ and $C_{-\ast}(\Omega_q Q;\mathbb{K})$ are Koszul dual as dg algebras.
\bigskip

Generalizations of the above Koszul duality in the context of symplectic topology have been obtained by Etg\"{u}-Lekili $\cite{etl1}$ and Ekholm-Lekili $\cite{ekl}$. More specifically, they considered symplectic manifolds $M$ which are plumbings of cotangent bundles $T^\ast Q_v$ of simply connected manifolds $Q_v$ according a tree $T=(T_0,T_1)$, where $T_0$ is the set of vertices, and $T_1$ is the set of edges. Denote by
\begin{equation}
\mathcal{F}_M:=\bigoplus_{v,w\in T_0}\mathit{CF}^\ast(Q_v,Q_w)
\end{equation}
the endomorphism algebra of the zero sections $Q_v\subset T^\ast Q_v,Q_w\subset T^\ast Q_w$ in the compact Fukaya category $\mathcal{F}(M)$, and by
\begin{equation}
\mathcal{W}_M:=\bigoplus_{v,w\in T_0}\mathit{CW}^\ast(L_v,L_w)
\end{equation}
the endomorphism algebra of cotangent fibers $L_v=T_q^\ast Q_v,L_w=T_q^\ast Q_w$ in the wrapped Fukaya category $\mathcal{W}(M)$. Up to quasi-isomorphism, these are strictly unital $A_\infty$-algebras over the semisimple ring $\Bbbk:=\bigoplus_{v\in T_0}\mathbb{K}e_v$, where $e_v$ is an idempotent in $\mathit{CF}^0(Q_v,Q_v)$ or $\mathit{CW}^0(L_v,L_v)$.

Note that $\mathcal{F}_M$ and $\mathcal{W}_M$ are equipped with augmentations $\varepsilon_{\mathcal{F}}:\mathcal{F}_M\rightarrow\Bbbk$ and $\varepsilon_{\mathcal{W}}:\mathcal{W}_M\rightarrow\Bbbk$, where $\varepsilon_{\mathcal{F}}$ is defined by projecting to $\Bbbk\cong\mathcal{F}_M^0$, while $\varepsilon_{\mathcal{W}}$ is induced from the exact Lagrangian filling $\left(\bigcup_{v\in T_0}Q_v\right)\cap D^{2n}$ of the Legendrian submanifold $\left(\bigcup_{v\in T_0}Q_v\right)\cap\partial D^{2n}\subset(S^{2n-1},\xi_\mathit{std})$. In $\cite{ekl,etl1}$ it is proved that there are quasi-isomorphisms
\begin{equation}\label{eq:Koszul}
R\mathrm{Hom}_{\mathcal{W}_M}(\Bbbk,\Bbbk)\cong\mathcal{F}_M,R\mathrm{Hom}_{\mathcal{F}_M}(\Bbbk,\Bbbk)\cong\mathcal{W}_M
\end{equation}
when
\begin{itemize}
	\item $\dim_\mathbb{R}(M)=4$ and $M$ is a plumbing of $T^\ast S^2$'s with $T=A_n$ or $D_n$ and $\mathrm{char}(\mathbb{K})\neq2$;
	\item $\dim_\mathbb{R}(M)\geq6$ and $M$ is a plumbing of $T^\ast Q_v$'s according to any tree $T$, with each $Q_v$ being simply-connected.
\end{itemize}

\paragraph{Conventions} 
\begin{itemize}
	\item In this paper, we will need to deal with a dg or $A_\infty$-algebra $\mathcal{A}$ which is bigraded, namely $\mathcal{A}=\bigoplus_{i,j}\mathcal{A}^{i,j}$ as a $\Bbbk$-bimodule, and the differential or $A_\infty$-operations changes only the first grading $i$, see Section \ref{section:Koszul}. In general, the Koszul dual $R\mathrm{Hom}_\mathcal{A}(\Bbbk,\Bbbk)$ of $\mathcal{A}$ regarded as a singly graded $A_\infty$-algebra with respect to the total grading $i+j$ differs from the Koszul dual of $\mathcal{A}$ regarded as a bigraded $A_\infty$-algebra. The latter case is the viewpoint taken in $\cite{lpwz}$. As an example, consider $\mathcal{A}=\mathbb{K}[x]/(x^2)$ with $|x|=(0,1)$. In the first case, $R\mathrm{Hom}_\mathcal{A}(\Bbbk,\Bbbk)$ is isomorphic to the ring of formal power series $\mathbb{K}[[y]]$ with $|y|=0$, while in the second case one gets the polynomial ring $\mathbb{K}[y]$ with $|y|=(1,-1)$. In order to distinguish between these two situations, we will keep the notation $R\mathrm{Hom}_\mathcal{A}(\Bbbk,\Bbbk)$ for taking the \textit{singly graded} Koszul dual of $\mathcal{A}$, and use $E(\mathcal{A})$ to stand for the \textit{bigraded} Koszul dual.
	\item We will always regard $\Bbbk$ as a left $\mathcal{F}_M$-module. One can equivalently view $\Bbbk$ as a right $\mathcal{F}_M$-module, so that in the second formula of (\ref{eq:Koszul}) above becomes $R\mathrm{Hom}_{\mathcal{F}_M^\mathit{op}}(\Bbbk,\Bbbk)\cong\mathcal{W}_M^\mathit{op}$.
\end{itemize}

However, we should remark that the Liouville manifolds which satisfy the Koszul duality (\ref{eq:Koszul}) form only a very restrictive class. Known counterexamples include cotangent bundles of most of the non-simply connected manifolds and their plumbings, and plumbings of $T^\ast S^2$'s according to a non-Dynkin tree $\cite{etl2}$. The common feature of these counterexamples is that the degree zero part of the symplectic cohomology $\mathit{SH}^\ast(M)$ is infinite dimensional. Since homological mirror symmetry predicts the ring isomorphism $\mathit{SH}^0(M)\cong H^0(M^\vee,\mathcal{O}_{M^\vee})$ when the mirror $M^\vee$ of $M$ is a smooth algebraic variety, many examples of Liouville manifolds with infinite dimensional $\mathit{SH}^0(M)$ arise from mirror symmetry. For example, the works of Gross-Hacking-Keel $\cite{ghk}$ and Abouzaid-Auroux-Katzarkov $\cite{aak}$ suggest that the $A_\infty$-algebras $\mathcal{F}_M$ and $\mathcal{W}_M$ are not Koszul dual when $M$ is a log Calabi-Yau surface or an affine conic bundle over $(\mathbb{C}^\ast)^{n-1}$. In the case when $M$ is the complement of an anticanonical divisor in some smooth projective variety, see $\cite{jp}$ and $\cite{gp2}$ for computations of $\mathit{SH}^0(M)$ using Morse-Bott spectral sequences. In general, the $A_\infty$-algebras $\mathcal{F}_M$ and $\mathcal{W}_M$ will mean the endomorphism algebras of a set of split-generators in the compact and wrapped Fukaya categories respectively.
\bigskip

Because of this, Koszul duality between the Fukaya $A_\infty$-algebras $\mathcal{F}_M$ and $\mathcal{W}_M$ should impose restrictions on the behavior of symplectic cohomology, so that it has ``tempered growth". Several theories of similar flavour have already appeared in symplectic topology. For instance, Seidel and Solomon ($\cite{ps2,ss}$) introduced the notion of a \textit{dilation} (or more generally, a \textit{quasi-dilation}) as a distinguished cohomology class in $\mathit{SH}^1(M)$, and showed that there is no exact Lagrangian $K(\pi,1)$ in a Liouville manifold which carries a dilation. From a more algebraic viewpoint, the existence of a quasi-dilation in $\mathit{SH}^1(M)$ corresponds under the closed-open string map to the existence of a \textit{dilating} $\mathbb{C}^\ast$-\textit{action} on the Fukaya category $\mathcal{F}(M)$ defined over $\mathbb{K}=\mathbb{C}$, see $\cite{ps2,ps7}$.

Let
\begin{equation}
\{f(z_1,\dots,z_n)=0\}\subset\mathbb{C}^n
\end{equation}
be any affine hypersurface with an isolated singularity at the origin. Denote by $M_f$ the Milnor fiber associated to $f$, and by $M^\sigma_f$ its stabilization, which a smooth fiber of the Lefschetz fibration
\begin{equation}
\tilde{f}+w^2:\mathbb{C}^{n+1}\rightarrow\mathbb{C},
\end{equation}
where by $\tilde{f}$ we mean a Morsification of $f$. There is the following theorem due to Seidel, which shows that the required tempered behaviour of symplectic cohomology can be achieved by iterated applications of stabilizations.
\begin{theorem}[Seidel $\cite{ps2},\cite{ps8}$]
Any Milnor fiber $M_f^{\sigma\sigma}$ obtained by double stabilization admits a quasi-dilation, and any Milnor fiber $M_f^{\sigma\sigma\sigma}$ obtained by triple stabilization admits a dilation.
\end{theorem}

One of the purposes of this paper is to show that suspending Lefschetz fibration is also a useful tool in producing new examples of Liouville manifolds whose Fukaya categories are related by $A_\infty$-Koszul duality.\bigskip

Let $\mathbb{K}$ be any field and let
\begin{equation}
\pi:E\rightarrow\mathbb{C}
\end{equation}
be an exact symplectic Lefschetz fibration on some $(2n-2)$-dimensional Liouville manifold $E$ with its smooth fiber given by the Liouville manifold $M$, where $n>2$. Associated to $\pi$ there is a directed $A_\infty$-category $\mathcal{A}(\pi)$ over $\mathbb{K}$, whose objects are objects of the compact Fukaya category $\mathcal{F}(E)$ together with the Lefschetz thimbles of $\pi$, and the morphism space $\mathit{CF}^\ast(\Delta_i,\Delta_j)$ between two thimbles $\Delta_i$ and $\Delta_j$, if non-trivial, is defined using a Hamiltonian perturbation which is small at infinity, see $\cite{ps1}$. Consider the \textit{suspension} of $\pi$, which is a Lefschetz fibration
\begin{equation}
\pi^\sigma:E\times\mathbb{C}\rightarrow\mathbb{C}
\end{equation}
defined by $\pi+w^2$, where $w$ is the holomorphic coordinate on the second factor $\mathbb{C}$. A smooth fiber of $\pi^\sigma$ is a Liouville manifold $M^\sigma$. Denote by $\mathcal{V}(M^\sigma)\subset\mathcal{F}(M^\sigma)$ the full $A_\infty$-subcategory formed by the Lagrangian spheres $V_i^\sigma\subset M^\sigma$ which are double branched covers of the Lefschetz thimbles $\Delta_i\subset E$ of $\pi$. One can further apply the suspension construction to $\pi^\sigma$, which produces a Lefschetz fibration $\pi^{\sigma\sigma}:E\times\mathbb{C}^2\rightarrow\mathbb{C}$. It is proved by Seidel in $\cite{ps4}$ that we have the following quasi-isomorphism between $A_\infty$-categories
\begin{equation}\label{eq:susp-Lef}
\mathcal{V}(M^{\sigma\sigma})\cong\mathcal{A}(\pi)\oplus\mathcal{A}(\pi)^\vee[-n],
\end{equation}
where the $A_\infty$-structure on the right-hand side is the trivial extension of that on $\mathcal{A}(\pi)$, see Section \ref{section:suspension}. Equivalently, (\ref{eq:susp-Lef}) can be expressed in terms of endomorphisms algebras, namely $\mathcal{V}_{M^{\sigma\sigma}}\cong\mathcal{A}_\pi\oplus\mathcal{A}_\pi^\vee[-n]$, where
\begin{equation}
\mathcal{V}_{M^{\sigma\sigma}}:=\bigoplus_{i,j}\mathit{CF}^\ast(V_i^{\sigma\sigma},V_j^{\sigma\sigma}) \textrm{ and } \mathcal{A}_\pi:=\bigoplus_{i,j}\mathit{CF}^\ast(\Delta_i,\Delta_j).
\end{equation}
\bigskip

In algebraic terms, the $A_\infty$-algebra $\mathcal{V}_{M^{\sigma\sigma}}$ is the \textit{cyclic completion} of the directed $A_\infty$-algebra $\mathcal{A}_\pi$ in the sense of Segal $\cite{es}$. There is also a Koszul dual construction, due to Keller $\cite{bk}$, called \textit{Calabi-Yau completion}, see Section \ref{section:CY-completion}. When applied to $\mathcal{A}_\pi$, it produces an $n$-Calabi-Yau algebra $\Pi_n(\mathcal{A}_\pi)$ in the sense of Ginzburg $\cite{vg}$. Therefore, to prove Koszul duality between the $A_\infty$-algebras $\mathcal{V}_{M^{\sigma\sigma}}$ and $\mathcal{W}_{M^{\sigma\sigma}}$, it suffices to identify $\mathcal{W}_{M^{\sigma\sigma}}$ with the $n$-Calabi-Yau completion $\Pi_n(\mathcal{A}_\pi)$. As remarked above, this has already been achieved in the case when $M^{\sigma\sigma}$ is the Milnor fiber associated to a simple singularity and $\dim_\mathbb{R}(M^{\sigma\sigma})\geq6$, in which case $\Pi_n(\mathcal{A}_\pi)$ is the Ginzburg algebra associated to a quiver with trivial potential. One of the purposes of this paper is to study some interesting symplectic manifolds whose wrapped Fukaya categories are described by quivers with non-trivial potentials.

\subsection{Main result}

In this paper, we consider the symplectic 6-manifolds $M_{p,q,r}\subset\mathbb{C}^4$ which are Milnor fibers associated to the isolated singularities
\begin{equation}
x^p+y^q+z^r+\lambda xyz+w^2=0,
\end{equation}
where
\begin{equation}
\frac{1}{p}+\frac{1}{q}+\frac{1}{r}\leq 1.
\end{equation}
The constant $\lambda\in\mathbb{C}$ above is allowed to take all but finitely many values, see $\cite{ak1}$ for details. To be explicit, we will take in this paper $\lambda=1$ when $\frac{1}{p}+\frac{1}{q}+\frac{1}{r}<1$, and $\lambda=0$ when $\frac{1}{p}+\frac{1}{q}+\frac{1}{r}=1$. Changing the value of $\lambda$ will not change the symplectic structure on $M_{p,q,r}$ up to exact symplectomorphism.

Note that the above manifolds are stabilizations of the Milnor fibers $T_{p,q,r}\subset\mathbb{C}^3$ associated to the non-simple isolated singularity of modality one
\begin{equation}\label{eq:4DMilnor}
t_{p,q,r}(x,y,z):=x^p+y^q+z^r+\lambda xyz=0
\end{equation}
studied by Keating in $\cite{ak1}$ and $\cite{ak2}$. Equivalently, they can be realized as smooth fibers of the suspension of the Lefschetz fibration $\tilde{t}_{p,q,r}:\mathbb{C}^3\rightarrow\mathbb{C}$, with $\tilde{t}_{p,q,r}$ being a Morsification of $t_{p,q,r}$. The endomorphism algebra of the directed $A_\infty$-category $\mathcal{A}(\tilde{t}_{p,q,r})$ will be denoted by $\mathcal{A}_{p,q,r}$.

In fact, our main result (Theorem \ref{theorem:main}) and its applications hold in the more general case when $p\geq2,q\geq2$ and $r\geq2$, and we can always take $\lambda=1$ for these triples $(p,q,r)=(k,2,2),(3,3,2),(4,3,2),(5,3,2)$, where $k\geq2$. The corresponding Weinstein manifolds $M_{p,q,r}$ are no longer Milnor fibers associated to any isolated singularity, but still share a lot of properties similar to Milnor fibers, e.g. they are homotopy equivalent to wedges of Lagrangian spheres. For this reason, these Weinstein manifolds are called \textit{generalized Milnor fibers}, see $\cite{ak1}$ for a detailed discussion of their symplectic topology. 

Unless otherwise stated, we shall impose $p\geq2,q\geq2,r\geq2$ as our standing assumptions when referring to the Weinstein manifolds $M_{p,q,r}$. In particular, this applies to our main result (Theorem \ref{theorem:main}) and its applications (Corollaries \ref{corollary:formal} to \ref{corollary:quasi-dilation}). In fact, $M_{p,q,r}$ is a well-define Stein manifold whenever $p,q,r$ are non-negative, and Koszul duality holds when $p\geq1,q\geq1,r\geq1$. See also Section \ref{section:exceptional} for a brief discussion for other triples $(p,q,r)$.
\bigskip

Before stating our main theorem, recall that any Weinstein domain $M^\mathrm{in}$ can be constructed by attaching critical handles to a subcritical Weinstein domain $M_0^\mathrm{in}$ along a disjoint union of Legendrian spheres in $\partial M_0^\mathrm{in}$. In our case, it turns out that the handlebody decomposition of $M_{p,q,r}$ is pretty simple in the sense that there are no subcritical handles involved (cf. Section \ref{section:surgery}), so that we can take $M_0^\mathrm{in}=D^6$ and apply Legendrian surgery along the link $\Lambda_{p,q,r}\subset(S^5,\xi_\mathit{std})$ in the standard contact 5-sphere, which is a disjoint union of standard unknotted $S^2$'s. Denote by $\mathcal{C}E(\Lambda_{p,q,r})$ the \textit{Chekanov-Eliashberg dg algebra} associated to $\Lambda_{p,q,r}$ (see Section \ref{section:comparison} for its definition), and let
\begin{equation}\label{eq:endW}
\mathcal{W}_{p,q,r}:=\bigoplus_{1\leq i\leq p+q+r-1}\mathit{CW}^\ast(L_i,L_j)
\end{equation}
be the endomorphism algebra of Lagrangian cocores of the critical handles. It is announced in Bourgeois-Ekholm-Eliashberg $\cite{bee}$ (see also $\cite{ekl}$ for a sketch of proof) that there is a quasi-isomorphism between the Chekanov-Eliashberg algebra and the wrapped Fukaya $A_\infty$-algebra of Lagrangian cocores, which when adapted to our case implies the quasi-isomorphism:
\begin{equation}\label{eq:BEE}
\mathcal{W}_{p,q,r}\cong\mathcal{C}E(\Lambda_{p,q,r}).
\end{equation}
We should emphasize at this point that the quasi-isomorphism (\ref{eq:BEE}) is crucial to all of our applications of Theorem \ref{theorem:main} in this paper. More specifically, \textit{all the corollaries stated in Section \ref{section:application} depend on the results claimed in $\cite{bee}$, which is still work-in-progress}.
\bigskip

With the notations fixed as above, the main result proved in this paper can be stated as follows.
\begin{theorem}\label{theorem:main}
Let $\mathbb{K}$ be any field, and set $\Bbbk:=\bigoplus_{1\leq i\leq p+q+r-1}\mathbb{K}e_i$. There is a quasi-isomorphism between dg algebras over $\Bbbk$:
\begin{equation}
\mathcal{C}E(\Lambda_{p,q,r})\cong\Pi_3(\mathcal{A}_{p,q,r}).
\end{equation}
\end{theorem}
We remark that the 3-Calabi-Yau completion $\Pi_3(\mathcal{A}_{p,q,r})$ is quasi-isomorphic to the Ginzburg dg algebra $\mathcal{G}_{p,q,r}$ associated to a quiver $Q_{p,q,r}$ with non-trivial potential $w_{p,q,r}$, see Section \ref{section:CY-completion}. As far as we know, this gives the first set of Liouville manifolds whose wrapped Fukaya categories can be identified with a Ginzburg dg algebra defined by a quiver with a \textit{non-trivial} potential.
\bigskip

It seems to be appropriate to say here a few words about the method we use to prove Theorem \ref{theorem:main}. In order to get the Legendrian frontal description of the Milnor fibers $M_{p,q,r}$, we start with a Lefschetz fibration on $T_{p,q,r}$ constructed by Keating in $\cite{ak2}$, whose smooth fiber is symplectomorphic to a 3-punctured torus. This fibration induces naturally a Lefschetz fibration on the stabilization $M_{p,q,r}$, whose smooth fiber is symplectomorphic to a 4-dimensional $D_4$ Milnor fiber. Using an algorithm due to Casals-Murphy $\cite{cm}$, this Lefschetz fibration on $M_{p,q,r}$ can then be translated to produce a Legendrian frontal description of $M_{p,q,r}$ which involves both 2-handles and 3-handles. After a standard procedure of handle cancellations and Reidemeister moves, we can simplify the front diagram so as to obtain the attaching link $\Lambda_{p,q,r}\subset S^5$ for $M_{p,q,r}$. In order to compute the Chekanov-Eliashberg algebra $\mathcal{C}E(\Lambda_{p,q,r})$, we use the cellular model introduced by Rutherford-Sullivan $\cite{rs1,rs2}$ for Legendrian surfaces. This enables us to simplify the analysis of Morse flow trees and compute $\mathcal{C}E(\Lambda_{p,q,r})$ over $\mathbb{Z}/2$. To get a computation of $\mathcal{C}E(\Lambda_{p,q,r})$ over any field $\mathbb{K}$, we appeal to the result of Karlsson $\cite{ck1,ck2}$ on the orientation data of Morse flow trees.

\subsection{Applications}\label{section:application}

Denote by
\begin{equation}
\mathcal{V}_{p,q,r}:=\bigoplus_{1\leq i\leq p+q+r-1}\mathit{CF}^\ast(V_i,V_j)
\end{equation}
the endomorphism algebra in the compact Fukaya category $\mathcal{F}(M_{p,q,r})$ of a basis of vanishing cycles $V_1,\dots,V_{p+q+r-1}\subset M_{p,q,r}$. To be specific, we assume that the basis of vanishing cycles is chosen so that $V_i$ is disjoint from the Lagrangian cocore $L_j$ in (\ref{eq:endW}) when $i\neq j$, and $V_i$ intersects $L_i$ transversely at a unique point. Since $M_{p,q,r}$ is the smooth fiber of a Lefschetz fibration on $\mathbb{C}^4$ obtained by suspending $\tilde{t}_{p,q,r}:\mathbb{C}^3\rightarrow\mathbb{C}$ once instead of twice, it is in general not clear whether $\mathcal{V}_{p,q,r}$ is a trivial extension of the directed $A_\infty$-algebra $\mathcal{A}_{p,q,r}$, see Lemma \ref{lemma:trivial-extension} \footnote{The author is grateful to Ailsa Keating for pointing out this.}. However, combining Theorem \ref{theorem:main} above with Theorem 4 of $\cite{ekl}$, we obtain the following formality result of $A_\infty$-algebras:
\begin{corollary}\label{corollary:formal}
There is a quasi-isomorphism
\begin{equation}
\mathcal{V}_{p,q,r}\cong\mathcal{A}_{p,q,r}\oplus\mathcal{A}_{p,q,r}^\vee[-3].
\end{equation}
In particular, the Fukaya $A_\infty$-algebra $\mathcal{V}_{p,q,r}$ is formal.
\end{corollary}
The detailed proof of the above corollary is given in Section \ref{section:quiver}. Note that when $r=2$, $M_{p,q,r}$ can be regarded as the smooth fiber of the \textit{double suspension} of the Lefschetz fibration on $\mathbb{C}^2$ defined by the Morsification of
\begin{equation}
(x^{p-2}-y^2)(x^2-\lambda y^{q-2}),
\end{equation}
see Section 2.2.5 of $\cite{ak1}$. In this case, the formality of $\mathcal{V}_{p,q,2}$ follows from a result of Seidel $\cite{ps4}$. We will give a brief sketch of his argument in Section \ref{section:suspension}, and the fact that $M_{p,q,r}$ being a double suspension is necessary for his formality theorem is also explained in detail there. 
\bigskip

As a consequence of Theorem \ref{theorem:main} and Corollary \ref{corollary:formal}, we have quasi-isomorphisms
\begin{equation}\label{eq:Koszul1}
R\mathrm{Hom}_{\mathcal{V}_{p,q,r}}(\Bbbk,\Bbbk)\cong\mathcal{C}E(\Lambda_{p,q,r}),R\mathrm{Hom}_{\mathcal{C}E(\Lambda_{p,q,r})}(\Bbbk,\Bbbk)\cong\mathcal{V}_{p,q,r}
\end{equation}
between $\mathbb{Z}$-graded $A_\infty$-algebras. This follows from a version Koszul duality between the Calabi-Yau completion and its corresponding cyclic completion, which is explained in Section \ref{section:Koszul}. In fact, it can be proved that the Fukaya category $\mathcal{F}(M_{p,q,r})$ is generated by vanishing cycles, see Corollary \ref{corollary:split-generation}.

\paragraph{Remark} It is worthy to compare the approach that we adopt in this paper to understand the wrapped Fukaya category with the dictionary of Seidel (Section 6 of $\cite{ps5}$), which describes the wrapped Fukaya category $\mathcal{W}(M)$ of the total space of a Lefschetz fibration $\pi:M\rightarrow\mathbb{C}$ as the result of a categorical localization. More precisely, one can form from the $A_\infty$-algebras $\mathcal{A}_\pi$ and $\mathcal{V}_F$ a \textit{curved} $A_\infty$-algebra
\begin{equation}
\mathcal{D}_M:=\mathcal{A}_\pi\oplus t\mathcal{V}_F[[t]]
\end{equation}
with curvature $\mu_\mathcal{D}^0=t\cdot\mathrm{id}$, $\mathrm{id}\in\mathcal{D}^0\cong\Bbbk$, where $F$ is the smooth fiber of $\pi$, and $t$ is a formal variable of degree 2. It is conjectured by Seidel in $\cite{ps5}$ and proved by Ganatra-Maydanskiy in the appendix of $\cite{bee}$ that we have a quasi-isomorphism
\begin{equation}
\mathcal{W}_M\cong T(\overline{\mathcal{D}}_M[1])^\vee,
\end{equation}
which expresses the linear dual of the wrapped Fukaya $A_\infty$-algebra as a tensor coalgebra, where $\overline{\mathcal{D}}_M\cong\mathcal{D}_M/\mathcal{D}_M^0$. Therefore in order to show that $\mathcal{W}_M$ is the $A_\infty$-Koszul dual of $\mathcal{V}_M$, it suffices to identify $\mathcal{D}_M$ with the $A_\infty$-algebra $\mathcal{V}_M$ with vanishing curvature. The simplest case when $M$ is a 4-dimensional $A_n$ Milnor fiber has been treated in detail by Pomerleano in Section 9 of $\cite{dp}$.
\bigskip

It is reasonable to expect that the Koszul duality (\ref{eq:Koszul1}) can be applied to understand the symplectic topologies of the Weinstein 6-manifolds $M_{p,q,r}$. As an application of the Koszul duality functor introduced in $\cite{bgs}$, we have the following split-generation result of the compact Fukaya categories of $M_{p,q,r}$.
\begin{corollary}\label{corollary:split-generation}
The Fukaya category of closed exact Lagrangian submanifolds $\mathcal{F}(M_{p,q,r})$ is split-generated by vanishing cycles.
\end{corollary}
This will be proved in Section \ref{section:split-generation}. Note that when $\frac{1}{p}+\frac{1}{q}+\frac{1}{r}=1$, the Milnor fiber $M_{p,q,r}$ can be defined by a quasi-homogeneous polynomial with weights $p,q,r,2$, and
\begin{equation}
\frac{1}{p}+\frac{1}{q}+\frac{1}{r}+\frac{1}{2}=\frac{3}{2}\neq1,
\end{equation}
therefore the split-generation of $\mathcal{F}(M_{p,q,r})$ by vanishing cycles follows from a theorem of Seidel ($\cite{ps2,ps3}$).

For earlier results concerning the generation of the compact Fukaya categories of plumbings, see $\cite{asm}$. In general, not much is known about the compact Fukaya category of a Liouville manifold, despite the fact that significant effort has been devoted to understand the wrapped Fukaya category of a Weinstein manifold $\cite{bee,gps1,gps2}$.\bigskip

Our next application ties Koszul duality between Fukaya $A_\infty$-algebras up with categorical dynamics in the sense of Seidel $\cite{ps2}$. Let $M$ be a Liouville manifold, recall that according to Seidel-Solomon $\cite{ss}$, a quasi-dilation is a pair $(b,h)\in\mathit{SH}^1(M)\times\mathit{SH}^0(M)^\times$ such that under the action of the BV operator $\Delta:\mathit{SH}^\ast(M)\rightarrow\mathit{SH}^{\ast-1}(M)$, it satisfies
\begin{equation}
\Delta(hb)=h.
\end{equation}
A proof of the following fact is given in Section \ref{section:quasi-dilation}.
\begin{corollary}\label{corollary:quasi-dilation}
The Liouville manifold $M_{p,q,r}$ admits a quasi-dilation over any field $\mathbb{K}$.
\end{corollary}
As we have remarked above, $M_{p,q,2}$ is Liouville isomorphic to an affine conic bundle over $\mathbb{C}^2$, therefore the existence of a quasi-dilation follows from an iterative application of the Lefschetz fibration argument due to Seidel-Solomon, see Lecture 19 of $\cite{ps2}$.

\section*{Acknowledgements}
I would like to thank my PhD supervisor Yank{\i} Lekili for very useful conversations during the preparation of this work, especially for his help with the proof of Corollary \ref{corollary:split-generation}. I am also grateful to Travis Schedler for helpful discussions concerning Koszul duality and BV structures; and Ailsa Keating for pointing out the mistake concerning my misunderstanding of Seidel's result on suspending Lefschetz fibrations in an earlier version of this paper, see Section \ref{section:application}. After writing up the first draft of this paper, the discussions with Roger Casals have contributed a lot to Section \ref{section:exceptional}, therefore I would also like to express my gratitude to him. Finally, I'm indebted to the referee for many useful comments and suggestions, which have considerably improved this paper.
\bigskip

This work was supported by the Engineering and Physical Sciences Research Council [EP/L015234/1], the EPSRC Centre for Doctoral Training in Geometry and Number Theory (The London School of Geometry and Number Theory), University College London. The author is also funded by King's College London for his PhD studies.

\section{Noncommutative symplectic geometry}

We collect here some basic algebraic concepts and facts concerning the noncommutative geometry of a symplectic Calabi-Yau manifold. They will be used later to study the commutative geometry of a Liouville manifold. Let $\mathbb{K}$ be any field.

\subsection{Calabi-Yau algebras}\label{section:CY algebra}

Let $\mathcal{A}$ be a $\mathbb{Z}$-graded $A_\infty$-algebra, denote by $\mathcal{A}^e$ the tensor algebra $\mathcal{A}\otimes\mathcal{A}^\mathit{op}$. For any $\mathcal{A}$-bimodule $\mathcal{M}$, define its derived dual in the derived category $D^\mathit{mod}(\mathcal{A}^e)$ of $\mathcal{A}$-bimodules to be
\begin{equation}\label{eq:dual-bimod}
\mathcal{M}^\vee:=R\mathrm{Hom}_{\mathcal{A}^e}(\mathcal{M},\mathcal{A}^e).
\end{equation}
$\mathcal{A}$ is said to be an $n$-\textit{Calabi-Yau algebra} if there is a non-degenerate class $\eta\in\mathit{HH}_{-d}(\mathcal{A},\mathcal{A})$, inducing a (self-dual) quasi-isomorphism
\begin{equation}
\mathcal{A}^\vee[n]\cong\mathcal{A}
\end{equation}
between $\mathcal{A}$-bimodules.\bigskip

In this paper, we will be mainly dealing with 3-Calabi-Yau algebras. An important class of 3-Calabi-Yau dg algebras are given by quivers with potentials, originally introduced by Ginzburg in $\cite{vg}$. Let $Q=(Q_0,Q_1)$ be a finite quiver with the set of vertices $Q_0$ and te set of arrows $Q_1$. By definition, a potential on $Q$ is an element
\begin{equation}
w\in\mathbb{K}Q/[\mathbb{K}Q,\mathbb{K}Q]
\end{equation}
in the space of all cyclic paths in $Q$. The \textit{Ginzburg algebra} $\mathcal{G}(Q,w)$ associated to $(Q,w)$ is the dg algebra freely generated over $\mathbb{K}$ by
\begin{itemize}
	\item the arrows $a\in Q_1$ with $\deg(a)=0$;
	\item the opposite arrows $a^\ast$ of $a$ with $\deg(a^\ast)=-1$;
	\item loops $z_v$ for every vertex $v\in Q_0$, with $\deg(z_v)=-2$.
\end{itemize}
The differential $d$ on $\mathcal{G}(Q,w)$ is defined to be
\begin{equation}
da=0,da^\ast=\frac{^\circ\partial \tilde{w}}{\partial a},d(\sum_{v\in Q_0}z_v)=\sum_{a\in Q_1}[a,a^\ast]
\end{equation}
on the set of generators, where $\tilde{w}$ is the sum of all cyclic permutations of $w$, and $^\circ\partial$ is the circular derivative introduced by Kontsevich, which is by definition
\begin{equation}
\frac{^\circ\partial\tilde{w}}{\partial a}=\sum_{\tilde{w}=uav}vu.
\end{equation}
One can then extend $d$ to a differential on the whole algebra $\mathcal{G}(Q,w)$ by graded Leibniz rule. Note that $\mathcal{G}(Q,w)$ can be regarded as a dg algebra over $\Bbbk=\bigoplus_{v\in Q_0}\mathbb{K}e_v$, with its $\Bbbk$-bimodule structure induced from the path algebra $\mathbb{K}\widetilde{Q}$ of the extended quiver $\widetilde{Q}$, which is obtained by adjoining opposite arrows $a^\ast$ and loops $z_v$ to the original quiver $Q$.
\bigskip

In the terminology of the work-in-progress by Cohen-Ganatra $\cite{cg}$, a Calabi-Yau algebra in the above sense is called a \textit{smooth 3-Calabi-Yau algebra}, meaning that the Calabi-Yau structure $\eta\in\mathit{HH}_{-d}(\mathcal{A},\mathcal{A})$ admits a lift in the negative cyclic homology $\mathit{HC}^-_{-d}(\mathcal{A},\mathcal{A})$. The Koszul dual notion, which is a chain level refinement of Poincar\'{e} duality between Floer cohomologies, is referred to as a \textit{compact 3-Calabi-Yau algebra}. To give an example, start again with the quiver with potential $(Q,w)$. Without generality, we can assume $w$ to be \textit{reduced}, namely there is no summand in $w$ which is a cycle in $Q$ with length $<3$. Consider the $\mathbb{Z}$-graded $A_\infty$-algebra $\mathcal{B}(Q,w)$, which as a graded $\Bbbk$-module is given by
\begin{equation}
\mathcal{B}(Q,w):=\bigoplus_{i,j\in Q_0}\mathbb{K}^{\delta_{ij}}\oplus V_{ij}^\vee[-1]\oplus V_{ji}^\vee[-2]\oplus\mathbb{K}^{\delta_{ij}}[-3],
\end{equation}
where $V_{ij}$ is the (trivially graded) vector space with basis consisting of the arrows in $Q_1$ which start from the vertex $i$ and end at the vertex $j$, and $V_{ji}$ is the vector space generated by the opposite of the arrows in $V_{ij}$. The $A_\infty$-operations
\begin{equation}
\mu^k_\mathcal{B}:\mathcal{B}(Q,w)^{\otimes k}\rightarrow\mathcal{B}(Q,w)[2-k]
\end{equation}
are defined explicitly by
\begin{equation}
\mu^k_\mathcal{B}(g_k^\vee,\dots,g_1^\vee)=(-1)^{(k-1)|g_k^\vee|+\dots+2|g_3^\vee|+|g_2^\vee|}\sum_g\mathrm{Coeff}_{g_k\dots g_1}(dg)\cdot g^\vee,
\end{equation}
where the sum on the right-hand side above is taken over all the generators of the Ginzburg algebra $\mathcal{G}(Q,w)$, and $\mathrm{Coeff}_{g_k\dots g_1}(dg)$ is the coefficient of $g_k\dots g_1$ in the differential of $g$ in $\mathcal{G}(Q,w)$, which is determined by the potential $w$. Moreover, the grading of $g^\vee$ is given by $|g^\vee|=1-\deg(g)$ in the $A_\infty$-algebra $\mathcal{B}(Q,w)$.

Consider the pairing $(\cdot,\cdot)$ on the vector space generated by the degree 0 and degree 1 generators of $\mathcal{G}(Q,w)$, which satisfies
\begin{itemize}
	\item $(g_i,g_j)=-(-1)^{|g_1||g_2|}(g_j,g_i)$;
	\item $(g_i,g_j)=0$ unless $t(g_i)=h(g_j)$ and $t(g_j)=h(g_i)$, where $h(g)$ and $t(g)$ are respectively the head and tail of an arrow or opposite arrow $g$;
	\item the matrix with entries given by $(g_i,g_j)$ is invertible.
\end{itemize}
Let $\langle\cdot,\cdot\rangle$ be the dual pairing of $(\cdot,\cdot)$, in this way we get a non-degenerate pairing
\begin{equation}
\langle\cdot,\cdot\rangle_\mathcal{B}:\mathcal{B}(Q,w)\times\mathcal{B}(Q,w)\rightarrow\Bbbk[-3]
\end{equation}
defined by
\begin{equation}
\langle g_1^\vee,g_2^\vee\rangle_\mathcal{B}=(-1)^{|g_1^\vee|}\langle g_1^\vee,g_2^\vee\rangle.
\end{equation}
Using the definition of $\mu_\mathcal{B}^k$, one can verify that
\begin{equation}
\left\langle\mu_\mathcal{B}^k(g_k^\vee,\dots,g_1^\vee),g_0^\vee\right\rangle_\mathcal{B}=(-1)^{k+|g_k^\vee|(|g_{k-1}^\vee|+\dots+|g_0^\vee|)}\left\langle\mu_\mathcal{B}^k(g_{k-1}^\vee,\dots,g_0^\vee),g_k^\vee\right\rangle_\mathcal{B},
\end{equation}
which shows that $\mathcal{B}(Q,w)$ is a \textit{cyclic} $A_\infty$-algebra with respect to the pairing $\langle\cdot,\cdot\rangle_\mathcal{B}$.

We refer the readers to $\cite{cg}$ for the precise definition of a compact Calabi-Yau structure. When $\mathrm{char}(\mathbb{K})=0$, any cyclic $A_\infty$-algebra $\mathcal{B}$ over $\Bbbk$ carries a canonical compact Calabi-Yau structure, and a compact Calabi-Yau structure on any $A_\infty$-algebra $\mathcal{B}$ determines a quasi-isomorphism between $\mathcal{B}$ and a cyclic $A_\infty$-algebra $\mathcal{B}'$. When $\mathrm{char}(\mathbb{K})\neq0$, the notions of a cyclic $A_\infty$-algebra and a compact Calabi-Yau $A_\infty$-algebra in general differ from each other. Since we shall be mainly dealing with $A_\infty$-algebras arising from quivers with potentials in this paper, we will not distinguish between these two notions.

\subsection{Cyclic completions}\label{section:cyclic}

We recall the construction of the cyclic completion of an $A_\infty$-algebra from $\cite{es}$. This is nothing else but the chain level refinement of trivial extensions of (graded) associative algebras.

Let $\Bbbk$ be a semisimple ring which is a direct sum of a finite number of copies of the field $\mathbb{K}$. Suppose that $\mathcal{A}$ is a $\mathbb{Z}$-graded $A_\infty$-algebra over $\Bbbk$, with $\mu_\mathcal{A}^k:\mathcal{A}^{\otimes k}\rightarrow\mathcal{A}[2-k]$ being its structure maps, then one can associate to $\mathcal{A}$ another $A_\infty$-algebra $\mathcal{B}$ which as an $\mathcal{A}$-bimodule is given by
\begin{equation}
\mathcal{A}\oplus\mathcal{A}^\vee[-n],
\end{equation}
where $\mathcal{A}^\vee[-n]$ is the dual diagonal $\mathcal{A}$-bimodule defined by (\ref{eq:dual-bimod}), whose bimodule operations will be denoted by
\begin{equation}
\mu^{s|1|r}_{\mathcal{A}^\vee[-n]}:\mathcal{A}^{\otimes s}\otimes\mathcal{A}^\vee[-n]\otimes\mathcal{A}^{\otimes r}\rightarrow\mathcal{A}^\vee[1-r-s-n].
\end{equation}
The $A_\infty$-operations
\begin{equation}
\mu^k_\mathcal{B}:\mathcal{B}^{\otimes k}\rightarrow\mathcal{B}[2-k]
\end{equation}
on the trivial extension are then defined to be the direct sum of $\mu^k_\mathcal{A}$ with $\mu^{i-1|1|k-i}_{\mathcal{A}^\vee[-n]}$. More precisely,
\begin{equation}
\begin{split}
&\mu_\mathcal{B}^k\left((a_k,a_k^\vee),\dots,(a_1,a_1^\vee)\right):=\Big(\mu_\mathcal{A}^k(a_k,\dots,a_1),\\
&\sum_{i=1}^k(-1)^{|a_1|+\dots+|a_{i-1}|-i+2}\mu^{i-1|1|k-i}_{\mathcal{A}^\vee[-n]}(a_k,\dots,a_{i+1},a_i^\vee,a_{i-1},\dots,a_1)\Big),
\end{split}
\end{equation}
where $a_j\in\mathcal{A}$ and $a_j^\vee\in\mathcal{A}^\vee[-n]$.

The $A_\infty$-algebra $\mathcal{B}$ defined above is called the $n$-\textit{cyclic completion} of $\mathcal{A}$. In many cases, the $A_\infty$ structure $\mu_\mathcal{B}^\bullet$ on $\mathcal{B}$ defined above is cyclic and extends the (usually) non-cyclic $A_\infty$-structure $\mu_\mathcal{A}^\bullet$ on $\mathcal{A}$. The following is an example.
\bigskip

Let $\mathcal{A}\cong\mathbb{K}\overrightarrow{Q}/\overrightarrow{I}$ be a finite-dimensional graded associative algebra over $\mathbb{K}$ defined by a quiver $\overrightarrow{Q}=(\overrightarrow{Q}_0,\overrightarrow{Q}_1)$ with the set of vertices $\overrightarrow{Q}_0$ and the set of arrows $\overrightarrow{Q}_1$, together with a 2-sided ideal $\overrightarrow{I}\subset\mathbb{K}\langle\overrightarrow{Q}_1\rangle^{\otimes 2}$ generated by quadratic relations (so in particular the $A_\infty$-structure on $\mathcal{A}$ is formal). In this case, the 3-cyclic completion $\mathcal{B}=\mathcal{A}\oplus\mathcal{A}^\vee[-3]$ has a simple realization. More precisely, denote by $\{\rho_1,\dots,\rho_r\}$ the set of relations which generate the ideal $\overrightarrow{I}$, and by $y_j$ an arrow which reverses the compositions of the arrows appeared in $\rho_j$. Then $\mathcal{B}$ is the cyclic $A_\infty$-algebra defined by the quiver with potential $(Q,w)$ (cf. Section \ref{section:CY algebra}), where $Q_0=\overrightarrow{Q}_0$,
\begin{equation}
Q_1=\overrightarrow{Q}_1\cup\{y_1,\dots,y_r\}
\end{equation}
and
\begin{equation}
w=\sum_{1\leq j\leq r}y_j\rho_j.
\end{equation}
\begin{proposition}\label{proposition:formality}
Let $\mathcal{A}\cong\mathbb{K}\overrightarrow{Q}/\overrightarrow{I}$ be as above, then the $A_\infty$-structure on its 3-cyclic completion $\mathcal{B}$ is formal.
\end{proposition}
\begin{proof}
This is obvious since $\mu_\mathcal{B}^k$ is defined additively in terms of $\mu_\mathcal{A}^k$, but $\mathcal{A}$ is formal as an $A_\infty$-algebra.
\end{proof}

\subsection{Calabi-Yau completions}\label{section:CY-completion}

We recall here another algebraic construction associated to $A_\infty$-algebras due to Keller $\cite{bk}$. This is Koszul dual to Segal's cyclic completion.
\bigskip

Let $\mathcal{A}$ be a $\mathbb{Z}$-graded $A_\infty$-algebra over the semisimple ring $\Bbbk=\bigoplus_{i=1}^r\mathbb{K}e_i$. Without loss of generality we can assume that $\mu_\mathcal{A}^k=0$ for $k\geq3$, namely it is differential graded, since in general we can always find a dg algebra $\mathcal{A}'$ which is quasi-isomorphic to $\mathcal{A}$ and apply the construction described below to $\mathcal{A}'$. Its $n$-\textit{Calabi-Yau completion} is defined to be the tensor dg algebra
\begin{equation}\Pi_n(\mathcal{A}):=T_\mathcal{A}(\mathrm{Res}(\mathcal{A})^\vee[n-1]),\end{equation}
where the $\mathcal{A}$-bimodule $\mathrm{Res}(\mathcal{A})^\vee$ is called the \textit{inverse dualizing complex}. 

Let us illustrate the above definition under the assumption that $\mathcal{A}$ is a \textit{semi-free dg algebra} over the field $\mathbb{K}$, which means its underlying graded algebra is freely generated over $\mathbb{K}$, with the set of generators given by $\{a_1,\dots,a_m\}$. Consider $\Omega_\mathcal{A}^1$, the bimodule of differentials on $\mathcal{A}$, which has $Da_1,\dots,Da_m$ as its basis, with $D:\mathcal{A}\rightarrow\Omega^1_\mathcal{A}$ being the universal derivation. Recall that we have a short exact sequence
\begin{equation}0\rightarrow\Omega_\mathcal{A}^1\xrightarrow{\alpha}\mathcal{A}\otimes\mathcal{A}\xrightarrow{\circ}\mathcal{A}\rightarrow0,\end{equation}
where $\alpha(Da)=a\otimes1-1\otimes a$, and $\circ$ is the multiplication on $\mathcal{A}$. In this case, $\mathrm{Res}(\mathcal{A})$ is the mapping cone of $\alpha$, which gives a semi-free bimodule resolution of $\mathcal{A}$ with basis $\left\{e_1\otimes e_1,\cdots,e_r\otimes e_r,Da_1[1],\cdots,Da_m[1]\right\}$. $\mathrm{Res}(\mathcal{A})^\vee$ is the dual bimodule of $\mathrm{Res}(\mathcal{A})$. 

If $\phi:\mathcal{A}\rightarrow\mathcal{A}'$ is a quasi-isomorphism between dg algebras over $\Bbbk$, then there is an induced quasi-isomorphism $\Pi_n(\phi):\Pi_n(\mathcal{A})\rightarrow\Pi_n(\mathcal{A}')$.

\begin{theorem}[Keller $\cite{bk,bke}$]
If $\mathcal{A}$ is homologically smooth, then its $n$-Calabi-Yau completion $\Pi_n(\mathcal{A})$ is a smooth $n$-Calabi-Yau algebra.
\end{theorem}
\bigskip

We now specialize the above construction to the case of quiver algebras. Denote again by $\overrightarrow{Q}=(\overrightarrow{Q}_0,\overrightarrow{Q}_1)$ a finite quiver. Let $\mathcal{A}=\mathbb{K}\overrightarrow{Q}/\overrightarrow{I}$ be the graded associative algebra defined by a quadratic ideal $\overrightarrow{I}\subset\mathbb{K}\langle\overrightarrow{Q}_1\rangle^{\otimes 2}$. Here we assume in addition that
\bigskip

\textit{$\mathcal{A}$ has global dimension less than or equal to 2.}
\bigskip

From these data we can construct a quiver with potential $(Q,w)$ exactly as in Section \ref{section:cyclic}. Namely one starts with the quiver $\overrightarrow{Q}$ and adds arrows to it according to the relations which generate $\overrightarrow{I}$.

In this case we have the following result due to Keller, which gives a concrete realization of the 3-Calabi-Yau completion $\Pi_3(\mathcal{A})$ as the Ginzburg algebra associated to $(Q,w)$.
\begin{proposition}[Theorem 6.10 of $\cite{bk}$]\label{proposition:CY-completion}
Under the above assumptions, the 3-Calabi-Yau completion $\Pi_3(\mathcal{A})$ is quasi-isomorphic to the Ginzburg dg algebra $\mathcal{G}(Q,w)$.
\end{proposition}

\subsection{Koszul duality}\label{section:Koszul}

We explain the Koszul duality between the compact and smooth 3-Calabi-Yau algebras $\mathcal{B}$ and $\mathcal{G}$ defined in Section \ref{section:CY algebra} associated to the same quiver with potential $(Q,w)$, where
\bigskip

\textit{$w$ is homogeneous and contains only cubic terms.}
\bigskip

We refer the interested readers to Section 2 of $\cite{ekl}$ for a detailed introduction to the algebraic preliminaries concerning Koszul duality.
\bigskip

In this section, we will follow the convention of $\cite{lpwz}$, and work with $\mathbb{Z}\times\mathbb{Z}$-graded augmented $A_\infty$-algebras $(\mathcal{A},\varepsilon)$ over a semisimple ring $\Bbbk$, where $\varepsilon:\mathcal{A}\rightarrow\Bbbk$ is an augmentation. This means that all the $A_\infty$-operations $\mu_\mathcal{A}^k$ have bidegree $(2-k,0)$. Let $\phi:\mathcal{A}\rightarrow\mathcal{B}$ be a morphism between bigraded $A_\infty$-algebras, which consists of a sequence of $\Bbbk$-linear maps $\phi_k:\mathcal{A}^{\otimes k}\rightarrow\mathcal{B}$, then $\phi_k$ has bidegree $(1-k,0)$. The second grading on $\mathcal{A}$ will be referred to as the \textit{Adams grading}, it is preserved by all the $A_\infty$-operations $\mu^k_\mathcal{A}$. With respect to the bigrading, as a $\Bbbk$-bimodule $\mathcal{A}$ decomposes as $\mathcal{A}=\bigoplus_{i,j}\mathcal{A}^{i,j}$. We say that $\mathcal{A}$ is \textit{locally finite} if each $\mathcal{A}^{i,j}$ is finite-dimensional over $\Bbbk$.

As a $\Bbbk$-bimodule, the Koszul dual $E(\mathcal{A})$ of $\mathcal{A}$ is defined explicitly by
\begin{equation}\label{eq:redbar}
E(\mathcal{A})^{m,n}:=\bigoplus_{d\geq1}\bigoplus_{\sum_{k=1}^di_k=m,\sum_{k=1}^dj_k=n}(\overline{\mathcal{A}}[1]^\#)^{i_1,j_1}\otimes\dots\otimes(\overline{\mathcal{A}}[1]^\#)^{i_d,j_d},
\end{equation}
where $\overline{\mathcal{A}}:=\ker(\varepsilon)$ is the augmentation ideal, $\#$ denotes the graded linear dual, and the shift functor $[1]$ acts on the first grading. Note that $E(\mathcal{A})$ carries te structure of an augmented dg algebra, with its differential induced from the $A_\infty$-structure on $\mathcal{A}$. 

Note that a map $\phi:\mathcal{A}\rightarrow\mathcal{B}$ between $\mathbb{Z}\times\mathbb{Z}$-graded augmented $A_\infty$-algebras is a quasi-isomorphism if and only if the induced coaugmented dg coalgebra map $B\phi:B\mathcal{A}\rightarrow B\mathcal{B}$ on their bar constructions is a quasi-isomorphism. This can be proved using the spectral sequence associated to the \textit{word length filtration} on the complex $B\mathcal{A}$, see Section 2.2.1 of $\cite{ekl}$. Recall that by definition, $E(\mathcal{A})=(B\mathcal{A})^\#$, therefore a quasi-isomorphism $\phi:\mathcal{A}\rightarrow\mathcal{B}$ between augmented $A_\infty$-algebras induces a quasi-isomorphism $E(\mathcal{A})\rightarrow E(\mathcal{B})$.

\begin{theorem}[Theorem 2.4 of $\cite{lpwz}$]\label{theorem:lf}
Let $\mathcal{A}$ be a locally finite augmented $A_\infty$-algebra. If its Koszul dual $E(\mathcal{A})$ is also locally finite, then $E(E(\mathcal{A}))$ is quasi-isomorphic to $\mathcal{A}$.
\end{theorem}
\paragraph{Remark} In $\cite{lpwz}$, there is no locally finiteness assumption on $\mathcal{A}$. However, this assumption is essential for the argument presented in $\cite{lpwz}$, see for example Lemma 9 and Theorem 17 of $\cite{ekl}$, where such an assumption is included.
\begin{proposition}\label{proposition:Koszul}
Let $\mathcal{B}$ and $\mathcal{G}$ be the compact and smooth Calabi-Yau 3-algebras defined by the same quiver with potential $(Q,w)$. We can equip them with Adams gradings so that there are quasi-isomorphisms between bigraded $A_\infty$-algebras
\begin{equation}
E(\mathcal{B})\cong\mathcal{G},E(\mathcal{G})\cong\mathcal{B}.
\end{equation}
\end{proposition}
\begin{proof}
To begin with, we equip $\mathcal{B}$ with a trivial bigrading $(0,j)$, where the first grading is always fixed to be 0, and the Adams grading is the total grading on the $\mathbb{Z}$-graded $A_\infty$-algebra $\mathcal{B}$. 

Note that although the original grading on $E(\mathcal{B})=\mathcal{G}$ is usually not locally finite, we can equip $\mathcal{G}$ with a bigrading by declaring the following:
\begin{itemize}
	\item the original arrows $a$ in $Q_1$ have bigrading $(1,-1)$;
	\item the opposite arrows $a^\ast$ to $a$ have bigrading $(1,-2)$;
	\item the loops $z_v$ at the vertex $v$ have bigrading $(1,-3)$,
\end{itemize}
so that the total grading recovers the original grading on $\mathcal{G}$. By our assumption that every term in the potential $w$ is cubic, we see that the differential $d$ on $\mathcal{G}$ has bidegree $(1,0)$ with respect to the bigrading defined above.

Note that $(\mathcal{B},\varepsilon_{\mathcal{B}})$ is a bigraded augmented $A_\infty$-algebra with $\varepsilon_\mathcal{B}:\mathcal{B}\rightarrow\Bbbk$ being the projection to the degree 0 part. By (\ref{eq:redbar}) we have
\begin{equation}
E(\mathcal{B})=T(\overline{\mathcal{B}}[1]^\vee).
\end{equation}
It follows from our definitions in Section \ref{section:CY algebra} that the right-hand side above is quasi-isomorphic to the Ginzburg dg algebra $\mathcal{G}$ equipped with the bigrading specified above. This proves the first quasi-isomorphism.
	
Since $\mathcal{G}$ is locally finite with respect to the double grading specified above, and $\mathcal{B}$ is clearly locally finite, we can apply Theorem \ref{theorem:lf} to $\mathcal{B}$, which gives
\begin{equation}
\mathcal{B}\cong E(E(\mathcal{B}))\cong E(\mathcal{G}).
\end{equation}
\end{proof}

\paragraph{Remark} The bigrading used in the above proof has potential applications in showing the primitivity of the homology classes of Lagrangian homology spheres in $M_{p,q,r}$ when $\mathbb{K}=\mathbb{C}$. In fact, Corollaries \ref{corollary:formal} and \ref{corollary:split-generation} boil the question of classifying Lagrangian homology spheres in $M_{p,q,r}$ down to that of classifying $\mathbb{C}^\ast$-equivariant $A_\infty$-modules over a quiver algebra $\mathcal{B}_{p,q,r}$, see Section \ref{section:quiver}. Meanwhile, the dg category of perfect $A_\infty$-modules over $\mathcal{B}_{p,q,r}$ admits a bigraded refinement, which comes essentially from the bigrading on $\mathcal{B}_{p,q,r}$ specified in the proof of the above Proposition. One can then imitate the argument of $\cite{ps10}$.

\begin{proposition}\label{proposition:lf}
Assume in addition that $H^\ast(\mathcal{G})$ is finite dimensional in each fixed degree, we have quasi-isomorphisms
\begin{equation}
R\mathrm{Hom}_{\mathcal{B}}(\Bbbk,\Bbbk)\cong\mathcal{G},R\mathrm{Hom}_{\mathcal{G}}(\Bbbk,\Bbbk)\cong\mathcal{B}
\end{equation}
as $\mathbb{Z}$-graded $A_\infty$-algebras.
\end{proposition}
\begin{proof}
According to Section 2.3 of $\cite{dy}$, in general we have
\begin{equation}
R\mathrm{Hom}_{\mathcal{B}}(\Bbbk,\Bbbk)\cong\widehat{\mathcal{G}},
\end{equation}
where $\widehat{\mathcal{G}}$ is the \textit{completed Ginzburg algebra} associated to the quiver with potential $(Q,w)$, namely the Ginzburg algebra $\mathcal{G}$ completed with respect to the path length in $\mathbb{K}\widetilde{Q}$, with $\widetilde{Q}$ being the ``double" of $Q$ obtained by adding to $Q$ the reversed arrows $a^\ast$ for each $a\in Q_1$ and the loops $z_v$ for each $v\in Q_0$. 

To conclude the proof, we need to show that with our assumptions, $\widehat{\mathcal{G}}$ is quasi-isomorphic to the uncompleted Ginzburg algebra $\mathcal{G}$. This can be seen by considering the filtration $F^\bullet H^\ast(\mathcal{G})$ on cohomology induced by the path length filtration on $\mathcal{G}$. By our standing assumption that the potential $w$ consists only of cubic terms, we see that the differentials of the generators in $\mathcal{G}$ consist of homogeneous terms with respect to the path length filtration, thus the filtration $F^\bullet$ on $H^\ast(\mathcal{G})$ is Hausdorff, which means that the completion map $\mathcal{G}\rightarrow\widehat{\mathcal{G}}$ is cohomologically injective. On the other hand, since $H^\ast(\mathcal{G})$ is finite dimensional in each degree, we conclude that the filtration $F^\bullet$ on $H^\ast(\mathcal{G})$ is complete, therefore $\mathcal{G}\rightarrow\widehat{\mathcal{G}}$ is cohomologically surjective. See Section 5.4 of $\cite{cw}$.
\end{proof}

\section{Formality of $A_\infty$-structures}

Assuming Theorem \ref{theorem:main}, we prove in this section the formality statement claimed in Corollary \ref{corollary:formal}, and therefore identify the Fukaya $A_\infty$-algebra $\mathcal{V}_{p,q,r}$ of vanishing cycles with the cyclic completion of a quiver algebra which is quasi-isomorphic to $\mathcal{A}_{p,q,r}$. We begin by recalling the suspension construction of Lefschetz fibrations due to Seidel $\cite{ps4}$, and a general form of the Eilenberg-Moore equivalence relating the Fukaya $A_\infty$-algebra and the Chekanov-Eliashberg algebra due to Ekholm-Lekili $\cite{ekl}$.

\subsection{Suspension of a Lefschetz fibration}\label{section:suspension}

Let $E$ be a $2n$-dimensional Liouville manifold and let $\pi:E\rightarrow\mathbb{C}$ be an exact symplectic Lefschetz fibration with smooth fiber $M$, which is also a Lioiuville manifold. Assume that $c_1(M)=0$. The \textit{suspension}
\begin{equation}\pi^\sigma:E^\sigma:=E\times\mathbb{C}\rightarrow\mathbb{C}\end{equation}
of the Lefschetz fibration $\pi$ is defined as $\pi^\sigma(x,y)=\pi(x)+y^2$, where $x\in E$ and $y\in\mathbb{C}$. Note in particular that $\pi^\sigma$ is still a Lefschetz fibration. Denote by $M^\sigma$ a smooth fiber of $\pi^\sigma$, which is again a $2n$-dimensional Liouville manifold with $c_1(M^\sigma)=0$, so we have two well-defined $\mathbb{Z}$-graded $A_\infty$-categories: the Fukaya category $\mathcal{F}(M)$ of the smooth fiber of $\pi$, and the Fukaya category $\mathcal{F}(M^\sigma)$ of the smooth fiber of $\pi^\sigma$. We describe here the algebraic construction of Seidel $\cite{ps4}$, which, when applied to geometry, describes a full $A_\infty$-subcategory $\mathcal{V}(M^{\sigma})\subset\mathcal{F}(M^{\sigma})$ in terms of the $A_\infty$-category $\mathcal{A}(\pi)$ associated to the Lefschetz fibration $\pi$.
\bigskip

Let $\mathcal{B}$ be a $\mathbb{Z}$-graded, strictly unital, proper $A_\infty$-category, defined over any field $\mathbb{K}$, fix a set $\{S_1,\dots,S_k\}$ of non-trivial objects of $\mathcal{B}$, which means that $\hom_\mathcal{B}(S_i,S_i)$ is never acyclic for $1\leq i\leq k$. Let $\mathcal{A}\subset\mathcal{B}$ be the directed $A_\infty$-subcategory with the same objects as $\mathcal{B}$ but whose morphism spaces are set to be
\begin{equation}
\hom_\mathcal{A}(S_i,S_j)=\left\{\begin{array}{ll}\hom_\mathcal{B}(S_i,S_j) & i<j \\ \mathbb{K}\cdot e_{S_i} & i=j \\ 0 & i>j \end{array}\right.
\end{equation}
The $A_\infty$-structure on $\mathcal{A}$ is defined to be the restriction of that of $\mathcal{B}$.

Let $\mathit{C\ell}_2(\mathcal{B})$ be another $A_\infty$-category with objects
\begin{equation}
(S_1^-,\dots,S_k^-,S_1^+,\dots,S_k^+),
\end{equation}
where $S_i^+$ is a copy of $S_i$, while $S_i^-$ is a shifted copy $S_i[1]$. As the simplest instance of $A_\infty$-Morita equivalence, $\mathit{C\ell}_2(\mathcal{B})$ is quasi-isomorphic to $\mathcal{B}$. Let $\mathcal{C}\subset\mathit{C\ell}_2(\mathcal{B})$ be the associated directed $A_\infty$-subcategory. Schematically we have
\begin{equation}
\mathcal{C}=\left[\begin{array}{ll} \mathcal{A} & 0 \\ \mathcal{B}[-1] & \mathcal{A} \end{array} \right]\subset\mathit{C\ell}_2(\mathcal{B})=\left[\begin{array}{ll} \mathcal{B} & \mathcal{B}[1] \\ \mathcal{B}[-1] & \mathcal{B} \end{array} \right].
\end{equation}

Finally, we introduce a third $A_\infty$-category $\mathcal{B}^\sigma$ with objects $(S_1^\sigma,\dots,S_k^\sigma)$. This is the full $A_\infty$-subcategory of $\mathcal{C}^\mathit{tw}$ consisting of the twisted complexes
\begin{equation}
S_i^\sigma=\mathit{Cone}(S_i^{-}[-1]\xrightarrow{e_{S_i}}S_i^+)=\left(S_i^-\oplus S_i^+,\delta_{S_i^\sigma}=\left[\begin{array}{ll} 0 & 0 \\ e_{S_i} & 0 \end{array}\right]\right).
\end{equation}
Under the assumption that each $S_i$ is simple, i.e.
\begin{equation}\label{eq:H0}
H^0(\hom_\mathcal{B}(S_i,S_i))=\mathbb{K}[e_{S_i}]
\end{equation}
for $i=1,\dots,k$, the directed $A_\infty$-category $\mathcal{A}$, together with the quasi-isomorphism class of the $\mathcal{A}$-bimodule $\mathcal{B}$ determine the $A_\infty$-category $\mathcal{B}^\sigma$ up to quasi-isomorphism.
\bigskip

The $A_\infty$-category $\mathcal{B}^\sigma$ is called the \textit{algebraic suspension} of $\mathcal{B}$. Let $\mathcal{A}^\sigma\subset\mathcal{B}^\sigma$ be its associated directed $A_\infty$-subcategory, it is easy to see that $\mathcal{A}^\sigma$ is quasi-isomorphic to $\mathcal{A}$. The main result of $\cite{ps4}$ is a description of the algebraic suspension $\mathcal{B}^\sigma$ in terms of the pairing $(\mathcal{A},\mathcal{B})$.
\begin{lemma}[Lemma 4.2 of $\cite{ps4}$]\label{lemma:trivial-extension}
Assume that (\ref{eq:H0}) holds, and as an $\mathcal{A}$-bimodule, $\mathcal{B}$ is quasi-isomorphic to $\mathcal{A}\oplus(\mathcal{B}/\mathcal{A})[-1]$, then the $A_\infty$-category $\mathcal{B}^\sigma$ is quasi-isomorphic to the trivial extension constructed from $\mathcal{A}$ and the $\mathcal{A}$-bimodule $(\mathcal{B}/\mathcal{A})[-1]$.
\end{lemma}
\bigskip

Geometrically, the construction above can be applied to the pairing
\begin{equation}
(\mathcal{A},\mathcal{B})=(\mathcal{A}(\pi),\mathcal{V}(M)),
\end{equation}
where $\mathcal{V}(M)\subset\mathcal{F}(M)$ is the full $A_\infty$-subcategory which consists of vanishing cycles $V_1,\dots,V_k\subset M_\ast$, where $M_\ast\cong M$ is the fiber of $\pi$ over some chosen base point $\ast\in\mathbb{C}$. Starting from a distinguished basis of vanishing cycles $V_1,\dots,V_k$, the Lagrangian spheres $V_1^\sigma,\dots,V_k^\sigma\subset M_\ast^\sigma$ can be described as double branched covers of the corresponding basis of Lefschetz thimbles $\Delta_1,\dots,\Delta_k$ whose vanishing paths share the common end point $\ast$. In particular, $\partial\Delta_i=V_i$. Denote by $\mathcal{V}(M^\sigma)\subset\mathcal{F}(M^\sigma)$ the full $A_\infty$-subcategory formed by $V_1^\sigma,\dots,V_k^\sigma$, the algebraic constructions above can be translated into geometry via the following:
\begin{proposition}[Lemma 6.3 of $\cite{ps4}$]\label{proposition:algebraic-geometric}
Let $\mathbb{K}$ be any field. There is a quasi-isomorphism between $A_\infty$-categories over $\mathbb{K}$:
\begin{equation}
\mathcal{V}(M^\sigma)\cong\mathcal{V}(M)^\sigma.
\end{equation}
\end{proposition}
\paragraph{Remark} In $\cite{ps4}$, there is an additional assumption that $\mathrm{char}(\mathbb{K})\neq2$ in the above proposition. This is because the Fukaya category $\mathcal{A}(\pi)$ is defined to be the $\mathbb{Z}/2$-invariant subcategory of the Fukaya category $\mathcal{F}(\widetilde{E})$, where $\widetilde{E}$ is a double cover of $E$ branched along the smooth fiber $M_\ast$. However, this assumption can actually be removed since the Fukaya category $\mathcal{A}(\pi)$ can also be defined directly on $E$ using a particular class of Hamiltonian perturbations specified in $\cite{ps9}$, without passing to the double branched cover $\widetilde{E}$. Alternatively, $\mathcal{A}(\pi)$ can be defined as a version of partially wrapped Fukaya category $\cite{gps1}$.
\bigskip

The weak Calabi-Yau property of the $A_\infty$-category $\mathcal{V}(M)$ implies the existence of a quasi-isomorphism
\begin{equation}\label{eq:CY}
\mathcal{V}(M)/\mathcal{A}(\pi)\cong\mathcal{A}(\pi)^\vee[-n+1]
\end{equation}
between $\mathcal{A}(\pi)$-bimodules. Note that this is \textit{strictly weaker} than the assumption in Lemma \ref{lemma:trivial-extension}, which requires the existence of a quasi-isomorphism
\begin{equation}\label{eq:sCY}
\mathcal{V}(M)\cong\mathcal{A}(\pi)\oplus\mathcal{A}(\pi)^\vee[-n+1]
\end{equation}
between $\mathcal{A}(\pi)$-bimodules.  Because of this, in general it is not true that the $A_\infty$-category $\mathcal{V}(M^\sigma)$ is quasi-equivalent to the trivial extension $\mathcal{A}(\pi)\oplus\mathcal{A}(\pi)^\vee[-n]$.

\paragraph{Remark} In a recent paper $\cite{lu}$, Lekili and Ueda prove that the Fukaya $A_\infty$-algebras of certain distinguished bases of vanishing cycles in the Milnor fibers of weighted homogeneous singularities with a conditions on the weights are non-formal. In particular, the Milnor fibers of the surface singularities of the form  
\begin{equation}
x^p+y^q+z^2=0,\textrm{ with }\frac{1}{p}+\frac{1}{q}<\frac{1}{2}
\end{equation}
provide explicit examples of Liouville manifolds $M$ which satisfy (\ref{eq:CY}) but not (\ref{eq:sCY}).
\bigskip

In the other direction, it is proved by Seidel in $\cite{ps4}$ that the condition (\ref{eq:sCY}) is satisfied for fibers $M^\sigma$ obtained by once suspensions, namely there is a quasi-isomorphism
\begin{equation}
\mathcal{V}(M^\sigma)\cong\mathcal{A}(\pi)\oplus\mathcal{A}(\pi)^\vee[-n]
\end{equation}
as $\mathcal{A}(\pi)$-bimodules. Denote by $\mathcal{V}(M^{\sigma\sigma})$ the $A_\infty$-category of vanishing cycles $V_1^{\sigma\sigma},\dots,V_k^{\sigma\sigma}\subset M^{\sigma\sigma}$ of the double suspension $\pi^{\sigma\sigma}:E\times\mathbb{C}^2\rightarrow\mathbb{C}$. As a corollary of Lemma \ref{lemma:trivial-extension}, we have the following:
\begin{corollary}[Corollary 6.5 of $\cite{ps4}$]\label{corollary:double-susp}
$\mathcal{V}(M^{\sigma\sigma})$ is quasi-equivalent to the trivial extension $\mathcal{A}(\pi)\oplus\mathcal{A}(\pi)^\vee[-n-1]$. In particular, its endomorphism algebra $\mathcal{V}_{M^{\sigma\sigma}}$ is formal as an $A_\infty$-algebra over $\Bbbk$ if $\mathcal{A}_\pi$ is formal.
\end{corollary}

\subsection{Generalized Eilenberg-Moore equivalence}\label{section:GEM}

Let $M_{-\Lambda}^\mathrm{in}$ be a Weinstein domain with $c_1(M_{-\Lambda})=0$, and let $S^\mathrm{in}_v\subset M_{-\Lambda}^\mathrm{in}$ be a Lagrangian submanifold with Legendrian boundary $\Lambda_v\subset\partial M_{-\Lambda}^\mathrm{in}$. Assume that $S^\mathrm{in}_v$ is oriented, $\mathit{Spin}$, and has vanishing Maslov class. Fix a finite set $\Gamma$, assume that the Lagrangian submanifolds $S_v$ for $v\in\Gamma$ intersect with each other transversely, and that $\Lambda_v$ are mutually disjoint. Put
\begin{equation}
S^\mathrm{in}=\bigcup_{v\in\Gamma}S_v^\mathrm{in},\Lambda=\bigsqcup_{v\in\Gamma}\Lambda_v.
\end{equation}
We denote by $M^\mathrm{in}$ the Weinstein domain obtained by attaching handles to $M_{-\Lambda}^\mathrm{in}$ along $\Lambda$, and by $S_v$ the union of $S_v^\mathrm{in}$ with the Lagrangian core disc of the Weinstein handle attached along $\Lambda_v$. Note that $S_v\subset M$ is a closed exact Lagrangian submanifold, and its Floer cochain complexes $\mathit{CF}^\ast(S_v,S_w)$ for $v,w\in\Gamma$ are well-defined. The geometric data above give rise to two $A_\infty$-algebras over the semisimple ring $\Bbbk:=\bigoplus_{v\in\Gamma}\mathbb{K}e_v$, namely the Fukaya $A_\infty$-algebra
\begin{equation}
\mathcal{V}_M:=\bigoplus_{v,w\in\Gamma}\mathit{CF}^\ast(S_v,S_w)
\end{equation}
and the Chekanov-Eliashberg algebra $\mathcal{C}E(\Lambda)$. By $\cite{ehk}$, the Lagrangian fillings $S_v^\mathrm{in}$ of $\Lambda_v$ give rise to an augmentation
\begin{equation}
\varepsilon_S:\mathcal{C}E(\Lambda)\rightarrow\Bbbk.
\end{equation}
\begin{theorem}[Theorem 4 of $\cite{ekl}$]\label{theorem:gem}
There exists a quasi-isomorphism between $A_\infty$-algebras
\begin{equation}\label{eq:gem}
R\mathrm{Hom}_{\mathcal{C}E(\Lambda)}(\Bbbk,\Bbbk)\cong\mathcal{V}_M.
\end{equation}
\end{theorem}
This should be regarded as a generalization of the Eilenberg-Moore equivalence (\ref{eq:em}). In $\cite{ekl}$, the theorem is also proved for the more general case when $\mathcal{C}E(\Lambda)$ is the Chekanov-Eliashberg algebra with loop space coefficients and $\mathcal{V}_M$ is the endomorphism algebra of the infinitesimally wrapped Fukaya category.

\subsection{Quiver algebras as Fukaya categories}\label{section:quiver}

We prove Corollary \ref{corollary:formal} in this subsection. As a by-product, it enables us to identify a full subcategory of the derived Fukaya category $D^\mathit{perf}\mathcal{F}(M_{p,q,r})$ with the derived category of perfect modules over an $A_\infty$-category associated to quiver with potential introduced in Section \ref{section:CY algebra}. Similar but more sophisticated results have been obtained by Smith in $\cite{is}$.

When $\frac{1}{p}+\frac{1}{q}+\frac{1}{r}\leq1$, the Weinstein manifold $M_{p,q,r}$ is the Milnor fiber associated to the corresponding isolated singularity $t_{p,q,r}+w^2=0$. However, in this paper we are also interested in the case when $\frac{1}{p}+\frac{1}{q}+\frac{1}{r}>1$ and $p,q,r\geq2$, where $M_{p,q,r}$ is no longer a Milnor fiber, but a generalized Milnor fiber in the sense of $\cite{ak1}$. More precisely, this means that the polynomial $t_{p,q,r}+w^2$ on $\mathbb{C}^4$ has a finite number of isolated critical points, instead of a unique isolated critical point at the origin. By taking the intersection of a smooth fiber of $t_{p,q,r}+w^2:\mathbb{C}^4\rightarrow\mathbb{C}$ with a large ball $B^8\subset\mathbb{C}^4$, we still get a well-defined Weinstein domain whose completion is Weinstein deformation equivalent to $M_{p,q,r}$. Our considerations here will work for both of these cases.
\bigskip

In order to understand a full $A_\infty$-subcategory of the compact Fukaya categoty $\mathcal{F}(M_{p,q,r})$, we shall use the results assembled in the last two subsections. Consider the Lefschetz fibration $\tilde{t}_{p,q,r}:\mathbb{C}^3\rightarrow\mathbb{C}$ defined as a Morsification of $t_{p,q,r}(x,y,z)$, see (\ref{eq:4DMilnor}). Note that in our case, the smooth fiber of $\tilde{t}_{p,q,r}$ is symplectomorphic to the Milnor fiber $T_{p,q,r}\subset\mathbb{C}^3$ associated to the singularity $t_{p,q,r}(x,y,z)=0$. The suspension of $\tilde{t}_{p,q,r}$ is the Lefschetz fibration on $\mathbb{C}^4$ defined by
\begin{equation}
\tilde{f}_{p,q,r}(x,y,z,w):=\tilde{t}_{p,q,r}(x,y,z)+w^2.
\end{equation}
Its smooth fiber is symplecomorphic to the Milnor fiber $M_{p,q,r}\subset\mathbb{C}^4$, whose symplectic topology is the main interest of this paper.
\bigskip

Let $\mathbb{K}$ be any field. Consider the full $A_\infty$-subcategory $\mathcal{V}(M_{p,q,r})\subset\mathcal{F}(M_{p,q,r})$ whose objects are vanishing cycles $V_1,\dots,V_{p+q+r-1}$ in the Milnor fiber $M_{p,q,r}$. Since $M_{p,q,r}$ is a fiber of the suspension $\tilde{t}_{p,q,r}^{\sigma}:\mathbb{C}^4\rightarrow\mathbb{C}$, the vanishing cycles in $M_{p,q,r}$ can be interpreted as double covers of the Lefschetz thimbles of $\tilde{t}_{p,q,r}$ branched along the vanishing cycles in the Milnor fiber $T_{p,q,r}$. Thus our notation for $\mathcal{V}(M_{p,q,r})$ is compatible with our previous convention, when the manifold of interest is not necessarily a Milnor fiber. As a consequence, the $A_\infty$-category $\mathcal{V}(M_{p,q,r})$ is a \textit{deformation} of the trivial extension
\begin{equation}
\mathcal{A}(\tilde{t}_{p,q,r})\oplus\mathcal{A}(\tilde{t}_{p,q,r})^\vee[-3].
\end{equation}

To show that this deformation is trivial, we use Theorem \ref{theorem:main} and the generalized Eilenberg-Moore equivalence (\ref{eq:gem}) learned in the last subsection. Take $M_{-\Lambda}^\mathrm{in}$ in Section \ref{section:GEM} to be the standard symplectic ball $D^6$, and let $\Lambda=\Lambda_{p,q,r}$ be the link of Legendrian 2-spheres obtained in Proposition \ref{proposition:link}. It follows that $M=M_{p,q,r}$ and $\mathcal{V}_M=\mathcal{V}_{p,q,r}$. By Theorem \ref{theorem:gem}, there is a quasi-isomorphism
\begin{equation}\label{eq:gem1}
R\mathrm{Hom}_{\mathcal{C}E(\Lambda_{p,q,r})}(\Bbbk,\Bbbk)\cong\mathcal{V}_{p,q,r}.
\end{equation}
On the other hand, our computation of the Chekanov-Eliashberg algebra in Theorem \ref{theorem:main} gives the quasi-isomorphism
\begin{equation}\label{eq:main}
\mathcal{C}E(\Lambda_{p,q,r})\cong\Pi_3(\mathcal{A}_{p,q,r}).
\end{equation}
In order to compute the Koszul dual of the 3-Calabi-Yau completion $\Pi_3(\mathcal{A}_{p,q,r})$ on the right-hand side above, it is more convenient to have an explicit model for it. the following Lemma enables us to identify $\Pi_3(\mathcal{A}_{p,q,r})$ with a Ginzburg algebra associated to some quiver with potential $(Q_{p,q,r},w_{p,q,r})$.
\begin{lemma}
The directed $A_\infty$-algebra $\mathcal{A}_{p,q,r}$ is quasi-isomorphic to a quiver algebra with global dimension no more than 2.
\end{lemma}
\begin{proof}
It is proved by Keating in $\cite{ak1}$ that $\mathcal{A}_{p,q,r}$ is quasi-isomorphic to the endomorphism algebra of a tilting object for the hereditary category $\mathit{Coh}(\mathbb{P}^1_{p,q,r})$, where $\mathit{Coh}(\mathbb{P}^1_{p,q,r})$ is the abelian category of coherent sheaves on the weighted projective line $\mathbb{P}^1_{p,q,r}$. More precisely, it follows from the computation in $\cite{ak1}$ that $\mathcal{A}_{p,q,r}$ can be identified with the graded associative algebra $\mathbb{K}\overrightarrow{Q}_{p,q,r}/\overrightarrow{I}_{p,q,r}$ associated to the following directed quiver
\begin{equation}
\begin{tikzcd}
& & \bullet_{P_1} \arrow[r,"x_1"]
& \bullet_{P_2} \arrow[r,"x_2"] & \bullet\dots\bullet \arrow[r,"x_{p-2}"] & \bullet_{P_{p-1}} \\
\bullet_A \arrow[r,bend left,"a_1"] \arrow[r,bend right,"a_2"] & \bullet_B
\arrow[ur,"b_1"]
\arrow[dr,"b_3"]
\arrow[r,"b_2"] & \bullet_{Q_1}
\arrow[r,"y_1"] & \bullet_{Q_2} \arrow[r,"y_2"] & \bullet\dots\bullet \arrow[r,"y_{q-2}"] & \bullet_{Q_{q-1}} \\
& & \bullet_{R_1} \arrow[r,"z_1"]
& \bullet_{R_2} \arrow[r,"z_2"] & \bullet\dots\bullet \arrow[r,"z_{r-2}"] & \bullet_{R_{r-1}}
\end{tikzcd}
\end{equation}
with relations in $\overrightarrow{I}_{p,q,r}$ given by
\begin{equation}\label{eq:rel}
b_2\circ a_1=0,b_1\circ a_2=0,b_3\circ(a_1-a_2)=0.
\end{equation}
It then follows that $\mathbb{K}\overrightarrow{Q}_{p,q,r}/\overrightarrow{I}_{p,q,r}$ is a \textit{quasi-tilted algebra of canonical type}, see $\cite{ls}$. Quasi-tilted algebras are studied in $\cite{hrs}$, and they are characterized by having global dimension at most 2 and each indecomposable module having projective dimension at most 1 or injective dimension at most 1. 
\end{proof}

By the above lemma, the conditions imposed on the quiver algebra $\mathcal{A}$ in Section \ref{section:CY-completion} is satisfied for $\mathbb{K}\overrightarrow{Q}_{p,q,r}/\overrightarrow{I}_{p,q,r}$. Since by its construction recalled in Section \ref{section:CY-completion}, the Calabi-Yau completion $\Pi_3(\mathcal{A}_{p,q,r})$ is unchanged up to quasi-isomorphism by replacing $\mathcal{A}_{p,q,r}$ with the quiver algebra $\mathbb{K}\overrightarrow{Q}_{p,q,r}/\overrightarrow{I}_{p,q,r}$, which is quadratic by (\ref{eq:rel}), using Proposition \ref{proposition:CY-completion} we can identify the 3-Calabi-Yau completion of $\mathcal{A}_{p,q,r}$ with the Ginzburg algebra $\mathcal{G}_{p,q,r}:=\mathcal{G}(Q_{p,q,r},w_{p,q,r})$ defined by the following quiver $Q_{p,q,r}$
\begin{equation}\label{eq:quiver}
\begin{tikzcd}
& & \bullet_{P_1} \arrow[r,"x_1"] \arrow[dll,orange,bend right,"c_1"]
& \bullet_{P_2} \arrow[r,"x_2"] & \bullet\dots\bullet \arrow[r,"x_{p-2}"] & \bullet_{P_{p-1}} \\
\bullet_A \arrow[r,bend left,"a_1"] \arrow[r,bend right,"a_2"] & \bullet_B
\arrow[ur,"b_1"]
\arrow[dr,"b_3"]
\arrow[r,"b_2"] & \bullet_{Q_1} \arrow[ll,bend left,orange,"c_2"]
\arrow[r,"y_1"] & \bullet_{Q_2} \arrow[r,"y_2"] & \bullet\dots\bullet \arrow[r,"y_{q-2}"] & \bullet_{Q_{q-1}} \\
& & \bullet_{R_1} \arrow[r,"z_1"] \arrow[ull,orange,bend left,"c_3"]
& \bullet_{R_2} \arrow[r,"z_2"] & \bullet\dots\bullet \arrow[r,"z_{r-2}"] & \bullet_{R_{r-1}}
\end{tikzcd}
\end{equation}
with potential
\begin{equation}
w_{p,q,r}=a_1b_2c_2+a_2b_1c_1+a_1b_3c_3-a_2b_3c_3.
\end{equation}
By Proposition \ref{proposition:Koszul},
\begin{equation}\label{eq:Koszul2}
R\mathrm{Hom}_{\mathcal{G}_{p,q,r}}(\Bbbk,\Bbbk)\cong\mathcal{B}_{p,q,r},
\end{equation}
where $\mathcal{B}_{p,q,r}:=\mathcal{B}(Q_{p,q,r},w_{p,q,r})$ is the compact 3-Calabi-Yau algebra associated to the same quiver with potential $(Q_{p,q,r},w_{p,q,r})$.
\begin{lemma}\label{lemma:quiver}
There is a quasi-isomorphism
\begin{equation}
\mathcal{V}_{p,q,r}\cong\mathcal{B}_{p,q,r}
\end{equation}
between $A_\infty$-algebras over $\Bbbk$.
\end{lemma}
\begin{proof}
By our discussions in Section \ref{section:Koszul}, one can use a quasi-isomorphic replacement of the Chekanov-Eliashberg dg algebra $\mathcal{C}E(\Lambda_{p,q,r})$ when computing its Koszul dual. In order to make use of (\ref{eq:Koszul2}) to compute the left hand side of (\ref{eq:gem1}), besides the quasi-isomorphism (\ref{eq:main}), we need to show that the augmentation $\varepsilon_V:\mathcal{C}E(\Lambda_{p,q,r})\rightarrow\Bbbk$ induced by the Lagrangian fillings $V_1^\mathrm{in},\dots,V_{p+q+r-1}^\mathrm{in}$ of the Legendrian link $\Lambda_{p,q,r}$ by vanishing cycles corresponds, under (\ref{eq:main}), to the trivial projection $\varepsilon:\mathcal{G}_{p,q,r}\rightarrow\Bbbk$, which is the augmentation that we used on the left hand side of (\ref{eq:Koszul2}) to compute the Koszul dual of $\mathcal{G}_{p,q,r}$. To see this, we need to refer to the explicit quasi-isomorphism (\ref{eq:222}) between the cellular dg algebra $\mathcal{C}(\Lambda_{2,2,2})$ and the Ginzburg algebra $\mathcal{G}_{2,2,2}$, which shows that the degree zero generators $a_1,a_2,b_1,b_2,b_3,c_1,c_2,c_3$ of the Ginzburg algebra $\mathcal{G}_{2,2,2}$ correspond geometrically to Reeb chords between different components of $\Lambda_{2,2,2}$ under the quasi-isomorphism (\ref{eq:main}). By definition, the image of these Reeb chords under $\varepsilon_V$ are zero, which shows that $(\mathcal{C}E(\Lambda_{2,2,2}),\varepsilon_V)$ and $(\mathcal{G}_{2,2,2},\varepsilon)$ are quasi-isomorphic as augmented dg algebras. The general case can be argued in a completely identical way, as the newly created degree 0 generators in $\mathcal{C}E(\Lambda_{p,q,r})$ after attaching unknoted Legendrian spheres to $\Lambda_{2,2,2}$ are also Reeb chords between different components of $\Lambda_{p,q,r}$.
\end{proof}

Since $\mathcal{B}_{p,q,r}$ is by construction the cyclic completion $\mathcal{A}_{p,q,r}\oplus\mathcal{A}_{p,q,r}^\vee[-3]$, we conclude that there is a quasi-isomorphism
\begin{equation}
\mathcal{V}_{p,q,r}\cong\mathcal{A}_{p,q,r}\oplus\mathcal{A}_{p,q,r}^\vee[-3],
\end{equation}
which implies the formality of the $A_\infty$-algebra $\mathcal{V}_{p,q,r}$, by Proposition \ref{proposition:formality}.

\section{Split-generation and quasi-dilations}

Assuming Koszul duality between the $A_\infty$-algebras $\mathcal{V}_{p,q,r}$ and $\mathcal{C}E(\Lambda_{p,q,r})$, we prove in this section the Corollaries \ref{corollary:split-generation} and \ref{corollary:quasi-dilation} stated in Section \ref{section:application}.

\subsection{Split-generation}\label{section:split-generation}

As in Section \ref{section:Koszul}, denote by $\mathcal{G}$ and $\mathcal{B}$ the Ginzburg dg algebra and the cyclic $A_\infty$-algebra associated to the same quiver with potential $(Q,w)$. By Proposition \ref{proposition:Koszul}, when $w$ is cubic, $\mathcal{B}=E(\mathcal{G})$ is the (bigraded) Koszul dual of $\mathcal{G}$.
\bigskip

Recall that all the $A_\infty$-modules over $\mathcal{G}$ form a dg category $\mathcal{G}^\mathit{mod}$. There is a \textit{Koszul duality functor}
\begin{equation}
\mathcal{K}:\mathcal{G}^\mathit{mod}\rightarrow\mathcal{B}^\mathit{mod}
\end{equation}
defined by $R\mathrm{Hom}_\mathcal{G}(\Bbbk,\cdot)$. This is first introduced by Beilinson-Ginzburg-Soergel in $\cite{bgs}$ for Koszul algebras, its generalization in the context of $A_\infty$-Koszul duality is straightforward.

Denote by $D^\mathit{prop}(\mathcal{G})$ the derived category of proper $\mathcal{G}$-modules, namely those $\mathcal{G}$-modules whose cohomologies are finite dimensional; and by $D^\mathit{perf}(\mathcal{B})$ the derived category of perfect $\mathcal{B}$-modules, which is obtained by taking the split-closure of the homotopy category of the $A_\infty$-category of twisted complexes over $\mathcal{B}$, i.e. $H^0(\mathit{Tw}(\mathcal{B}))$.
\begin{proposition}[$\cite{la,kn}$]\label{proposition:functor}
Assume that $H^\ast(\mathcal{G})$ is finite-dimensional in each fixed degree. The restriction of the Koszul duality functor $\mathcal{K}$ induces an equivalence
\begin{equation}
D\mathcal{K}:D^\mathit{prop}(\mathcal{G})\rightarrow D^\mathit{perf}(\mathcal{B}).
\end{equation}
\end{proposition}

\paragraph{Remark} Without the assumption that $H^\ast(\mathcal{G})$ is finite-dimensional in each degree, we have instead an equivalence
\begin{equation}
D^\mathit{prop}(\widehat{\mathcal{G}})\cong D^\mathit{perf}(\mathcal{B}),
\end{equation}
where $\widehat{\mathcal{G}}$ is the complete Ginzburg algebra. When $H^\ast(\mathcal{G})$ is locally finite with respect to the total grading, there is a quasi-isomorphism $\mathcal{G}\cong\widehat{\mathcal{G}}$.

The bigraded version of the Koszul duality functor $\mathcal{K}$ induces an equivalence between $D^\mathit{prop}(\mathcal{G})$ and $D^\mathit{perf}(\mathcal{B})$ whenever $\mathcal{G}$ is \textit{Adams connected} as a bigraded $A_\infty$-algebra, see Theorem B of $\cite{lpwz}$. Although the Adams connectedness condition is satisfied for any Ginzburg algebra $\mathcal{G}$ defined by a quiver $Q$ with cubic potential $w$, this version of derived equivalence cannot be used to prove the split-generation of the compact Fukaya category $\mathcal{F}(M_{p,q,r})$ by vanishing cycles. This is because the derived category of bigraded $A_\infty$-modules over the bigraded $A_\infty$-algebra $\mathcal{G}$ or $\mathcal{B}$ is in general unrelated to the corresponding singly graded version.
\bigskip

Another ingredient which is relevant for the proof of Corollary \ref{corollary:split-generation} is the split-generation of the wrapped Fukaya category by Lagrangian cocores, which should follow by combing the work of Bourgeois-Ekholm-Eliashberg $\cite{bee}$ on Legendrian surgery with Abouzaid's geometric generation criterion $\cite{ma1}$. For convenience, we state the following generation result, which appears in the recent work $\cite{cggr}$. A more general statement, which takes into account also Weinstein domains with stops, is proved in $\cite{gps2}$.
\begin{theorem}[$\cite{cggr,gps2}$]\label{theorem:generation}
Let $M$ be a Weinstein manifold, which can be realized as the result of Weinstein handle attachment to $D^{2n}$. Denote by $L_1,\dots,L_k$ the Lagrangian cocore discs in the $n$-handles, the wrapped Fukaya category $\mathcal{W}(M)$ over any field $\mathbb{K}$ is generated by $L_1,\dots,L_k$, equipped with appropriate brane structures which make them objects of $\mathcal{W}(M)$.
\end{theorem}
\bigskip

We are now prepared to prove Corollary \ref{corollary:split-generation} by combining the facts stated above. In our specific setting, the Milnor fiber $M_{p,q,r}$ can be constructed by attaching $(p+q+r-1)$ 3-handles along $\Lambda_{p,q,r}$ to the standard symplectic ball $D^6$, see Figure \ref{fig:front-without-2-handles}. Recall that $\mathcal{W}_{p,q,r}$ is the endomorphism algebra of the Lagrangian cocores in $M_{p,q,r}$.

Since we have proved the quasi-isomorphism $\mathcal{V}_{p,q,r}\cong\mathcal{B}_{p,q,r}$ in Section \ref{section:quiver}, and it follows from Theorem \ref{theorem:main} that $\mathcal{W}_{p,q,r}\cong\mathcal{G}_{p,q,r}$. By Proposition \ref{proposition:Koszul}, there is Koszul duality between the $A_\infty$-algebras $\mathcal{V}_{p,q,r}$ and $\mathcal{W}_{p,q,r}$, namely
\begin{equation}
E(\mathcal{V}_{p,q,r})\cong\mathcal{W}_{p,q,r},E(\mathcal{W}_{p,q,r})\cong\mathcal{V}_{p,q,r},
\end{equation}
where $\mathcal{V}_{p,q,r}$ and $\mathcal{W}_{p,q,r}$ above are equipped with bigradings which coincide with the ones on the quiver algebras $\mathcal{B}_{p,q,r}$ and $\mathcal{G}_{p,q,r}$ described in the proof of Proposition \ref{proposition:Koszul}. However, for geometric applications, we have to get rid of the double gradings and obtain a version of Koszul duality between $\mathcal{V}_{p,q,r}$ and $\mathcal{W}_{p,q,r}$ as $\mathbb{Z}$-graded $A_\infty$-algebras. This is possible by Proposition \ref{proposition:lf} and the following lemma.

\begin{lemma}\label{lemma:Adams}
Let $p\geq2,q\geq2,r\geq2$. $H^\ast(\mathcal{G}_{p,q,r})$ is finite-dimensional in each fixed degree.
\end{lemma}
\begin{proof}
From (\ref{eq:quiver}), we see that the only cycles of arrows in the quiver $Q_{p,q,r}$ are of the form $a_ib_jc_j$, where $i=1,2$ and $j=1,2,3$. To show that they vanish in the cohomology algebra $H^\ast(\mathcal{G}_{p,q,r})$, we need to take into account the relations coming from the differentials of the potential $w_{p,q,r}$. From $\partial w_{p,q,r}/\partial{c_1}$ and $\partial w_{p,q,r}/\partial{c_2}$, we see that $a_1b_2=a_2b_1=0$ in $H^\ast(\mathcal{G}_{p,q,r})$, so the only possible non-zero cycles in the cohomology algebra are $a_1b_1c_1$, $a_1b_3c_3$, $a_2b_2c_2$ and $a_2b_3c_3$. From $\partial w_{p,q,r}/\partial{a_1}$ and $\partial w_{p,q,r}/\partial{a_2}$, one gets the relations
\begin{equation}
b_2c_2+b_3c_3=0,b_1c_1-b_3c_3=0
\end{equation}
in $H^\ast(\mathcal{G}_{p,q,r})$, so we are reduced to show that $a_1b_1c_1=a_2b_2c_2=0$. From $\partial w_{p,q,r}/\partial c_3$, we get $a_1b_3=a_2b_3$, so it suffices to prove that $a_1b_1c_1=0$. But from the above we get
\begin{equation}
a_1b_1c_1=a_1b_3c_3=-a_1b_2c_2=0.
\end{equation}
\end{proof}

The above lemma enables us to apply Proposition \ref{proposition:functor}, from which we get an equivalence
\begin{equation}
D^\mathit{prop}(\mathcal{G}_{p,q,r})\cong D^\mathit{perf}(\mathcal{B}_{p,q,r}).
\end{equation}
By Lemma 3.25 of $\cite{ps1}$ and Lemma \ref{lemma:quiver}, there is a derived equivalence $D^\mathit{perf}(\mathcal{B}_{p,q,r})\cong D^\mathit{perf}(\mathcal{V}_{p,q,r})$. Similarly, Theorem \ref{theorem:main} implies that $D^\mathit{perf}(\mathcal{G}_{p,q,r})\cong D^\mathit{perf}(\mathcal{W}_{p,q,r})$. Since both of $\mathcal{G}_{p,q,r}$ and $\mathcal{W}_{p,q,r}$ are homologically smooth $A_\infty$-algebras, by restricting to the full subcategory of proper modules we get an equivalence $D^\mathit{prop}(\mathcal{G}_{p,q,r})\cong D^\mathit{prop}(\mathcal{W}_{p,q,r})$, see Lecture 7 of $\cite{ps2}$. We have thus interpreted the derived equivalence in Proposition \ref{proposition:functor} as an equivalence between derived categories of certain modules over Fukaya $A_\infty$-algebras:
\begin{equation}\label{eq:pp}
D^\mathit{prop}(\mathcal{W}_{p,q,r})\cong D^\mathit{perf}(\mathcal{V}_{p,q,r}).
\end{equation}
By Theorem \ref{theorem:generation}, we have an identification
\begin{equation}
D^\mathit{perf}(\mathcal{W}(M_{p,q,r}))\cong D^\mathit{perf}(\mathcal{W}_{p,q,r}).
\end{equation}
It is proved by Ganatra in $\cite{sg}$ that the wrapped Fukaya category of any Weinstein manifold is a homologically smooth $A_\infty$-category (in its general form, this fact depends on the results claimed in the work-in-progress $\cite{bee}$), therefore the above equivalence restricts to an equivalence
\begin{equation}
D^\mathit{prop}(\mathcal{W}(M_{p,q,r}))\cong D^\mathit{prop}(\mathcal{W}_{p,q,r}).
\end{equation}
Using this equivalence, (\ref{eq:pp}) can be interpreted equivalently as an equivalence
\begin{equation}
D\mathcal{K}:D^\mathit{prop}(\mathcal{W}(M_{p,q,r}))\xrightarrow{\simeq}D^\mathit{perf}(\mathcal{V}(M_{p,q,r})).
\end{equation}
On the other hand, since each object of $\mathcal{F}(M_{p,q,r})$ can be regarded as a proper $A_\infty$-module over $\mathcal{W}_{p,q,r}$ under the Yoneda functor
\begin{equation}
\mathcal{Y}:\mathcal{W}(M_{p,q,r})\rightarrow\mathcal{W}_{p,q,r}^\mathit{mod},
\end{equation}
we get a fully faithful embedding
\begin{equation}
D\mathcal{J}:D^\mathit{perf}(\mathcal{F}(M_{p,q,r}))\hookrightarrow D^\mathit{prop}(\mathcal{W}(M_{p,q,r})).
\end{equation}
Combining with the equivalence $D\mathcal{K}$, we obtain a fully faithful embedding
\begin{equation}
D\mathcal{K}\circ D\mathcal{J}:D^\mathit{perf}(\mathcal{F}(M_{p,q,r}))\hookrightarrow D^\mathit{perf}(\mathcal{V}(M_{p,q,r})),
\end{equation}
which proves the split-generation of $\mathcal{F}(M_{p,q,r})$ by vanishing cycles.

\paragraph{Remark} As a by-product of the above argument, we get the equivalence
\begin{equation}
D^\mathit{prop}(\mathcal{W}(M_{p,q,r}))\cong D^\mathit{perf}(\mathcal{F}(M_{p,q,r})),
\end{equation}
which is expected to be true for a general Weinstein manifold $M$, namely when $\mathcal{F}(M)$ and $\mathcal{W}(M)$ are not necessarily related by $A_\infty$-Koszul duality. However, the fully faithfulness of the functor $D\mathcal{K}$ above relies on the fact that
\begin{equation}
R\mathrm{Hom}_{\mathcal{V}_{p,q,r}}(\Bbbk,\Bbbk)\cong\mathcal{W}_{p,q,r}
\end{equation}
as $\mathbb{Z}$-graded $A_\infty$-algebras. Because of this, in general the generation of the wrapped Fukaya category $\mathcal{W}(M)$ by cocores does not lead to a split-generation result of the compact Fukaya category $\mathcal{F}(M)$.

\subsection{Quasi-dilation}\label{section:quasi-dilation}

In this subsection, $\mathcal{A}$ will be a special kind of a directed $A_\infty$-algebra over some semisimple ring $\Bbbk$.

\begin{definition}[Definition 5.10 of $\cite{cmrs}$]
A one-way algebra $\mathcal{A}$ is a finite dimensional algebra over $\Bbbk$ with a complete set $\{e_1,\cdots,e_r\}$ of orthogonal idempotents such that
\begin{itemize}
	\item for $i\neq j$, if $e_i\mathcal{A}e_j\neq0$, then $e_j\mathcal{A}e_i=0$;
	\item for any idempotent $e_i$, we have $\dim_\mathbb{K}(e_i\mathcal{A}e_i)=1$;
	\item $r>1$ and $\mathcal{A}$ is indecomposable.
\end{itemize}
\end{definition}

Note that many of the known examples (say those studied in $\cite{ak2}$) of the directed Fukaya categories $\mathcal{A}(\pi)$ associated to Lefschetz fibrations $\pi$ can actually be identified with one-way algebras. In particular, the directed $A_\infty$-algebras $\mathcal{A}_{p,q,r}$ encountered in Section \ref{section:quiver} are one-way algebras.
\bigskip

From now on, assume that $\mathcal{A}$ is an one-way algebra. Denote by $\mathcal{B}=\mathcal{A}\oplus\mathcal{A}^\vee[-n]$ the cyclic completion of $\mathcal{A}$. In particular, $\mathcal{B}$ is a compact $n$-Calabi-Yau algebra. Since $\mathcal{A}$ is formal, then so is $\mathcal{B}$ by Proposition \ref{proposition:formality}. Since we are actually working with $\mathbb{Z}$-graded associative algebras over $\Bbbk$, we will mainly take the algebraic, rather than the categorical point of view in this subsection. Recall that the Hochschild cochain complex $\mathit{CC}^\ast(\mathcal{B},\mathcal{B})$ is defined to be
\begin{equation}
\hom_\Bbbk(T\overline{\mathcal{B}},\mathcal{B}),
\end{equation}
the space of $\Bbbk$-linear maps from the reduced tensor algebra of $\mathcal{B}$ to $\mathcal{B}$. Using the grading on $\mathcal{B}$, we get the decomposition
\begin{equation}
\mathit{CC}^d(\mathcal{B},\mathcal{B})=\bigoplus_{d=r+s}\mathit{CC}^r(\mathcal{B},\mathcal{B}[s]),
\end{equation}
where the right-hand side is the subspace of linear maps $\mathcal{B}^{\otimes r}\rightarrow\mathcal{B}$ of degree $s$. In particular, the Hochschild complex of a graded algebra is bigraded, and the corresponding Hochschild differential has bidegree $(1,0)$.
\bigskip

Formality of the $A_\infty$-algebra $\mathcal{B}$ implies that there is a distinguished Hochschild cocycle
\begin{equation}
\mathit{eu}_\mathcal{B}\in\mathit{CC}^1(\mathcal{B},\mathcal{B})
\end{equation}
defined by sending a homogeneous element $b$ with $|b|=i$ to $i\cdot b$. The fact that
\begin{equation}
|b_2b_1|=|b_1|+|b_2|
\end{equation}
then implies that $\mathit{eu}_\mathcal{B}$ is a derivation, namely $\mathit{eu}_\mathcal{B}\in\mathit{HH}^1(\mathcal{B},\mathcal{B})$. We call $\mathit{eu}_\mathcal{B}$ the \textit{Euler vector field}.
\bigskip

As a trivial extension, $\mathcal{B}$ is easily seen to be a ($\mathbb{Z}$-graded) symmetric algebra, whose non-degenerate inner product
\begin{equation}
\langle\cdot,\cdot\rangle:\mathcal{B}\otimes\mathcal{B}\rightarrow\Bbbk
\end{equation}
is defined by
\begin{equation}
\langle b_1,b_2\rangle=a_1a_2^\vee+a_1^\vee a_2,
\end{equation}
where $b_i=(a_i,a_i^\vee)$ with $a_i\in\mathcal{A}$ and $a_i^\vee\in\mathcal{A}^\vee[-n]$ for $i=1,2$. This inner product is induced from the compact Calabi-Yau structure on $\mathcal{B}$.

It is proved by Tradler in $\cite{tt}$ that there is an algebraically defined BV operator
\begin{equation}
\Delta_{cyc}:\mathit{HH}^\ast(\mathcal{B},\mathcal{B})\rightarrow\mathit{HH}^{\ast-1}(\mathcal{B},\mathcal{B})
\end{equation}
on the Hochschild cohomology of any cyclic $A_\infty$-algebra $\mathcal{B}$, which has bidegree $(-1,0)$. For our purposes here, we shall omit the full formulae, and concentrate on its simplest piece on Hochschild cochains with pure degree $(1,0)$:
\begin{equation}
\langle\Delta_{cyc}(c),b\rangle=\langle c(b),\mathrm{id}_\mathcal{B}\rangle,
\end{equation}
where $b\in\mathcal{B}$, $c\in\mathit{CC}^1(\mathcal{B},\mathcal{B}[0])$ and $\mathrm{id}_\mathcal{B}$ is the identity of $\mathcal{B}$.
\bigskip

Our main result in this subsection the following:
\begin{proposition}\label{proposition:quasi-dilation}
Let $\mathcal{B}$ be the n-cyclic completion of a one-way algebra $\mathcal{A}$ over $\Bbbk$, then
\begin{equation}
\Delta_{cyc}\left(\frac{1}{n}\mathit{eu}_\mathcal{B}\right)=1.
\end{equation}
\end{proposition}
\begin{proof}
Let $a$ be a generator of $\mathcal{A}$. Since $\mathcal{A}$ is trivially graded, $|(a,0)|=0$ in $\mathcal{B}$. It follows from the definition of $\mathcal{B}$ that $|(0,a^\vee[-n])|=n$ in $\mathcal{B}$. Since $\mathit{eu}_\mathcal{B}$ has pure degree $(1,0)$, from the definition of the BV operator $\Delta_{cyc}$ it follows that
\begin{equation}\label{eq:exp1}
\left\langle\Delta_{cyc}\left(\frac{1}{n}\mathit{eu}_\mathcal{B}\right),(0,a^\vee[-n])\right\rangle=\left\langle\frac{1}{n}\mathit{eu}_\mathcal{B}(0,a^\vee[-n]),\mathrm{id}_\mathcal{B}\right\rangle,
\end{equation}
By definition of the derivation $\frac{1}{n}\mathit{eu}_\mathcal{B}$ and the symmetric pairing $\langle\cdot,\cdot\rangle$ on a trivial extension algebra, the right-hand side of (\ref{eq:exp1}) is equal to
\begin{equation}\label{eq:int1}
a^\vee[-n]\cdot\mathrm{id}_\mathcal{A}.
\end{equation}

For any trivial extension algebra $\mathcal{B}$ of $\mathcal{A}$, it is easy to find that
\begin{equation}\label{eq:center}
Z(\mathcal{B})=Z(\mathcal{A})\ltimes\mathit{Ann}_{\mathcal{A}^\vee}(C(\mathcal{A})),
\end{equation}
where $Z(\mathcal{A})$ and $Z(\mathcal{B})$ denote the (ungraded) centres of $\mathcal{A}$ and $\mathcal{B}$ respectively, $C(\mathcal{A})\subset\mathcal{A}$ is the subspace of commutators, and
\begin{equation}
\mathit{Ann}_{\mathcal{A}^\vee}(V):=\{a^\vee\in\mathcal{A}^\vee|a^\vee(V)=0\}
\end{equation}
for any subspace $V\subset\mathcal{A}$, see $\cite{bhz}$.

The assumption that $\mathcal{A}$ is a one-way algebra implies easily that $Z(\mathcal{A})\cong\mathbb{K}$. On the other hand, since the BV operator $\Delta_{cyc}$ has bidegree $(-1,0)$,
\begin{equation}
\Delta_{cyc}\left(\frac{1}{n}\mathit{eu}_\mathcal{B}\right)\in Z(\mathcal{B})\subset\mathit{HH}^0(\mathcal{B},\mathcal{B}).
\end{equation}
By (\ref{eq:center}), $\Delta_{cyc}\left(\frac{1}{n}\mathit{eu}_\mathcal{B}\right)$ can be expressed as $(\lambda\cdot\mathrm{id}_\mathcal{A},\alpha^\vee[-n])$, where $\lambda\in\mathbb{K}$ and $\alpha^\vee\in\mathit{Ann}_{\mathcal{A}^\vee}(C(\mathcal{A}))$. Using this expression, the left-hand side of (\ref{eq:exp1}) becomes $(\lambda a^\vee[-n]+a\cdot\alpha^\vee[-n])\cdot\mathrm{id}_\mathcal{A}$, from which we get
\begin{equation}\label{eq:int2}
\lambda a^\vee[-n]+a\cdot\alpha^\vee[-n]=a^\vee[-n].
\end{equation}

Since we may choose $a\in C(\mathcal{A})$ in (\ref{eq:int2}), by definition of $\alpha^\vee$ we deduce $\lambda=1$ in the above. It follows that $a\cdot\alpha^\vee[-n]$ for all $a\in\mathcal{A}$, which forces $\alpha^\vee[-n]=0$. Thus we have proved that $\Delta_{cyc}(\frac{1}{n}\mathit{eu}_\mathcal{B})=1\in\mathit{HH}^0(\mathcal{B},\mathcal{B})$.
\end{proof}

\paragraph{Remark} A related result has been obtained by Schedler in $\cite{ts}$ for preprojective algebras $\Pi_Q$ associated to a non-Dynkin quiver $Q$. Since $\Pi_Q$ is 2-Calabi-Yau, the half-Euler vector field plays an important role there.
\bigskip

Suppose two Calabi-Yau algebras $\mathcal{B}$ and $\mathcal{G}$ are Koszul dual as $\mathbb{Z}$-graded $A_\infty$-algebras, it is first proved in $\cite{cyz}$ for classical Koszul Calabi-Yau algebras and later generalized to the chain level in $\cite{eh}$ that there is a BV algebra isomorphism
\begin{equation}
\mathit{HH}^\ast(\mathcal{B},\mathcal{B})\cong\mathit{HH}^\ast(\mathcal{G},\mathcal{G}).
\end{equation}
When $\mathcal{B}=\mathcal{A}\oplus\mathcal{A}^\vee[-n]$ is a cyclic completion of a one-way algebra, Proposition \ref{proposition:quasi-dilation} then implies the existence of a cohomology class in $\mathit{HH}^1(\mathcal{G},\mathcal{G})$, which we denote by $\frac{1}{n}\mathit{eu}_\mathcal{G}$. satisfying
\begin{equation}
\Delta_{CY}\left(\frac{1}{n}\mathit{eu}_\mathcal{G}\right)=1,
\end{equation}
where $\Delta_{CY}$ denotes the BV operator on $\mathit{HH}^\ast(\mathcal{G},\mathcal{G})$ defined by the Calabi-Yau structure on $\mathcal{G}$.
\bigskip

Back to our case, it follows from Lemma \ref{lemma:Adams} that the Koszul duality between the $\mathbb{Z}$-graded $A_\infty$-algebras $\mathcal{B}_{p,q,r}$ and $\mathcal{G}_{p,q,r}$ implies the existence of a cohomology class $\frac{1}{3}\mathit{eu}_{p,q,r}\in\mathit{HH}^1(\mathcal{G}_{p,q,r},\mathcal{G}_{p,q,r})$ which is mapped to the identity by the BV operator $\Delta_{CY}$, where $\Delta_{CY}$ is defined using the smooth 3-Calabi-Yau structure on the Ginzburg dg algebra $\mathcal{G}_{p,q,r}$. On the other hand, although the quasi-isomorphism
\begin{equation}
\mathcal{W}_{p,q,r}\cong\mathcal{G}_{p,q,r}
\end{equation}
induces an isomorphism between Hochschild cohomologies
\begin{equation}\label{eq:Gerstanhaber}
\mathit{HH}^\ast(\mathcal{W}(M_{p,q,r}),\mathcal{W}(M_{p,q,r}))\cong\mathit{HH}^\ast(\mathcal{G}_{p,q,r},\mathcal{G}_{p,q,r})
\end{equation}
as Gerstenhaber algebras, it is not in general true that (\ref{eq:Gerstanhaber}) preserves the underlying BV structures, as the Calabi-Yau structure on the wrapped Fukaya category $\mathcal{W}(M_{p,q,r})$ coming from symplectic geometry may differ from the Calabi-Yau structure on the Ginzburg dg algebra $\mathcal{G}_{p,q,r}$. The variation of the Calabi-Yau structure changes the BV operator $\Delta_{CY}$ on $\mathit{HH}^\ast(\mathcal{G}_{p,q,r},\mathcal{G}_{p,q,r})$ by the conjugate action of an invertible element
\begin{equation}
u\in\mathit{HH}^0(\mathcal{G}_{p,q,r},\mathcal{G}_{p,q,r})^\times.
\end{equation}
By the BV algebra isomorphism
\begin{equation}
\mathit{SH}^\ast(M_{p,q,r})\cong\mathit{HH}^\ast(\mathcal{W}(M_{p,q,r}),\mathcal{W}(M_{p,q,r}))
\end{equation}
proved by Ganatra $\cite{sg}$, we conclude the existence of a quasi-dilation $(\frac{1}{3}\mathit{eu}_{p,q,r},u)\in\mathit{SH}^1(M_{p,q,r})\times\mathit{SH}^0(M_{p,q,r})^\times$.

\paragraph{Remark} Another way of proving that $M_{p,q,r}$ admits a quasi-dilations over a field $\mathbb{K}$ with $\mathrm{char}(\mathbb{K})\neq2$ is to use the Lefschetz fibrations
\begin{equation}
\pi_{p,q,r}:M_{p,q,r}\rightarrow\mathbb{C}.
\end{equation}
Since the smooth fiber of $\pi_{p,q,r}$ is symplectomorphic to the 4-dimensional $D_4$ type Milnor fiber, one can apply Proposition \ref{proposition:quasi-dilation} to the zig-zag algebra $B_T$ with the tree $T=D_4$ to see that the fiber of $\pi_{p,q,r}$ admits a quasi-dilation. Together with Seidel-Solomon's inductive argument based on Lefschetz fibration techniques $\cite{ss}$, this implies that the total space $M_{p,q,r}$ also admits a quasi-dilation. However, this argument does not apply when $\mathbb{K}=\mathbb{Z}/2$, since in this case Koszul duality does not hold for $D_4$ Milnor fibers, see Theorem 14 of $\cite{etl1}$ for details.

\section{From Lefschetz fibrations to Legendrian fronts}

We describe a Lefschetz fibration, whose construction is essentially due to Keating $\cite{ak2}$, on $M_{p,q,r}$. With Casals-Murphy recipe $\cite{cm}$, we are able to get the Legendrian front presentation of this fibration, which, after simplifications, gives us the Legendrian attaching link $\Lambda_{p,q,r}\subset S^5$ of the Weinstein manifold $M_{p,q,r}$.

\subsection{A Lefschetz fibration on $M_{p,q,r}$}\label{section:Lefschetz 6-fold}

We start by recalling a Lefschetz fibration on the Milnor fibers $T_{p,q,r}\subset\mathbb{C}^3$ (defined by $(\ref{eq:4DMilnor})$) constructed by Keating in $\cite{ak2}$. The construction is divided into three steps.\bigskip

First, consider the Milnor fiber of two variables $F_{p,q}\subset\mathbb{C}^2$ defined by the equation
\begin{equation}
g_{p,q}(x,y):=(x^{p-2}-y^2)(x^2-\lambda y^{q-2})-1.
\end{equation}
Using A'Campo's method, the vanishing cycles of $F_{p,q}$ can be explicitly described by a divide of $\mathbb{R}^2$ induced by a real deformation $\tilde{g}_{p,q}$ of the polynomial $g_{p,q}$. After some mutations, we get a basis of vanishing cycles labelled by
\begin{equation}
A,B,P_1,\dots,P_{p-1},Q_1,\dots,Q_{q-1},R,
\end{equation}
see Figure \ref{fig:vsf}.
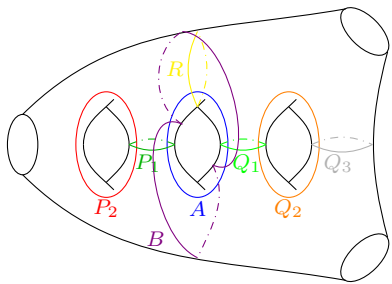
\begin{figure}
	\centering
	\begin{tikzpicture}
	\draw (0,0) ellipse (0.2cm and 0.4cm);
	\draw [rotate around={45:(4.5,1.5)}] (4.5,1.5) ellipse (0.2cm and 0.4cm);
	\draw [rotate around={-45:(4.5,-1.5)}] (4.5,-1.5) ellipse (0.2cm and 0.4cm);
	\draw (0,0.4) to [in=-170,out=45] (4.25,1.8);
	\draw (0,-0.4) to [in=170,out=-45] (4.25,-1.8);
	\draw (4.8,1.23) to [in=105,out=-105] (4.8,-1.23);
	
	\draw (1.2,0.6) to [in=90,out=-135] (0.8,0);
	\draw (0.8,0) to [in=135,out=-90] (1.2,-0.6);
	\draw (1.1,0.5) to [in=90,out=-45] (1.4,0);
	\draw (1.4,0) to [in=45,out=-90] (1.1,-0.5);
	
	\draw (2.4,0.6) to [in=90,out=-135] (2,0);
	\draw (2,0) to [in=135,out=-90] (2.4,-0.6);
	\draw (2.3,0.5) to [in=90,out=-45] (2.6,0);
	\draw (2.6,0) to [in=45,out=-90] (2.3,-0.5);
	
	\draw (3.6,0.6) to [in=90,out=-135] (3.2,0);
	\draw (3.2,0) to [in=135,out=-90] (3.6,-0.6);
	\draw (3.5,0.5) to [in=90,out=-45] (3.8,0);
	\draw (3.8,0) to [in=45,out=-90] (3.5,-0.5);
	
	\draw [red] (1.1,0) ellipse (0.4cm and 0.7cm);
	\node [red] at (1.1,-0.85) {\footnotesize $P_2$};
	\draw [blue] (2.3,0) ellipse (0.4cm and 0.7cm);
	\node [blue] at (2.3,-0.85) {\footnotesize $A$};
	\draw [orange] (3.5,0) ellipse (0.4cm and 0.7cm);
	\node [orange] at (3.5,-0.85) {\footnotesize $Q_2$};
	
	\draw [green] (2.6,0) [in=-155,out=-25] to (3.2,0);
	\node [green] at (2.95,-0.25) {\footnotesize $Q_1$};
	\draw [green,dash dot] (2.6,0) [in=155,out=25] to (3.2,0);
	
	\draw [black!30!green] (1.4,0) [in=-155,out=-25] to (2,0);
	\node [black!30!green] at (1.65,-0.25) {\footnotesize $P_1$};
	\draw [black!30!green,dash dot] (1.4,0) [in=155,out=25] to (2,0);
	
	\draw [black!30!white] (3.8,0) [in=-155,out=-25] to (4.6,0);
	\node [black!30!white] at (4.15,-0.25) {\footnotesize $Q_3$};
	\draw [black!30!white,dash dot] (3.8,0) [in=155,out=25] to (4.6,0);
	
	\draw [yellow] (2.3,0.5) [in=-115,out=115] to (2.3,1.5);
	\node [yellow] at (2,1) {\footnotesize $R$};
	\draw [yellow,dash dot] (2.3,0.5) [in=-65,out=65] to (2.3,1.5);
	
	\draw [violet,dash dot] (2.5,-0.27) [in=65,out=-65] to (2.3,-1.5);
	\draw [violet] (2.3,-1.5) [in=155,out=155] to (2.1,0.27);
	\draw [violet,dash dot] (2.1,0.27) [in=-150,out=150] to (2.1,1.5);
	\draw [violet] (2.1,1.5) [in=-25,out=-5] to (2.5,-0.27);
	\node [violet] at (1.75,-1.25) {\footnotesize $B$};
	\end{tikzpicture}
	\caption{The vanishing cycles of $F_{3,4}$}
	\label{fig:vsf}
\end{figure}
\bigskip

The second step is to produce a Lefschetz fibration on $T_{p,q,r}$. In this case, the Morsification of the polynomial
\begin{equation}
\tilde{g}_{p,q}(x,y)+z^r
\end{equation}
defines a generalized Milnor fiber $\widehat{T}_{p,q,r}$, and $T_{p,q,r}^\mathrm{in}$ embeds in it as a Liouville sub-domain. There is a Lefschetz fibration $\hat{\pi}_T:\widehat{T}_{p,q,r}\rightarrow\mathbb{C}$ defined by a Morsification of $z^r$. To get a Lefschetz fibration on $T_{p,q,r}$, one chooses carefully a 1-parameter family of polynomials $m_t(x,y,z)$, such that $m_0=\tilde{g}_{p,q}+\tilde{z}^r$, with $\tilde{z}^r$ being a Morsification of $z^r$, and the smooth affine surface defined by $m_1$ is symplectomorphic to $T_{p,q,r}$. During the deformation procedure, some of the critical values of $m_t$ disappear and the remaining ones correspond to the vanishing paths for the Lefschetz fibration on $\mathbb{C}^3$ defined by $\tilde{t}_{p,q,r}$. Similarly, the corresponding matching paths of the Lefschetz fibration $\hat{\pi}_T$ on $\widehat{T}_{p,q,r}$ also disappear under the deformation when $t\rightarrow1$, and the remaining ones are the matching paths for the Lefschetz fibration that we want on $T_{p,q,r}$. In this way, one gets the description of $T_{p,q,r}$ as the total space of a Lefschetz fibration $\pi_T:T_{p,q,r}\rightarrow\mathbb{C}$, whose smooth fiber is symplectomorphic to $F_{p,q}$.
\bigskip

The third step is to apply a sequence of destabilizations to the Lefschetz fibration $\pi_T$. We briefly recall the general construction of a stabilization. Let $\pi:M^\mathrm{in}\rightarrow D^2$ be an exact symplectic Lefschetz fibration on a 4-dimensional Liouville domain $M^\mathrm{in}$ with smooth fibers $F^\mathrm{in}$ a Riemann surface with boundary. Denote by $V_1,\dots,V_n\subset F^\mathrm{in}$ the vanishing cycles of $\pi$. Given an embedded arc $\gamma\subset F^\mathrm{in}$ with $\partial\gamma\subset\partial F^\mathrm{in}$ such that $\gamma$ is an exact Lagrangian submanifold in $F$ relative to its boundary $\partial\gamma$, one can construct a new Lefschetz fibration $\pi^s:M^\mathrm{in}\rightarrow D^2$, called the \textit{stabilization} of $\pi$, as follows:
\begin{itemize}
	\item replace $F^\mathrm{in}$ with the Riemann surface $(F')^\mathrm{in}$, which is $F^\mathrm{in}$ with a 1-handle attached along the endpoints of $\gamma$, so that the exact Lagrangian submanifold with boundary $\gamma\subset F^\mathrm{in}$ becomes a closed curve $\gamma'\subset (F')^\mathrm{in}$;
	\item add a new critical point to the base $D^2$ of the Lefschetz fibration $\pi$ corresponding to the new vanishing cycle $\gamma'$.
\end{itemize}
Applying the stabilization construction reversely to $\pi_T$ results in a Lefschetz fibration $\pi_T^{-s}:T_{p,q,r}\rightarrow\mathbb{C}$, whose smooth fiber $F^{-s}$ is symplectomorphic to a thrice-punctured torus, which can be regarded as a plumbing of four copies of $T^\ast S^1$ according to a $D_4$-tree. Denote the zero sections of these cotangent bundles by $P,Q,R$ and $T$ respectively, see Figure \ref{fig:3-punctured}, they form a Lagrangian skeleton of the Liouville domain $F^{-s}$.
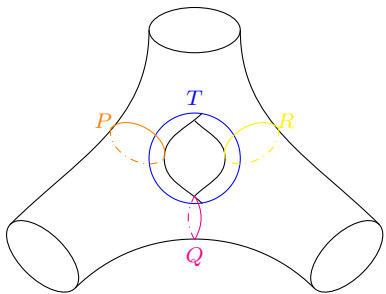
\begin{figure}
	\centering
	\begin{tikzpicture}
	\draw (0,0) ellipse (0.6cm and 0.3cm);
	\draw [rotate around={-45:(-2,-3)}] (-2,-3) ellipse (0.6cm and 0.3cm);
	\draw [rotate around={45:(2,-3)}] (2,-3) ellipse (0.6cm and 0.3cm);
	\draw (-0.6,0) to[in=45,out=-90] (-2.43,-2.58);
	\draw (0.6,0) to[in=135,out=-90] (2.43,-2.58);
	\draw (-1.57,-3.42) to [in=135,out=45] (1.57,-3.42);
	\draw (0.1,-1.1) to [in=90,out=-135] (-0.4,-1.7);
	\draw (-0.4,-1.7) to [in=135,out=-90] (0.1,-2.3);
	\draw (0,-1.2) to [in=90,out=-45] (0.4,-1.7);
	\draw (0.4,-1.7) to [in=45,out=-90] (0,-2.2);
	\draw [blue] (0,-1.7) circle (0.6);
	\node [blue] at (0,-0.9) {\footnotesize $T$};
	\draw [orange] (-0.4,-1.7) to [in=45, out=75] (-1.1,-1.3);
	\draw [orange,dash dot] (-1.1,-1.3) to [in=-135,out=-105] (-0.4,-1.7);
	\node [orange] at (-1.2,-1.2) {\footnotesize $P$};
	\draw [yellow] (0.4,-1.7) to [in=135, out=105] (1.1,-1.3);
	\draw [yellow,dash dot] (1.1,-1.3) to [in=-45,out=-75] (0.4,-1.7);
	\node [yellow] at (1.2,-1.2) {\footnotesize $R$};
	\draw [magenta] (0,-2.2) [in=60,out=-60] to (0,-2.78);
	\draw [magenta,dash dot] (0,-2.78) [in=-120,out=120] to (0,-2.2);
	\node [magenta] at (0,-3) {\footnotesize $Q$};
	\end{tikzpicture}
	\caption{Lagrangian skeleton of the fiber of $\pi_T^{-s}$}
	\label{fig:3-punctured}
\end{figure}
\bigskip

Finally, the Milnor fiber $M_{p,q,r}$ is defined by adding a quadratic term $w^2$ to the defining equation of $T_{p,q,r}$, so obtaining a Lefschetz fibration on $M_{p,q,r}$ from the Lefschetz fibration $\pi_T^{-s}$ on $T_{p,q,r}$ can be realized as the case $r=2$ of the situation described in the second step above. In fact, one can check that $\pi_T+w^2$ defines a Lefschetz fibration on $M_{p,q,r}$, which we denote as
\begin{equation}
\pi_{p,q,r}:M_{p,q,r}\rightarrow\mathbb{C}.
\end{equation}
It is clear from its definition that a smooth fiber of $\pi_{p,q,r}$ is symplectomorphic to 4-dimensional $D_4$ Milnor fiber. By abuse of notations, the compact cores of these $D_4$-plumbings of $T^\ast S^2$'s will still be denoted by $P,Q,R$ and $T$.
\bigskip

After some mutations, the vanishing cycles of $\pi_{p,q,r}$ are given by
\begin{equation}
T,\tau_P^2\tau_Q^2\tau_R^2(T),\tau_P\tau_Q\tau_R(T)
\end{equation}
together with $p$ copies of $P$, $q$ copies of $Q$ and $r$ copies of $R$, where $\tau_V$ denotes the Dehn twist along the vanishing cycle $V$.

\subsection{Casals-Murphy recipe}

We recall how to obtain the Legendrian front description of a Weinstein manifold $M$ starting from a symplectic Lefschetz fibration. This is written down systematically by Casals-Murphy in $\cite{cm}$.
\bigskip

Let $\pi:M\rightarrow\mathbb{C}$ be a Lefschetz fibration with smooth fiber $F_T$ which is a plumbing of $T^\ast S^{n-1}$'s according to some tree $T$. Given these data, Casals and Murphy suggests in $\cite{cm}$ the following procedure to obtain a Legendrian handle body decomposition of $M$.
\begin{itemize}
\item Draw $r$ $(n-1)$-handles which correspond to the zero sections $L_1,\dots,L_r$ of $T^\ast S^{n-1}$ in the plumbing $F_T$.
\item Find a Lefschetz fibration $\pi_F:F_T\rightarrow\mathbb{C}$ so that the Lagrangian spheres $\{L_i\}$ appear as matching cycles of $\pi_F$ with matching paths $\gamma_1,\dots,\gamma_r\subset\mathbb{C}$.
\item For any vanishing cycle $V_j\subset F_T$ of $\pi$, draw the embedded path $\beta_j\subset\mathbb{C}$ under the projection of $\pi_F$.
\item Express each matching path $\beta_j$ of $V_j$ as a word in half-twists along the paths in $\{\gamma_i\}$. The vanishing cycles $\{V_j\}$ are thus expressed in terms of words in Dehn twists along the Lagrangian spheres in $\{L_i\}$.
\item Using handle slides, one is able to draw the front projections of their Legendrian lifts $\{\Lambda_j\}$ in the contact boundary $\partial(F_T\times D^2)$.
\item The above step produces a Legendrian link $\Lambda=\bigcup_j\Lambda_j\subset\partial(F_T\times D^2)$ going through the $(n-1)$-handles. We then push each component $\Lambda_j$ of $\Lambda$ in the Reeb direction of $\partial(F_T\times D^2)$ by height $i$.
\item Simplify the Legendrian front projection of $\Lambda$ using Reidemeister moves and handle cancellations.
\end{itemize}

Casals-Murphy recipe is extremely useful in obtaining Legendrian frontal descriptions of Weinstein manifolds $M$ with $\dim_\mathbb{R}(M)\geq6$, since the existence of a Lefschetz fibration on $M$ is proved by Giroux-Pardon in $\cite{gp}$. In the special case when $M$ is obtained by stabilizing a 4-dimensional Milnor fiber $T$ by adding quadratic terms to its defining equation, the Legendrian surgery picture of $M$ is realized locally as an $S^{n-2}$ spin of that of $T$.

\paragraph{Remark} By definition, there is an obvious Lefschetz fibration $\pi:M_{p,q,r}\rightarrow\mathbb{C}$ given by projecting to the $w$ coordinate plane, whose fiber is symplectomorphic to $T_{p,q,r}$. However, $T_{p,q,r}$ is not a plumbing of $T^\ast S^2$'s. In fact, it is proved by Keating in $\cite{ak1}$ that the compact Fukaya category $\mathcal{F}(T_{p,q,r})$ is not split-generated by vanishing cycles over any field $\mathbb{K}$ with $\mathrm{char}(\mathbb{K})\neq2$. This explains why we choose to apply Casals-Murphy recipe to the destabilized Lefschetz fibration $\pi_{p,q,r}$ in Section \ref{section:surgery} to get the Legendrian front associated to $M_{p,q,r}$.

\subsection{Legendrian surgery presentation of $M_{p,q,r}$}\label{section:surgery}

We now apply Casals-Murphy recipe to the Lefschetz fibration $\pi_{p,q,r}:M_{p,q,r}\rightarrow\mathbb{C}$ described in Section \ref{section:Lefschetz 6-fold}. Label the vanishing cycles of $\pi_{p,q,r}$ by
\begin{equation}
\begin{split}
&V_{-1}=\tau_P^2\tau_Q^2\tau_R^2(T),V_0=\tau_P\tau_Q\tau_R(T),V_1=\dots=V_p=P,\\
&V_{p+1}=\dots=V_{p+q}=Q,V_{p+q+1}=\dots=V_{p+q+r}=R,V_{p+q+r+1}=T.
\end{split}
\end{equation}
This enables us to draw the Legendrian frontal presentation of $M_{p,q,r}$ based on the data given by the Lefschetz fibration $\pi_{p,q,r}$. The picture consists of four 2-handles labelled by $P,Q,R$ and $T$ and $(p+q+r+3)$ 3-handles corresponding to the vanishing cycles $V_{-1},\dots,V_{p+q+r+1}$. The two non-trivial 3-handles corresponding to $V_{-1}$ and $V_0$ are depicted in Figures \ref{fig:A} and \ref{fig:B}, where the thick dots in the Figure represent cone singularities. As Legendrian surfaces, they are denoted respectively by $\Lambda_A$ and $\Lambda_B$. All the other Legendrian attaching spheres $\Lambda_{P_i}$, $\Lambda_{Q_j}$, $\Lambda_{P_k}$ and $\Lambda_T$ are just parallel strands which go through a single handle $P,Q,R$ and $T$ respectively. Their interactions with the Legendrian attaching spheres $\Lambda_A$ and $\Lambda_B$ are illustrated in the Figures \ref{fig:A} and \ref{fig:B}, where each set of Legendrian spheres $\{\Lambda_{P_i}\}$, $\{\Lambda_{Q_j}\}$, $\{\Lambda_{R_k}\}$ is represented by a single Legendrian sphere $\Lambda_P$, $\Lambda_Q$ and $\Lambda_R$.
\bigskip

\begin{figure}[h!]
	\centering
	\begin{tikzpicture}[scale=1,auto=left,every node/.style={circle}]
	\tikzset{->-/.style={decoration={ markings, mark=at position #1 with {\arrow{>}}},postaction={decorate}}}
	
	\draw (-2,3) circle (0.5);
	\draw (8,3) circle (0.5);
	\draw (-2,0) circle (0.5);
	\draw (0,0) circle (0.5);
	\draw (2,0) circle (0.5);
	\draw (4,0) circle (0.5);
	\draw (6,0) circle (0.5);
	\draw (8,0) circle (0.5);

	\draw [green] (-1.6,0.3) to[in=180,out=45] (1,2);
	\draw [green] (1,2) to[in=135,out=0] (3.6,0.3);
	\draw [green] (2.4,0.3) to[in=180,out=45] (5,2);
	\draw [green] (5,2) to[in=135,out=0] (7.6,0.3);
	\draw [green] (-0.4,0.3) to[in=-45,out=135] (-1.6,2.7);
	\draw [green] (6.4,0.3) to[in=225,out=45] (7.6,2.7);
	
	\draw [green] (-0.87,1.07) node[circle,fill,inner sep=1pt]{};
	\draw [green] (6.87,1.07) node[circle,fill,inner sep=1pt]{};
	\draw [green] (3,0.93) node[circle,fill,inner sep=1pt]{};
	
	\draw [orange] (-1.7,0.4) to [in=180,out=45] (-1,1.3);
	\draw [orange] (-1,1.3) to [in=135,out=0] (-0.3,0.4);
	
	\draw [magenta] (2.3,0.4) to [in=180,out=45] (3,1.3);
	\draw [magenta] (3,1.3) to [in=135,out=0] (3.7,0.4);
	
	\draw [yellow] (6.3,0.4) to [in=180,out=45] (7,1.3);
	\draw [yellow] (7,1.3) to [in=135,out=0] (7.7,0.4);
	
	\node [green] at (1,2.2) {\footnotesize $\Lambda_A$};
	\node [orange] at (-1.5,1.2) {\footnotesize $\Lambda_P$};
	\node [magenta] at (3,1.5) {\footnotesize $\Lambda_Q$};
	\node [yellow] at (7.5,1.2) {\footnotesize $\Lambda_R$};
	
	\node at (-2,0) {\footnotesize $P$};
	\node at (0,0) {\footnotesize $P$};
	\node at (2,0) {\footnotesize $Q$};
	\node at (4,0) {\footnotesize $Q$};
	\node at (6,0) {\footnotesize $R$};
	\node at (8,0) {\footnotesize $R$};
	\node at (-2,3) {\footnotesize $T$};
	\node at (8,3) {\footnotesize $T$};
	
	\end{tikzpicture}
	\caption{Front projection of the components $\Lambda_A$, $\Lambda_P$, $\Lambda_Q$ and $\Lambda_R$}
	\label{fig:A}
\end{figure}
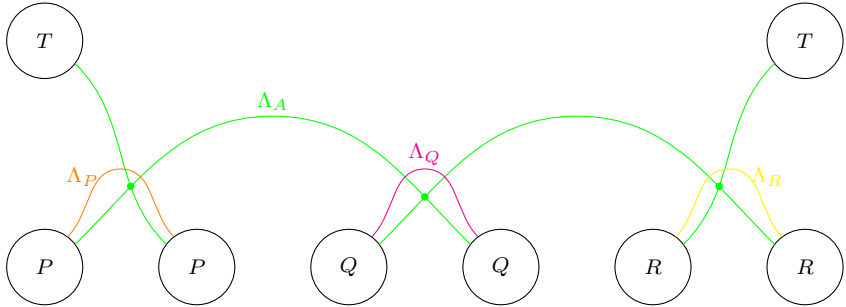

\begin{figure}[h!]
	\centering
	\begin{tikzpicture}[scale=1,auto=left,every node/.style={circle}]
	\tikzset{->-/.style={decoration={ markings, mark=at position #1 with {\arrow{>}}},postaction={decorate}}}
	
	\draw (-2,3) circle (0.4);
	\draw (8.1,3) circle (0.4);
	\draw (-2,0) circle (0.4);
	\draw (0.7,0) circle (0.4);
	\draw (1.7,0) circle (0.4);
	\draw (4.4,0) circle (0.4);
	\draw (5.4,0) circle (0.4);
	\draw (8.1,0) circle (0.4);
	
	\draw [violet] (0.4,0.3) to[in=0,out=180] (-0.1,0);
	\draw [violet] (0.4,-0.3) to[in=0,out=180] (-0.1,0);
	\draw [violet] (4.1,0.3) to[in=0,out=180] (3.6,0);
	\draw [violet] (4.1,-0.3) to[in=0,out=180] (3.6,0);
	\draw [violet] (7.8,0.3) to[in=0,out=180] (7.3,0);
	\draw [violet] (7.8,-0.3) to[in=0,out=180] (7.3,0);
	\draw [violet] (-1.7,2.7) to[in=180,out=0] (-0.8,-0.3);
	\draw [violet] (-0.8,-0.3) to[in=180,out=0] (-0.4,0);
	\draw [violet] (-0.4,0) to[in=0,out=180] (-1.05,0.5);
	\draw [violet] (-1.05,0.5) to[in=0,out=180] (-1.7,0.3);
	\draw [violet] (-1.7,-0.3) to[in=180,out=0] (1.2,2);
	\draw [violet] (1.2,2) to[in=180,out=0] (2.9,-0.3);
	\draw [violet] (2.9,-0.3) to[in=180,out=0] (3.3,0);
	\draw [violet] (3.3,0) to[in=0,out=180] (2.65,0.5);
	\draw [violet] (2.65,0.5) to[in=0,out=180] (2,0.3);
	\draw [violet] (2,-0.3) to[in=180,out=0] (4.9,2);
	\draw [violet] (4.9,2) to[in=180,out=0] (6.6,-0.3);
	\draw [violet] (6.6,-0.3) to[in=180,out=0] (7,0);
	\draw [violet] (7,0) to[in=0,out=180] (6.35,0.5);
	\draw [violet] (6.35,0.5) to[in=0,out=180] (5.7,0.3);
	\draw [violet] (5.7,-0.3) to[in=180,out=0] (7.8,2.7);
	
	\draw [violet] (-1.15,-0.17) node[circle,fill,inner sep=1pt]{};
	\draw [violet] (2.53,-0.18) node[circle,fill,inner sep=1pt]{};
	\draw [violet] (6.18,-0.14) node[circle,fill,inner sep=1pt]{};
	
	\draw [orange] (-1.85,0.4) to [in=180,out=45] (-0.65,1);
	\draw [orange] (-0.65,1) to [in=135,out=0] (0.55,0.4);
	
	\draw [magenta] (1.85,0.4) to [in=180,out=45] (3.05,1);
	\draw [magenta] (3.05,1) to [in=135,out=0] (4.25,0.4);
	
	\draw [yellow] (5.55,0.4) to [in=180,out=45] (6.75,1);
	\draw [yellow] (6.75,1) to [in=135,out=0] (7.95,0.4);
	
	\node [violet] at (1.2,2.2) {\footnotesize $\Lambda_B$};
	\node [orange] at (-0.65,1.2) {\footnotesize $\Lambda_P$};
	\node [magenta] at (3.05,1.2) {\footnotesize $\Lambda_Q$};
	\node [yellow] at (6.75,1.2) {\footnotesize $\Lambda_R$};
	
	\node at (-2,0) {\footnotesize $P$};
	\node at (0.7,0) {\footnotesize $P$};
	\node at (1.7,0) {\footnotesize $Q$};
	\node at (4.4,0) {\footnotesize $Q$};
	\node at (5.4,0) {\footnotesize $R$};
	\node at (8.1,0) {\footnotesize $R$};
	\node at (-2,3) {\footnotesize $T$};
	\node at (8.1,3) {\footnotesize $T$};
	
	\end{tikzpicture}
	\caption{Front projection of the components $\Lambda_B$, $\Lambda_P$, $\Lambda_Q$ and $\Lambda_R$}
	\label{fig:B}
\end{figure}
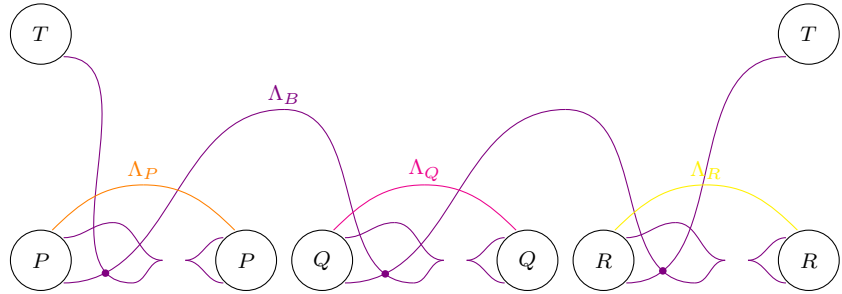

\begin{proposition}\label{proposition:link}
The Weinstein 6-manifold $M_{p,q,r}$ is obtained by attaching Weinstein 3-handles to $D^6$ along the link of 2-dimensional Legendrian unknots
\begin{equation}
\Lambda_{p,q,r}=\Lambda_A\cup\Lambda_B\cup\bigcup_{i=1}^{p-1}\Lambda_{P_i}\cup\bigcup_{j=1}^{q-1}\Lambda_{Q_j}\cup\bigcup_{k=1}^{r-1}\Lambda_{R_k}.
\end{equation}
In particular, when $p=q=r=2$, $\Lambda_{2,2,2}$ is Legendrian isotopic to a link of Legendrian surfaces whose front projection is depicted in Figure \ref{fig:front-without-2-handles}. In general, one can obtain the Legendrian front by replacing the component $\Lambda_P$ (resp. $\Lambda_Q$ and $\Lambda_R$) in Figure \ref{fig:front-without-2-handles} by an $A_{p-1}$ (resp. $A_{q-1}$ and $A_{r-1}$) chain of standard unknots which are parallel to each other. Moreover, $\Lambda_{P_1}$, $\Lambda_{Q_1}$ and $\Lambda_{R_1}$ are the only Legendrian spheres in the sets $\{\Lambda_{P_i}\}$, $\{\Lambda_{Q_j}\}$ and $\{\Lambda_{R_k}\}$ whose fronts have non-trivial intersections with the fronts of $\Lambda_A$ and $\Lambda_B$.
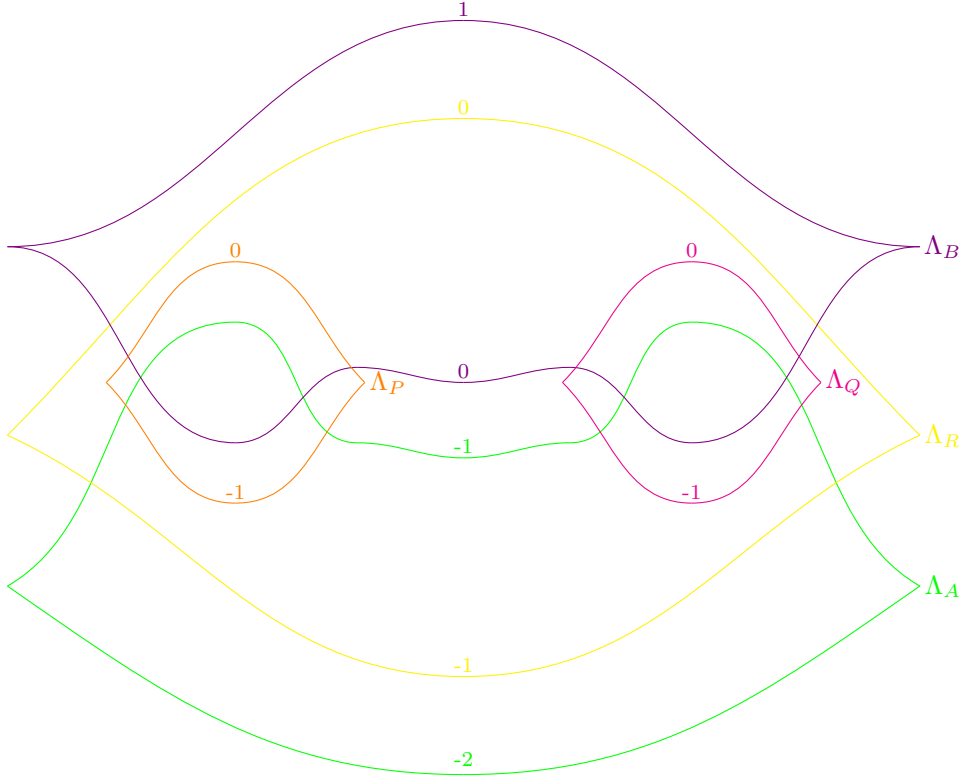
\begin{figure}
	\centering
	\begin{tikzpicture}[scale=1,auto=left,every node/.style={circle}]
	
	\draw [green] (-2,-4.5) to [in=180,out=30] (1,-1);
	\draw [green] (1,-1) to [in=180,out=0] (2.6,-2.6);
	\draw [green] (2.6,-2.6) to [in=180,out=0] (4,-2.8);
	\draw [green] (4,-2.8) to [in=180,out=0] (5.4,-2.6);
	\draw [green] (5.4,-2.6) to [in=180,out=0] (7,-1);
	\draw [green] (7,-1) to [in=150,out=0] (10,-4.5);
	\draw [green] (-2,-4.5) to [in=180,out=-35] (4,-7);
	\draw [green] (4,-7) to [in=-145,out=0] (10,-4.5);
	\node [green] at (4,-6.8) {\footnotesize -2};
	\node [green] at (4,-2.65) {\footnotesize -1};
	\node [green] at (10.3,-4.5) {$\Lambda_A$};
	
	\draw [yellow] (-2,-2.5) to [in=180,out=45] (4,1.7);
	\draw [yellow] (4,1.7) to [in=135,out=0] (10,-2.5);
	\draw [yellow] (-2,-2.5) to [in=180,out=-25] (4,-5.7);
	\draw [yellow] (4,-5.7) to [in=-155,out=0] (10,-2.5);
	\node [yellow] at (4,1.85) {\footnotesize 0};
	\node [yellow] at (4,-5.55) {\footnotesize -1};
	\node [yellow] at (10.3,-2.5) {$\Lambda_R$};
	
	\draw [violet] (-2,0) to [in=180,out=0] (1,-2.6);
	\draw [violet] (1,-2.6) to [in=180,out=0] (2.6,-1.6);
	\draw [violet] (2.6,-1.6) to [in=180,out=0] (4,-1.8);
	\draw [violet] (4,-1.8) to [in=180,out=0] (5.4,-1.6);
	\draw [violet] (5.4,-1.6) to [in=180,out=0] (7,-2.6);
	\draw [violet] (7,-2.6) to [in=180,out=0] (10,0);
	\draw [violet] (-2,0) to [in=180,out=0] (4,3);
	\draw [violet] (4,3) to [in=180,out=0] (10,0);
	\node [violet] at (4,3.15) {\footnotesize 1};
	\node [violet] at (4,-1.65) {\footnotesize 0};
	\node [violet] at (10.3,0) {$\Lambda_B$};
	
	\draw [orange] (-0.7,-1.8) to [in=180,out=45] (1,-0.2);
	\draw [orange] (1,-0.2) to [in=135,out=0] (2.7,-1.8);
	\draw [orange] (-0.7,-1.8) to [in=180,out=-45] (1,-3.4);
	\draw [orange] (1,-3.4) to [in=-135,out=0] (2.7,-1.8);
	\node [orange] at (1,-0.05) {\footnotesize 0};
	\node [orange] at (1,-3.25) {\footnotesize -1};
	\node [orange] at (3,-1.8) {$\Lambda_P$};
	
	\draw [magenta] (5.3,-1.8) to [in=180,out=45] (7,-0.2);
	\draw [magenta] (7,-0.2) to [in=135,out=0] (8.7,-1.8);
	\draw [magenta] (5.3,-1.8) to [in=180,out=-45] (7,-3.4);
	\draw [magenta] (7,-3.4) to [in=-135,out=0] (8.7,-1.8);
	\node [magenta] at (7,-0.05) {\footnotesize 0};
	\node [magenta] at (7,-3.25) {\footnotesize -1};
	\node [magenta] at (9,-1.8) {$\Lambda_Q$};
	
	\end{tikzpicture}
	\caption{Legendrian front of $\Lambda_{2,2,2}$ after handle cancellation, where the numbers above the sheets are values taken by a Maslov potential $\mu_{2,2,2}:\Lambda_{2,2,2}\rightarrow\mathbb{Z}$}
	\label{fig:front-without-2-handles}
\end{figure}
\end{proposition}
\begin{proof}
We apply handle cancellation to the Legendrian surgery diagram associated to the Lefschetz fibration $\pi_{p,q,r}:M_{p,q,r}\rightarrow\mathbb{C}$. The vanishing cycles
\begin{equation}
V_1,V_{p+1},V_{p+q+1},V_{p+q+r+1}
\end{equation}
of $\pi_{p,q,r}$ are all $(-1)$-Legendrian spheres in the vertical boundary $\partial^\mathrm{v} M_{p,q,r}^\mathrm{in}$ which intersect the belt spheres of the Weinstein 2-handles labelled by $P,Q,R$ and $T$, where by $(-1)$-Legendrian spheres we mean the Legendrian spheres in the boundary of the pre-surgery Weinstein manifold $F^\mathrm{in}_{D_4}\times D^2$ along which the Weinstein 3-handles are attached. In particular, each of them can be cancelled with the corresponding subcritical handle, so that we obtain a Legendrian surgery diagram without 2-handles.

For the component $\Lambda_B$ as pictured in Figure \ref{fig:B}, after cancelling all the 2-handles with $V_1$, $V_{p+1}$, $V_{p+q+1}$ and $V_{p+q+r+1}$, we get a Legendrian knot as pictured in the upper half of Figure \ref{fig:B1}. After a sequence of Reidemeister I and Reidemeister II moves, one can simplify it to a Legendrian 2-sphere with three cone singularities, as is shown in the lower left of Figure \ref{fig:B1}. This Legendrian 2-sphere can be seen to be Legendrian isotopic to the standard unknot by applying the Legendrian isotopy pictured in Figure \ref{fig:R1}. Note that the move in Figure \ref{fig:R1} is a Legendrian isotopy since it is the $S^1$-rotation of the Reidemeister I move with respect to the axis which passes through the crossing point.

Similarly, the Legendrian knot $\Lambda_A$ can be seen to be Hamiltonian isotopic to the standard unknot after cancelling all the 2-handles. The fact that after handle cancellation, the components $\Lambda_{P_i}$, $\Lambda_{Q_j}$ and $\Lambda_{R_k}$ are isotopic to the standard unknots is obvious.

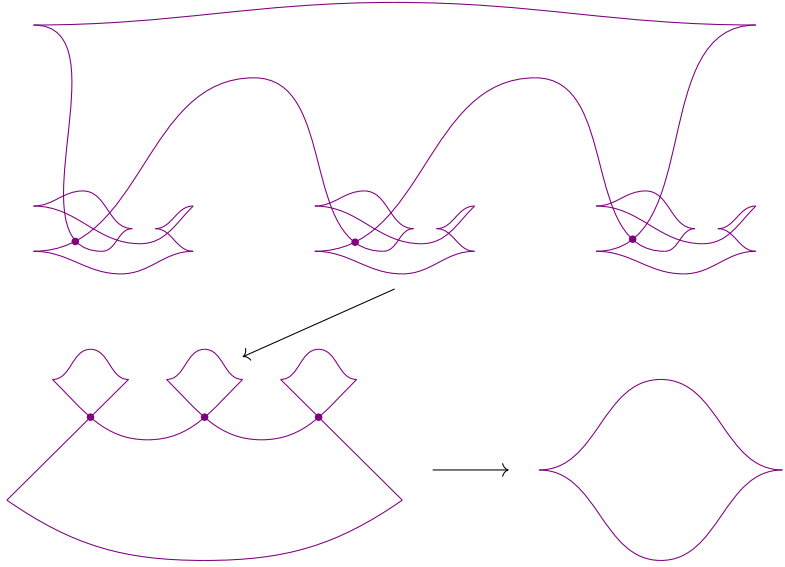
\begin{figure}
	\centering
	\begin{tikzpicture}[scale=1,auto=left,every node/.style={circle}]
	\draw [violet] (1.85,2.3) to[in=0,out=180] (1.35,2);
	\draw [violet] (1.85,1.7) to[in=0,out=180] (1.35,2);
	\draw [violet] (5.55,2.3) to[in=0,out=180] (5.05,2);
	\draw [violet] (5.55,1.7) to[in=0,out=180] (5.05,2);
	\draw [violet] (9.25,2.3) to[in=0,out=180] (8.75,2);
	\draw [violet] (9.25,1.7) to[in=0,out=180] (8.75,2);
	\draw [violet] (-0.25,4.7) to[in=180,out=0] (0.65,1.7);
	\draw [violet] (0.65,1.7) to[in=180,out=0] (1.05,2);
	\draw [violet] (1.05,2) to[in=0,out=180] (0.4,2.5);
	\draw [violet] (0.4,2.5) to[in=0,out=180] (-0.25,2.3);
	\draw [violet] (-0.25,1.7) to[in=180,out=0] (2.65,4);
	\draw [violet] (2.65,4) to[in=180,out=0] (4.35,1.7);
	\draw [violet] (4.35,1.7) to[in=180,out=0] (4.75,2);
	\draw [violet] (4.75,2) to[in=0,out=180] (4.1,2.5);
	\draw [violet] (4.1,2.5) to[in=0,out=180] (3.45,2.3);
	\draw [violet] (3.45,1.7) to[in=180,out=0] (6.35,4);
	\draw [violet] (6.35,4) to[in=180,out=0] (8.05,1.7);
	\draw [violet] (8.05,1.7) to[in=180,out=0] (8.45,2);
	\draw [violet] (8.45,2) to[in=0,out=180] (7.8,2.5);
	\draw [violet] (7.8,2.5) to[in=0,out=180] (7.15,2.3);
	\draw [violet] (7.15,1.7) to[in=180,out=0] (9.25,4.7);
	\draw [violet] (-0.25,4.7) to [in=180,out=0] (4.5,5);
	\draw [violet] (4.5,5) to [in=180,out=0] (9.25,4.7);
	\draw [violet] (-0.25,1.7) to [in=180,out=0] (0.9,1.4);
	\draw [violet] (0.9,1.4) to [in=180,out=0] (1.85,1.7);
	\draw [violet] (3.45,1.7) to [in=180,out=0] (4.6,1.4);
	\draw [violet] (4.6,1.4) to [in=180,out=0] (5.55,1.7);
	\draw [violet] (7.15,1.7) to [in=180,out=0] (8.3,1.4);
	\draw [violet] (8.3,1.4) to [in=180,out=0] (9.25,1.7);
	\draw [violet] (-0.25,2.3) to [in=180,out=0] (1.15,1.8);
	\draw [violet] (1.15,1.8) to [in=-135,out=0] (1.85,2.3);
	\draw [violet] (3.45,2.3) to [in=180,out=0] (4.85,1.8);
	\draw [violet] (4.85,1.8) to [in=-135,out=0] (5.55,2.3);
	\draw [violet] (7.15,2.3) to [in=180,out=0] (8.55,1.8);
	\draw [violet] (8.55,1.8) to [in=-135,out=0] (9.25,2.3);
	
	\draw [violet] (0.3,1.83) node[circle,fill,inner sep=1pt]{};
	\draw [violet] (3.98,1.82) node[circle,fill,inner sep=1pt]{};
	\draw [violet] (7.63,1.86) node[circle,fill,inner sep=1pt]{};
	
	\draw [violet] (0,0) to [in=180,out=0] (0.5,0.4);
	\draw [violet] (0.5,0.4) to [in=180,out=0] (1,0);
	\draw [violet] (1,0) to (-0.6,-1.6);
	\draw [violet] (0,0) to [in=180,out=-45] (1.25,-0.8);
	\draw [violet] (1.25,-0.8) to [in=-135,out=0] (2.5,0);
	\draw [violet] (1.5,0) to [in=180,out=0] (2,0.4);
	\draw [violet] (2,0.4) to [in=180,out=0] (2.5,0);
	\draw [violet] (1.5,0) to [in=180,out=-45] (2.75,-0.8);
	\draw [violet] (2.75,-0.8) to [in=-135,out=0] (4,0);
	\draw [violet] (3,0) to [in=180,out=0] (3.5,0.4);
	\draw [violet] (3.5,0.4) to [in=180,out=0] (4,0);
	\draw [violet] (3,0) to (4.6,-1.6);
	\draw [violet] (-0.6,-1.6) to [in=180,out=-35] (2,-2.4);
	\draw [violet] (2,-2.4) to [in=-145,out=0] (4.6,-1.6);
	
	\draw [violet] (0.5,-0.5) node[circle,fill,inner sep=1pt]{};
	\draw [violet] (2,-0.5) node[circle,fill,inner sep=1pt]{};
	\draw [violet] (3.5,-0.5) node[circle,fill,inner sep=1pt]{};
	
	\draw [->] (5,-1.2) to (6,-1.2);
	\draw [->] (4.5,1.2) to (2.5,0.3);
	\draw [violet] (6.4,-1.2) to [in=180,out=0] (8,0);
	\draw [violet] (8,0) to [in=180,out=0] (9.6,-1.2);
	\draw [violet] (6.4,-1.2) to [in=180,out=0] (8,-2.4);
	\draw [violet] (8,-2.4) to [in=180,out=0] (9.6,-1.2);
	\end{tikzpicture}
	\caption{Front projection of $\Lambda_B$ after handle cancellation, and its simplifications}
	\label{fig:B1}
\end{figure}

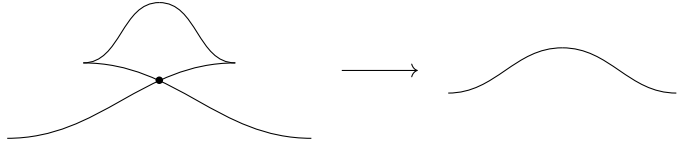
\begin{figure}
	\centering
	\begin{tikzpicture}[scale=1,auto=left,every node/.style={circle}]
    \draw (2.6,0) to [in=180,out=0] (3.6,0.8);
    \draw (3.6,0.8) to [in=180,out=0] (4.6,0);
    \draw (4.6,0) to [in=0,out=180] (1.6,-1);
    \draw (2.6,0) to [in=180,out=0] (5.6,-1);
    \draw [->] (6,-0.1) to (7,-0.1);
    \draw (7.4,-0.4) to [in=180,out=0] (8.9,0.2);
    \draw (8.9,0.2) to [in=180,out=0] (10.4,-0.4);
    \draw  (3.6,-0.23) node[circle,fill,inner sep=1pt]{};
	\end{tikzpicture}
	\caption{Symmetric rotation of the Reidemeister I move}
	\label{fig:R1}
\end{figure}

One dimensional lower, the Legendrian front of the affine surface $T_{p,q,r}$ is a link of unknots $K_{p,q,r}\subset(\mathbb{R}^3,\xi_\mathit{std})$. The linking numbers between the components
\begin{equation}
K_A,K_B,K_{P_i},K_{Q_j},K_{R_k}\subset K_{p,q,r}
\end{equation}
are determined by the intersection numbers of their Lagrangian fillings, which, after capping off using the core discs of the 2-handles attached along $K_{p,q,r}$, becomes a basis of vanishing cycles of the Milnor fiber $T_{p,q,r}$. For example, the Lagrangian fillings of $K_A$ and $K_B$ are Lagrangian discs intersecting transversely at two points, which implies that the corresponding Lagrangian cocore discs $L_A$ and $L_B$ of the 2-handles attached to $K_A$ and $K_B$ also have transversal intersections at two points. Since each of the intersection points in $L_A\cap L_B$ gives rise to two Reeb chords, one from $K_A$ to $K_B$, and the other one from $K_B$ to $K_A$, we see that the front projections of $K_A$ and $K_B$ intersect at 4 points. Since the Legendrian front of $\Lambda_{p,q,r}$ is given locally by an $S^1$-symmetric rotation of $K_{p,q,r}$, the intersections of the different components of $\Lambda_{p,q,r}$ under the front projection is determined by that of $K_{p,q,r}$. In particular, the fronts of $\Lambda_A$ and $\Lambda_B$ intersect along two circles.
\end{proof}

\section{Cellular dg algebra}\label{section:cellular}

This section is an exposition of the paper $\cite{rs1}$ by Rutherford-Sullivan. We will be mainly focusing on the definitions and results that are relevant to our computations in Section \ref{section:combi}. In this section, we fix the coefficient field $\mathbb{K}$ to be $\mathbb{Z}/2$.

\subsection{The definition}\label{section:def-cell}

We recall here the definition of the cellular dg algebra associated to a closed Legendrian surface $\Lambda\subset(S^5,\xi_\mathit{std})$, with the additional assumption that there is no swallowtail singularity in the front projection of $\Lambda$.

Let $S$ be a surface. A \textit{polygonal decomposition} of $S$ is a decomposition of $S$ into CW complexes
\begin{equation}S=\bigsqcup_{i=0}^2\bigsqcup_\alpha e_\alpha^i\end{equation}
equipped with characteristic maps $c_\alpha^i:D^i\rightarrow S$ which satisfy the following two properties:
\begin{itemize}
\item $c_\alpha^1$ are smooth for all $\alpha$.
\item For any 2-cell $e_\alpha^2$, pre-images of 0-cells divide $\partial D^2$ into intervals that are mapped homeomorphically to 1-cells by $c_\alpha^2$.
\end{itemize}

Denote by $J^1(S):=T^\ast S\times\mathbb{R}$ the 1-jet space. Let $\Lambda\subset J^1(S)$ be a Legendrian surface. Then there are two natural projections, namely the front projection
\begin{equation}p_{x,z}:J^1(S)\rightarrow S\times\mathbb{R}\end{equation}
and the base projection
\begin{equation}p_x:J^1(S)\rightarrow S,\end{equation}
where $x=(x_1,x_2)$ denotes the local coordinates on $S$ and $z$ is coordinate in the Reeb direction $\mathbb{R}$.

Suppose that $\Lambda$ has generic front projection, denote by $\Sigma\subset S$ the image of the singular set of $\Lambda$ under the base projection $p_x$, then it decomposes as $\Sigma=\Sigma_1\sqcup\Sigma_2$, where $\Sigma_i$ denotes the base projection of the set of codimension $i$ singularities. A polygonal decomposition of $S$ is $\Lambda$-\textit{compatible} if $\Sigma$ is contained in the 1-skeleton $\bigcup_\alpha e_\alpha^1$ of the polygonal decomposition. It then follows that $\Sigma_2$ is contained in the 0-skeleton $\bigsqcup_\alpha e_\alpha^0$.\bigskip

We now proceed to define the cellular dg algebra
\begin{equation}(\mathcal{C}(\Lambda),d_\mathcal{C})\end{equation}
associated to the Legendrian surface $\Lambda\subset J^1(S)$.
\bigskip

We first describe the set of generators of $\mathcal{C}(\Lambda)$. Given a $\Lambda$-compatible polygonal decomposition of $S$, which has $e_\alpha^i$ as one of its cells, the connected components $\Lambda\cap p_x^{-1}(e_\alpha^i)$ which are not contained in any cusp edge will be referred to as \textit{sheets} of $\Lambda$ above $e_\alpha^i$. Denote by $\Lambda(e_\alpha^i)$ the set of sheets of $\Lambda$ above $e_\alpha^i$, using their $z$-coordinates we can equip $\Lambda(e_\alpha^i)$ with a partial ordering $\prec$. More precisely, for $S_1,S_2\in\Lambda(e_\alpha^i)$, $S_1\prec S_2$ if $z(S_1)>z(S_2)$ above $e_\alpha^i$. Note that two sheets are incomparable if and only if they meet in a crossing arc above $e_\alpha^i$ in the front projection of $\Lambda$.

For each cell $e_\alpha^i$ in the $\Lambda$-compatible polygonal decomposition, we associate one generator of $\mathcal{C}(\Lambda)$ for each pair of sheets $S_m,S_n\in\Lambda(e_\alpha^i)$ with $S_m\prec S_n$. These generators will be denoted by $a^{m,n}_\alpha$, $b^{m,n}_\alpha$ and $c_\alpha^{m,n}$ when the corresponding cells are 0-dimensional, 1-dimensional and 2-dimensional respectively. As a graded algebra, $\mathcal{C}(\Lambda)$ is freely generated by these generators. 

We shall be interested in the case when $\Lambda=\bigsqcup_v\Lambda_v$ is a link of Legendrian surfaces.  In this case, $\mathcal{C}(\Lambda)$ carries the structure of a $\Bbbk$-bimodule, with $\Bbbk=\bigoplus_v\mathbb{K}e_v$, which we describe as follows. First, consider the vector spaces $\mathbb{K}\langle a^{m,n}_\alpha,b^{m,n}_\alpha,c^{m,n}_\alpha\rangle$ spanned by the generators in $A_\alpha,B_\alpha$ and $C_\alpha$, we can endow it with a $\Bbbk$-bimodule structure by declaring
\begin{equation}\label{eq:ss}
e_w\{a^{m,n}_\alpha,b^{m,n}_\alpha,c^{m,n}_\alpha\}e_v
\end{equation}
to be the set of generators associated to the pair of sheets $(S_m,S_n)$ with $S_m\subset\Lambda_v$ and $S_n\subset\Lambda_w$. As a $\Bbbk$-bimodule, $\mathcal{C}(\Lambda)$ is the tensor algebra over $\Bbbk$ defined by 
\begin{equation}\mathcal{C}(\Lambda):=\bigoplus_{i=0}^\infty\mathbb{K}\langle a^{m,n}_\alpha,b^{m,n}_\alpha,c^{m,n}_\alpha\rangle^{\otimes_\Bbbk i}.\end{equation}

We shall assume from now on that the Maslov class of $\Lambda$ vanishes. Since we will deal only with the case when $\Lambda$ is a disjoint union of Legendrian spheres, this assumption is automatically satisfied for all the examples studied in this paper. This allows us to endow $\mathcal{C}(\Lambda)$ with a $\mathbb{Z}$-grading. Fix a \textit{Maslov potential}
\begin{equation}\mu:\Lambda\rightarrow\mathbb{Z},\end{equation}
which is a locally constant function whose value increases by 1 when passing from the lower sheet to the upper sheet at a cusp edge. Each of the generators $a^{m,n}_\alpha,b^{m,n}_\alpha,c^{m,n}_\alpha$ is homogeneous in $\mathcal{C}(\Lambda)$, and their degrees are specified as follows:
\begin{equation}|a_\alpha^{m,n}|=\mu(S_n)-\mu(S_m)+1,|b_\alpha^{m,n}|=\mu(S_n)-\mu(S_m),|c_\alpha^{m,n}|=\mu(S_n)-\mu(S_m)-1.\end{equation}
Note that our convention here is different from $\cite{rs1}$ as we shall use a cohomological grading on $\mathcal{C}(\Lambda)$.\bigskip

In order to define the differential $d_\mathcal{C}$, we require $d_\mathcal{C}(1)=0$, and specify the effect of $d_\mathcal{C}$ on the generators $a^{m,n}_\alpha$, $b^{m,n}_\alpha$ and $c^{m,n}_\alpha$ of $\mathcal{C}(\Lambda)$ separately.\bigskip

Consider a 0-cell $e_\alpha^0$. By extending the partial ordering $\prec$ on $\Lambda(e_\alpha^0)$ we can define a total ordering on the set of sheets above $e_\alpha^0$, which can be equivalently described by a bijection $\rho:\{1,\dots,r\}\rightarrow I_\alpha$, where $I_\alpha$ is the index set recording the subscripts of the sheets in $\Lambda(e_\alpha^0)$. This total ordering on $\Lambda(e_\alpha^0)$ enables us to assemble the generators $a_\alpha^{m,n}$ in a strictly upper triangular matrix $A_\alpha$, with its $(i,j)$-th entry given by $a_\alpha^{\rho(i),\rho(j)}$ if $S_{\rho(i)}\prec S_{\rho(j)}$, and 0 otherwise. $d_\mathcal{C}$ on the generators $a_\alpha^{m,n}$ is then determined by the matrix equation
\begin{equation}\label{eq:diff1}
d_\mathcal{C} A_\alpha=A_\alpha^2,
\end{equation}
where on the left-hand side $d_\mathcal{C}$ is applied entrywisely. It is proved in $\cite{rs1}$ that $d_\mathcal{C} a_\alpha^{m,n}$ is independent of the choice of the total ordering on $\Lambda(e_\alpha^0)$ which extends $\prec$.\bigskip

In the case of a 1-cell $e_\alpha^1$, we can again enhance the partially ordered set $(\Lambda(e_\alpha^1),\prec)$ to get a totally ordered set by specifying a bijection $\rho$ as above. This then gives us a strictly upper triangular $r\times r$ matrix $B_\alpha$ whose $(i,j)$-th entry equals $b_\alpha^{\rho(i),\rho(j)}$ if $S_{\rho(i)}\prec S_{\rho(j)}$, and $b_\alpha^{^{\rho(i),\rho(j)}}=0$ otherwise. Since the structure of a polygonal decomposition of $S$ includes as its data a set of characteristic maps $c_\alpha^i$, we are allowed to distinguish between the initial and terminal 0-cells $e_{\alpha,+}^0$ and $e_{\alpha,-}^0$. Notice however that a 1-cell can have identical initial and terminal points. For the 0-cells $e_{\alpha,+}^0$ and $e_{\alpha,-}^0$, we can associate to them two $r\times r$ matrices $A_{\alpha,+}$ and $A_{\alpha,-}$ as follows.

Each sheet above the 0-cell $e_{\alpha,+}^0$ belongs to the closure of a unique sheet in $\Lambda(e_\alpha^1)$. Under the bijection $\rho$, this induces an order-preserving injective map
\begin{equation}\label{eq:inj}
\iota:\Lambda(e_{\alpha,+}^0)\hookrightarrow\{1,\dots,r\}.
\end{equation}
Those sheets of $e_\alpha^1$ not in the image of $\iota$ meet in pairs at cusp points above $e_{\alpha,+}^0$. The $(i,j)$-th entry of $A_{\alpha,+}$ is defined to be $a_{\alpha,+}^{m,n}$ if $S_m\prec S_n$ and $\iota(m)=i,\iota(n)=j$. The $(k,k+1)$-th entry of $A_{\alpha,+}$ is 1 if the sheets numbered $k$ and $k+1$ of $\Lambda(e_\alpha^1)$ under $\rho$ meet at a cusp singularity above $e_{\alpha,+}^0$. All the other entries of $A_{\alpha,+}$ are set to be 0. Alternatively, one can use the total ordering specified by $\iota$ to form a matrix out of the generators $a_{\alpha,+}^{m,n}$, and then insert $2\times 2$ blocks $\left[\begin{array}{ll}0 & 1\\0 & 0\end{array}\right]$ along the diagonal for each pair of sheets of $\Lambda(e_\alpha^1)$ that meet at a cusp edge above $e_{\alpha,+}^0$. The definition of $A_{\alpha,-}$ is completely identical.

With these matrices at hand, $d_\mathcal{C} b_\alpha^{m,n}$ can be defined by the following matrix equation:
\begin{equation}\label{eq:diff2}
d_\mathcal{C} B_\alpha=A_{\alpha,+}(E+B_\alpha)+(E+B_\alpha)A_{\alpha,-},
\end{equation}
with $E$ being the identity matrix. Again, $d_\mathcal{C} b_\alpha^{m,n}$ is independent of the choice of the total ordering on $\Lambda(e_\alpha^1)$.\bigskip

When it comes to a 2-cell $e_\alpha^2$, the partial ordering on $\Lambda(e_\alpha^2)$ is already a total ordering, so we can label the sheets in $\Lambda(e_\alpha^2)$ directly by $S_1,\dots,S_r$ so that $z(S_1)>\dots>z(S_r)$. The sheets above the edges and vertices in $\partial e_\alpha^2$ are therefore naturally identified with subsets of $\{1,\dots,r\}$. For each such edge or vertex we can define a strictly upper triangular $r\times r$ matrix by using the corresponding generators $b_\alpha^{m,n}$ or $a_\alpha^{m,n}$ and $2\times2$ blocks
\begin{equation}\label{eq:block}
O:=\left[\begin{array}{ll}0 & 0\\0 & 0\end{array}\right](\text{for an edge}), N:=\left[\begin{array}{ll}0 & 1\\0 & 0\end{array}\right](\text{for a vertex})
\end{equation}
inserted in the diagonal whenever $S_k,S_{k+1}$ meet at a cusp singularity above the edge or vertex. Just as what we have done for $A_{\alpha,+}$ and $A_{\alpha,-}$ above. Notice that in the case when $\Lambda=\bigsqcup_v\Lambda_v$ and the cusp singularity formed by $S_k,S_{k+1}$ belongs to the component $\Lambda_v$, then the non-zero entry in the matrix $N$ above should be replaced by the idempotent $e_v$ in the semisimple ring $\Bbbk$.

For each 2-cell $e_\alpha^2$, the characteristic map $c_\alpha^2:D^2\rightarrow S$ determines the initial and terminal vertices $v_-^\alpha,v_+^\alpha\in\partial D^2$, whose associated matrices defined in the last paragraph will again be denoted by $A_{\alpha,-}$ and $A_{\alpha,+}$. Let $\gamma_+$ and $\gamma_-$ be the arcs in $\partial D^2$ that go counterclockwisely and clockwisely from $v_-$ to $v_+$ respectively. Note that these paths can be constant or the entire circle. Consider the image of $c_\alpha^2\circ\gamma_+$, it contains a set of successive 1-cells $e_{\alpha,1}^1,\dots,e_{\alpha,n_+}^1$, whose associated matrices defined in the last paragraph will be denoted by $B_{\alpha,1},\dots,B_{\alpha,n_+}$. Similarly, for the path $\gamma_-$ we get another set of 1-cells $e_{\alpha,n_++1}^1,\dots,e_{\alpha,n_++n_-}^1$ with associated matrices $B_{\alpha,n_++1},\dots,B_{\alpha,n_++n_-}$. Finally, similar to the cases of 0-cells and 1-cells, we can form the matrix $C_\alpha$ using the generators $c_\alpha^{m,n}$ corresponding to the 2-cell $e_\alpha^2$. Now the differential $d_\mathcal{C}$ on $c_\alpha^{m,n}$ is defined via the matrix equation
\begin{equation}\label{eq:c}
\begin{split}
d_\mathcal{C} C_\alpha&=A_{\alpha,+}C_\alpha+C_\alpha A_{\alpha,-}+(E+B_{\alpha,n_+})^{\varepsilon_{n_+}}\dots(E+B_{\alpha,1})^{\varepsilon_1} \\
& +(E+B_{\alpha,n_++n_-})^{\varepsilon_{n_++n_-}}\dots(E+B_{\alpha,n_++1})^{\varepsilon_{n_++1}},
\end{split}
\end{equation}
where $\varepsilon_i=1$ if the orientation on the 1-cell $e_{\alpha,i}^+$ as an edge of $e_\alpha^2$ coincides with the orientation determined by the characteristic map of the 1-cell, otherwise $\varepsilon_i=-1$.
\bigskip

In all cases, it can be checked that $d_\mathcal{C}^2=0$. This defines the cellular dg algebra $(\mathcal{C}(\Lambda),d_\mathcal{C})$ when there is no swallowtail points in the front projection of $\Lambda$. When $\Lambda=\bigsqcup_v\Lambda_v$ is a link of Legendrian surfaces, one can check that the differential $d_\mathcal{C}$ defined above is compatible with the $\Bbbk$-bimodule structure (\ref{eq:ss}) on $\mathcal{C}(\Lambda)$, which shows that $\mathcal{C}(\Lambda)$ is a dg algebra over $\Bbbk$. The definition of $\mathcal{C}(\Lambda)$ can be extended to the case when swallowtail singularities present in the front projection of $\Lambda$, with some modifications to the matrices $A_\alpha$, $B_\beta$ and $C_\gamma$. Since this will not be used for later computations, its definition will not be recalled here. See Section 3.11 of $\cite{rs1}$ for details.
\bigskip

Up to quasi-isomorphism, the dg algebra $\mathcal{C}(\Lambda)$ is independent of the choice of the polygonal decomposition. Together with the quasi-isomorphism (\ref{eq:RS2}) established in $\cite{rs2}$, we see that the quasi-isomorphism type of $\mathcal{C}(\Lambda)$ defines an invariant of $\Lambda$ under Legendrian isotopy.

\subsection{Non-genericity}\label{section:non-gen}

For computational convenience, we will allow another type of non-genericity of Legendrian fronts, namely when there are multiple crossings or cusp edges above a 1-cell in the cellular decomposition. In these cases, we can modify slightly our original definitions of the cellular dg algebra $(\mathcal{C}(\Lambda),d_\mathcal{C})$ to get a (usually simpler) dg algebra $(\mathcal{C}^\ell(\Lambda),d_\mathcal{C}^\ell)$, whose definition we will recall below.\bigskip

We deal first with the case when multiple crossing arcs appear above some subset of $\Sigma_1\subset S$. To be precise, let $\Lambda\subset J^1(S)$ be a Legendrian surface. Consider a polygonal decomposition of $p_x(\Lambda)$ which is $\Lambda$-compatible except near a 2-sided simple closed curve $\ell\subset S$. Suppose that in a neighborhood $\overline{U}\cong S^1\times[0,1]$ of $\ell:=S^1\times\{\frac{1}{2}\}$, there are several crossing arcs of $\Lambda$ which project to small shifts of $\ell$ in the normal direction, and no other crossings or cusp edges. In this case, we can assume that all the crossing arcs near $\ell$ project precisely to $\ell$ and define the generators in the dg algebra $\mathcal{C}^\ell(\Lambda)$ associated to this incompatible cellular decomposition as follows. See Figure \ref{fig:non-generic} for an illustration.

More precisely, label the sheets of $\Lambda$ above one side of the neighborhood $\overline{U}$ as $S_1,\dots,S_{r}$, so that $z(S_1)>\cdots>z(S_{r})$ above that side of $\overline{U}$. The key point is that there exists a permutation $\sigma$ on the set $\{1,\dots,r\}$ such that the sheet $S_i$ appears as the sheet $S_{\sigma(i)}$ on the other side of the small neighbourhood $\overline{U}$, so that all the crossings of the sheets of $\Lambda$ happened in $\overline{U}$ (or equivalently, one can treat them as crossings over $\ell\subset\overline{U}$ ) are recorded by this permutation.

To each 1-cell $e_\beta^1$ (resp. 0-cell $e_\alpha^0$) of $\ell$, assign generators $b^{m,n}_\beta$ (resp. $a^{m,n}_\alpha$) for all $m<n$ with $\sigma(m)<\sigma(n)$, so that there are multiple zeros in the corresponding matrices $B_\beta$ (resp. $A_\alpha$) of generators, since sheets which cross with each other in $\overline{U}$ will then satisfy $\sigma(m)>\sigma(n)$. The differential $d_\mathcal{C}^\ell$ is defined by the same formulas (\ref{eq:diff1}), (\ref{eq:diff2}) and (\ref{eq:c}) as in the case of a usual cellular dg algebra $\mathcal{C}(\Lambda)$. Note that if $B_\beta$ is the matrix of generators associated to the sheets labelled by $S_1,\dots,S_r$, then the corresponding matrix of generators associated to the sheets labelled by $S_{\sigma^{-1}(1)},\dots,S_{\sigma^{-1}(r)}$ is $Q_\sigma^{-1}B_\beta Q_\sigma$, with $Q_\sigma=\sum_{m=1}^r\Delta_{\sigma(m),m}$,and $\sigma$ being a composition of transpositions on $\{1,\dots,r\}$.
\begin{figure}[htb!]
	\centering
	\begin{tikzpicture}
	\draw [cyan] (0,-0.9) node[circle,fill,inner sep=1pt](a){} -- (3,-0.9) node[circle,fill,inner sep=1pt](b){};
	\node [cyan] at (1.5,-0.65) {$\ell$};
	\draw (5,0) node[circle,fill,inner sep=1pt](a){} to (8,0) node[circle,fill,inner sep=1pt](b){};
	\draw (5,-0.6) node[circle,fill,inner sep=1pt](a){} to (8,-0.6) node[circle,fill,inner sep=1pt](b){};
	\draw (5,-1.2) node[circle,fill,inner sep=1pt](a){} to (8,-1.2) node[circle,fill,inner sep=1pt](b){};
	\draw (5,-1.8) node[circle,fill,inner sep=1pt](a){} to (8,-1.8) node[circle,fill,inner sep=1pt](b){};
	\draw [cyan] (5,-0.9) to (8,-0.9);
	\node [cyan] at (6.5,-0.75) {\footnotesize $\ell$};
	\draw (5,0) to (5,-1.8);
	\draw (8,0) to (8,-1.8);
	\end{tikzpicture}
	\caption{Local picture of a $\Lambda$-compatible polygonal decomposition associated to a non-generic Legendrian front of $\Lambda$, which projects several crossings/cusp edges to the unique 1-cell $\ell$ (left), and the neighbourhood $\overline{U}$ of a $\Lambda$-compatible polygonal decomposition associated to a generic Legendrian front of $\Lambda$, where different crossings/cusp edges are projected to multiple 1-cells, the dg algebra $(\mathcal{C}^\ell(\Lambda),d_\mathcal{C}^\ell)$ is defined by assembling these 1-cells to $\ell$ by making the front of $\Lambda$ non-generic (right)}
	\label{fig:non-generic}
\end{figure}
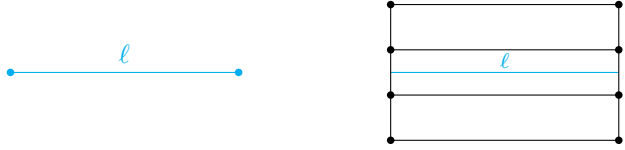
\bigskip

When several cusp edges appear above the curve $\ell$, the situation is much simpler. For each 0-cell (resp. 1-cell) of $\ell$, it follows from (\ref{eq:block}) that we only need to insert multiple copies of $2\times2$ nilpotent blocks $N$ (resp. zero blocks $O$) in the construction of the matrices $A_\alpha$ (resp. $B_\beta$). For example, let $\Lambda\subset(\mathbb{R}^5,\xi_\mathit{std})$ be a Legendrian surface so that locally its base projection $p_x(\Lambda)$ is depicted on the left-hand side of Figure \ref{fig:non-gen-cusp}. Suppose that the Legendrian front of $\Lambda$ above the 0-cell $e_\alpha^0$ (resp. 1-cell $e_\beta^1$) consists only of cusp edges formed by couples of sheets $(S_1,S_2),\cdots,(S_{2r-1},S_{2r})$, and the ordering of these sheets is chosen so that $z(S_1)>\cdots>z(S_{2r})$ locally in a small neighborhood on the right-hand side of the solid arc labelled $e_\beta^1$. In this case, one can take
\begin{equation}\label{eq:cusps}
A_\alpha=\sum_{i=1}^{r}\Delta_{2i-1,2i},B_\beta=0.
\end{equation}
Again, the relevant differentials are defined by the same formulas as in the front generic case. This completes the definition of $(\mathcal{C}^\ell(\Lambda),d_\mathcal{C}^\ell)$.

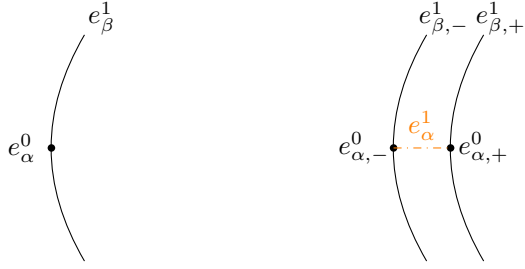
\begin{figure}
	\centering
	\begin{tikzpicture}[scale=1.5]
	\draw (0,-1) to [in=-120,out=120] (0,1);
	\draw (-0.29,0) node[circle,fill,inner sep=1pt] {};
	\node at (-0.55,0) {$e_\alpha^0$};
	\node at (0.15,1.15) {$e_\beta^1$};
	
	\draw (3,-1) to [in=-120,out=120] (3,1);
	\draw (2.71,0) node[circle,fill,inner sep=1pt] {};
	\node at (2.45,0) {$e_{\alpha,-}^0$};
	\node at (3.15,1.15) {$e_{\beta,-}^1$};
	\draw (3.5,-1) to [in=-120,out=120] (3.5,1);
	\draw (3.21,0) node[circle,fill,inner sep=1pt] {};
	\node at (3.5,0) {$e_{\alpha,+}^0$};
	\node at (3.65,1.15) {$e_{\beta,+}^1$};
	\draw [orange,dash dot] (2.71,0) to (3.21,0);
	\node [orange] at (2.96,0.2) {$e_\alpha^1$};
	\end{tikzpicture}
	\caption{The base projections $p_x$ of a Legendrian surface $\Lambda$ near its cusp edges, where the images of $\Lambda$ under $p_x$ appear on the right-hand side of the leftmost solid arcs in both of the figures. On the left hand side, all the cusp edges in the front of $\Lambda$ formed by the sheets $(S_1,S_2),\cdots,(S_{2r-1},S_{2r})$ are projected to the same solid arc labelled as the 1-cell $e_\beta^1$. On the right hand side, we are in the special case when $r=2$, and a small Legendrian isotopy of $\Lambda$ has been chosen so that the base projections of the cusp edges formed by $(S_3,S_4)$ and $(S_1,S_2)$ are projected to different solid arcs labelled respectively by $e_{\beta,-}^1$ and $e_{\beta,+}^1$.}
	\label{fig:non-gen-cusp}
\end{figure}
\bigskip

The following proposition can be proved by applying small isotopies to $\Lambda$ to make it front generic, and then simplifying $\mathcal{C}(\Lambda)$ using stable tame isomorphisms. The latter step can be achieved by applying Lemma \ref{lemma:tame}, which will be recalled in the next subsection, to the original definition of $\mathcal{C}(\Lambda)$ recalled in Section \ref{section:def-cell}. 
\begin{proposition}[Proposition 5.5 of $\cite{rs1}$]\label{propositioN:non-gen}
The dg algebra $\mathcal{C}^\ell(\Lambda)$ defined above using a $\Lambda$-compatible polygonal decomposition associated to a non-generic Legendrian front of $\Lambda$ is quasi-isomorphic to the cellular dg algebra $\mathcal{C}(\Lambda)$ defined using any $\Lambda$-compatible cellular decomposition associated to a generic front projection of $\Lambda$.
\end{proposition}

To see the how Proposition \ref{propositioN:non-gen} can be used to simplify our computations, we consider the simplest case when the non-generic Legendrian front of $\Lambda$ above the 1-cell $e_\beta^1$ consists of two cusp edges formed by the pairs of sheets $(S_1,S_2)$ and $(S_3,S_4)$ respectively, see the left-hand side of Figure \ref{fig:non-gen-cusp}, where as before we require that $z(S_1)>z(S_2)>z(S_3)>z(S_4)$ locally on the right-hand side of $e_\beta^1$. By (\ref{eq:cusps}), we see that $A_\alpha=\left[\begin{array}{ll}N & 0\\0 & N\end{array}\right]$ in the dg algebra $\mathcal{C}^\ell(\Lambda)$. However, after applying a small isotopy to make $\Lambda$ front generic, the 0-cell $e_\alpha^0$ in the original cellular decomposition is replaced by two 0-cells $e_{\alpha,-}^0$ and $e_{\alpha,+}^0$ in the $\Lambda$-compatible cellular decomposition, and correspondingly the 1-cell $e_\beta^1$ is replaced by the 1-cells $e_{\beta,-}^1$ and $e_{\beta,+}^1$, see the right-hand side of Figure \ref{fig:non-gen-cusp}. More explicitly, we require that above the 1-cell $e_{\beta,+}^1$, the top two sheets meet above the cusp edge. By definition of the cellular dg algebra $\mathcal{C}(\Lambda)$ in Section \ref{section:def-cell}, $A_{\alpha,-}=N$. To find the values of $A_{\alpha,+}$, recall from (\ref{eq:inj}) the definition of the map $\iota:\Lambda(e_{\alpha,+}^0)\rightarrow\{1,2,3,4\}$. Since in this case the sheets numbered 1 and 2 of $\Lambda(e_{\beta,+}^1)$ meet at a cusp edge, we see that 
\begin{equation}
A_{\alpha,+}=\begin{bmatrix}
0 & 1 & 0 & 0 \\
0 & 0 & 0 & 0 \\
0 & 0 & 0 & a_{\alpha,+}^{3,4} \\
0 & 0 & 0 & 0
\end{bmatrix}
\end{equation}
Consider the 1-cell $e_\alpha^1$ which connects $e_{\alpha,-}^0$ to $e_{\alpha,+}^0$ on the right-hand side of Figure \ref{fig:non-gen-cusp}, it follows from (\ref{eq:diff2}) that
\begin{equation}
d_\mathcal{C}b_\alpha^{3,4}=a_{\alpha,+}^{3,4}+1,
\end{equation}
which shows the existence of a quasi-isomorphism $\mathcal{C}(\Lambda)\cong\mathcal{C}(\Lambda)/\langle a_{\alpha,+}^{3,4}+1,b_\alpha^{3,4}\rangle$, under which $A_{\alpha,+}$ simplifies to $A_\alpha$.

In general, Proposition \ref{propositioN:non-gen} says that by applying Lemma \ref{lemma:tame} to the formula (\ref{eq:diff2}), $A_{\alpha,+}$ in the cellular dg algebra $\mathcal{C}(\Lambda)$ can finally be simplified to the form of $A_\alpha$ in the dg algebra $\mathcal{C}^\ell(\Lambda)$, which is quasi-isomorphic to $\mathcal{C}(\Lambda)$. It is therefore more convenient to start with a non-generic front of $\Lambda$ and replace $\mathcal{C}(\Lambda)$ with its quotient dg algebra $\mathcal{C}^\ell(\Lambda)$. In this way, the formal generators in the matrix $A_{\alpha,+}$ are therefore ignored, and one can work directly with the matrix $A_\alpha$, whose entries are purely scalars. This simplification will be used frequently below, in various computations and arguments in Section \ref{section:combi}.

\subsection{Suspension of a dg algebra}\label{section:susp-dga}

Let $(\mathcal{A},d)$ be a \textit{based} dg algebra over $\mathbb{K}$, which means that it is freely generated over $\mathbb{K}$ by a set of homogeneous elements $a_1,\dots,a_n$. Fix an absolute ordering $a_1<\cdots<a_n$ on the generating set $\{a_1,\dots,a_n\}$, there is then a natural filtration $F^\bullet$ on $\mathcal{A}$ with $F^0\mathcal{A}=\mathbb{K}$ and $F^i\mathcal{A}\subset\mathcal{A}$ is the subalgebra generated by $a_1,\dots,a_i$. We remark that in general, this filtration $F^\bullet$ has nothing to do with the grading on $\mathcal{A}$. Recall that the differential $d$ is said to be \textit{triangular} if for all $1\leq i\leq n$, $da_i\in F^{i-1}\mathcal{A}$.\bigskip

A map between based dg algebras $\phi:\mathcal{A}\rightarrow\mathcal{B}$ is a \textit{tame isomorphism} if it is the composition of a sequence of elementary automorphisms of $\mathcal{A}$ followed by an identification between the sets of generators of $\mathcal{A}$ and $\mathcal{B}$. By an \textit{elementary tame automorphism} of $\mathcal{A}$ we mean a dg algebra automorphism which sends a fixed generator $a_i\in\mathcal{A}$ to $a_i+v$, where $v$ belongs to the dg subalgebra of $\mathcal{A}$ generated by $a_1,\dots,a_{i-1},a_{i+1},\dots,a_n$. In general, any tame isomorphism $\phi$ can be expressed as the composition of a sequence of elementary tame automorphisms of $\mathcal{A}$ followed by an identification between the sets of generators of $\mathcal{A}$ and $\mathcal{B}$.
\bigskip

Denote by $(S_i\mathcal{A},d_i)$ the \textit{degree} $i$ \textit{stabilization} of $\mathcal{A}$. This is the dg algebra with generating set consisting of the original generators of $\mathcal{A}$, together with two additional elements $a,b\in S_i\mathcal{A}$ such that $|a|=|b|+1=i$, which satisfy
\begin{equation}
d_ia=b,d_ib=0
\end{equation}
and $d_i$ coincides with $d$ when restricted to the dg subalgebra $\mathcal{A}\subset S_i\mathcal{A}$. Two dg algebras $\mathcal{A}$ and $\mathcal{B}$ are \textit{stable tame isomorphic} if after stabilizing $\mathcal{A}$ and $\mathcal{B}$ for (possibly different) finite number of times, they become tame isomorphic to each other. We record the following simple result, whose proof dates back to the fundamental paper of Chekanov $\cite{yc}$, which is useful in simplifying dg algebras up to stable tame isomorphism.
\begin{lemma}\label{lemma:tame}
Let $(\mathcal{A},d)$ be a based dg algebra over $\mathbb{K}$ such that $d$ is triangular with respect to the ordered generating set $\{a_1,\dots,a_n\}$, and
\begin{equation}da_i=a_j+v,v\in F^{j-1}\mathcal{A}.\end{equation}
Then we have a stable tame isomorphism $(\mathcal{A},d)\cong(\mathcal{A}/\mathcal{I},d)$ between based dg algebras, where $\mathcal{I}\subset\mathcal{A}$ is the 2-sided ideal generated by $a_i$ and $da_i$.
\end{lemma}
The above lemma will be frequently applied to the cellular dg algebra $\mathcal{C}(\Lambda)$ in Section \ref{section:combi} to cancel excessive generators, and we will denote by $\mathcal{C}'(\Lambda)$ the simplified dg algebra obtained by applying all the possible cancellations to $\mathcal{C}(\Lambda)$ in the sense of Lemma \ref{lemma:tame}. Note that this does not in general lead to a model of the Chekanov-Eliashberg algebra $\mathcal{C}E(\Lambda_{p,q,r})$ that is minimal, especially when $\mathcal{C}E(\Lambda_{p,q,r})$ is formal as an $A_\infty$-algebra. The Legendrian surface $\Lambda_{1,1,0}$ mentioned in Section \ref{section:exceptional} is such an example.
\bigskip

The following notion is useful in the computation of the Chekanov-Eliashberg dg algebras for Legendrian surfaces which can be realized as front spins of Legendrian knots, or more generally, Legendrian arcs. Its geometric application is discussed in Section \ref{section:spin}.

\begin{definition}[Section 5.5.1 of \cite{rs1}]\label{definition:susp}
Let $(\mathcal{A},d)$ be a dg algebra over $\mathbb{K}$ freely generated by the ordered set $\{a_1,\dots,a_n\}$, and $d$ is triangular with respect to these generators. We define the suspension $(\mathcal{A}^s,d^s)$ of $(\mathcal{A},d)$ to be the dg algebra freely generated by $\{a_i,\check{a}_i\}_{i=1,\dots,n}$ such that
\begin{itemize}
\item $|\check{a}_i|=|a_i|-1$, where the grading of $a_i$ in $\mathcal{A}^s$ is the same as its grading in $\mathcal{A}$.
\item Let $D:\mathcal{A}\rightarrow\mathcal{A}^s$ be the derivation determined by $D(a_i)=\check{a}_i$, then the differential $d^s$ on $\mathcal{A}^s$ is determined by
\begin{equation}d^s(a_i)=d(a_i),d^s(\check{a}_i)=D(d(a_i)).\end{equation}
\end{itemize}
Let $\mathcal{O}\subset\mathcal{A}$ be a based dg subalgebra generated by some subset of $\{a_1,\dots,a_n\}$. Define the suspension of $\mathcal{A}$ relative to $\mathcal{O}$ to be the dg algebra $\mathcal{A}^s$ as above except that we set $\check{o}=0$ for all $o\in\mathcal{O}$.
\end{definition}

\subsection{Spinning a Legendrian arc}\label{section:spin}

Let $J^1(\mathbb{R})$ denote $\mathbb{R}^3$ with its standard contact structure. The corresponding front and base projections will be denoted by $p_{x_1,z}$ and $p_{x_1}$ . Given a Legendrian submanifold $K\subset J^1(\mathbb{R})$ which is a union of a number of knots in $\{x_1<0\}$ and a number of arcs in $\{x_1\leq0\}$ whose endpoints under the front projection lie on the $z$-axis. In addition, we shall require that when we reflect $p_{x_1,z}(K)\subset\mathbb{R}^2$ with respect to the $z$-axis, the image should be the front projection of a Legendrian link. This can actually be achieved for any Legendrian embedding $K\subset\mathbb{R}^3$ whose connected components are knots and arcs by a suitable Legendrian isotopy. By rotating along the $z$-axis, such a $K\subset J^1(\mathbb{R})$ gives rise to a Legendrian surface $\Lambda_K\subset J^1(\mathbb{R}^2)$, whose front $p_{x,z}(\Lambda_K)\subset\mathbb{R}^3$ is also the rotation of $p_{x_1,z}(K)\subset\mathbb{R}^2$ along the same axis. We call $\Lambda_K$ the \textit{front spin} of $K$. It is easy to see that the surface $\Lambda_K$ is a disjoint union of Legendrian tori and Legendrian spheres. Note that when two endpoints of a Legendrian arc coincide with each other on the $z$-axis under the front projection, then its front spin will have a cone singularity. In this case, although $\Lambda_K$ is not front generic, its cellular dg algebra $\mathcal{C}(\Lambda_K)$ can still be defined and computed with only slight modifications. However, we will not need to deal with cone singularities in the present paper, since they have already been cancelled in the Legendrian front of $\Lambda_{p,q,r}$ using higher dimensional Reidemeister moves in the proof of Proposition \ref{proposition:link}.
\bigskip

Analogous to the surface case, we can consider the $K$-compatible polygonal decomposition of the real line $\mathbb{R}$. When $K$ has an arc component, this decomposition will include the origin of $\mathbb{R}$ as a 0-cell. For each cell $e$ in the cellular decomposition of $p_{x_1}(K)$, we can associate to it two cells $e$ and $\check{e}$ in the $\Lambda_K$-compatible polygonal decomposition of $\mathbb{R}^2$. The first one $e$ can be identified with the original cell through the embedding of the $x_1$-axis into $\mathbb{R}^2_{x_1,x_2}$. The second one $\check{e}$ is the spinning of $e$ around the $z$-axis. Note that $e$ and $\check{e}$ are the same if $e$ is the origin $\{x_1=0\}$, and we denote by $o$ the unique 0-cell it induces in the cellular decomposition of $p_x(\Lambda_K)$.
\bigskip

From the above we get two cellular dg algebras: one is associated to the cellular decomposition of $p_{x_1}(K)$, and the other one is associated to the $\Lambda_K$-compatible polygonal decomposition of $\mathbb{R}^2$ induced by the cellular decomposition of $p_{x_1}(K)$. We denote these dg algebras by $(\mathcal{C}(K),d_\mathcal{C})$ and $(\mathcal{C}(\Lambda_K),d_\mathcal{C}^s)$ respectively. We remark that explicitly the differential $d_\mathcal{C}$ of the former cellular dg algebra $\mathcal{C}(K)$ is defined using (\ref{eq:diff1}) and (\ref{eq:diff2}). In fact, $\mathcal{C}(K)$ is quasi-isomorphic to the bordered Chekanov-Eliashberg algebra introduced by Sivek $\cite{ssi}$. The way that we used to obtain the cellular decomposition of $p_x(\Lambda_K)$ suggests that algebraically these two dg algebras are related to each other through the suspension construction of Section \ref{section:susp-dga}. In fact, we have the following result.

\begin{proposition}[Proposition 5.2 of $\cite{rs1}$]\label{proposition:front-spin}
Suppose $K\subset J^1(\mathbb{R})$ has arc components whose front projections have distinct endpoints, then the cellular dg algebra $(\mathcal{C}(\Lambda_K),d_\mathcal{C}^s)$ is the suspension of $(\mathcal{C}(K),d_\mathcal{C})$ relative to the dg subalgebra $\mathcal{O}(K)$ associated to the 0-cell $e_o^0$ in the cellular decomposition of $p_x(\Lambda_K)$. If $K$ has two arc components whose front projections have the same end point, so that its front spin $\Lambda_K$ contains a unique cone singularity above $e_o^0$, $\mathcal{C}(\Lambda_K)$ is the suspension of $\mathcal{C}(K)$ relative to $\mathcal{O}(K)$ with the modification that in this case
\begin{equation}
D(A_o)=\sum_{m<k}a_o^{m,k+1}\Delta_{m,k}+\sum_{k+1<n}a_o^{k,n}\Delta_{k+1,n}
\end{equation}
instead of $D(A_o)=0$, where $A_o$ is the matrix of generators associated to the 0-cell $e_o^0$, and the cone point over $e_o^0$ connects the sheets $S_k$ and $S_{k+1}$.
\end{proposition}

\paragraph{Remark} When doing computations of $\mathcal{C}(\Lambda_K)$ using Proposition \ref{proposition:front-spin}, one is supposed to start with a Maslov potential $\mu_K:K\rightarrow\mathbb{Z}$ on the Legendrian arc $K$, and then appeal to the grading convention of Definition \ref{definition:susp} to determine the gradings on $\mathcal{C}(\Lambda_K)$. However, care must be taken as a disconnected Legendrian arc $K$ may become connected after spinning it around. As an example, we have the Legendrian surface $\Lambda_{1,1,0}$ obtained by spinning a two-component Legendrian arc $K_{1,1,0}$, see Figure \ref{fig:trefoil}. In this situation, we need to make sure that the Maslov potential $\mu_K$ that we use in the computation of the dg algebra $\mathcal{C}(K)$ is induced from some well-defined potential $\mu_\Lambda:\Lambda_K\rightarrow\mathbb{Z}$.

\section{Combinatorial computations}\label{section:combi}

The main purpose of this section is to compute the Chekanov-Eliashberg dg algebras for the 2-dimensional Legendrian links $\Lambda_{p,q,r}\subset J^1(\mathbb{R}^2)$. For the computations in this section, we will work over $\mathbb{K}=\mathbb{Z}/2$, which enables us to adopt the cellular model introduced in the last subsection. As an illustration to more complicated computations, we will first compute in Section \ref{section:example} the cellular dg algebra of the $A_r$ type Legendrian attaching link $\Lambda_r$ of standard unknots, which also enables us to reduce the computation for a general link $\Lambda_{p,q,r}$ to the special case when $p=q=r=2$. The cellular dg algebra $\mathcal{C}(\Lambda_{2,2,2})$ will be computed in Section \ref{section:computation}. For the links of Legendrian surfaces $\Lambda_r$, their Legendrian fronts can be obtained by spinning a Legendrian arc, the algebraic construction in Section \ref{section:spin} will be used in order to simplify the computations.

\subsection{The $A_r$ link}\label{section:example}

We compute the cellular dg algebra $\mathcal{C}(\Lambda_r)$ of the Legendrian link $\Lambda_r\subset(S^5,\xi_\mathit{std})$, which is the Legendrian front associated to the 3-dimensional $A_r$ Milnor fiber
\begin{equation}
\{x^2+y^2+z^2+w^{r+1}=1\}\subset\mathbb{C}^4.
\end{equation}
This is aimed at helping the readers understand our computation in the more complicated case in Section \ref{section:computation}. Meanwhile, the computation here also provides an alternative proof of the Koszul duality result due to Ekholm-Lekili $\cite{ekl}$ for tree plumbings of $T^\ast S^3$'s in the special case when the plumbing tree $T=A_r$. The same method can be easily generalized to verify Koszul duality between the Fukaya categories of plumbings of $T^\ast S^3$'s according to any tree $T$.
\bigskip

Our proof uses the description of the Legendrian front of $\Lambda_r$ as the front spin of a Legendrian arc $K_r\subset J^1(\mathbb{R})$, see Figure \ref{figure:K3}. The computation of the cellular dg algebra $\mathcal{C}(\Lambda_r)$ can then be reduced to the computation of $\mathcal{C}(K_r)$ by Proposition \ref{proposition:front-spin}.

\begin{figure}[htb!]
	\centering
	\begin{tikzpicture}
	\draw [red] (0,0.15) -- (0,3.95);
	\draw (0,3.5) to [out=180,in=0] ++ (-1.2,-0.6);
	\draw (-1.2,2.9) to [out=0,in=180] ++ (1.2,-0.6);
	\draw (0,2.7) to [out=180,in=0] ++ (-1.2,-0.6);
	\draw (-1.2,2.1) to [out=0,in=180] ++ (1.2,-0.6);
	\draw (0,1.9) to [out=180,in=0] ++ (-1.2,-0.6);
	\draw (-1.2,1.3) to [out=0,in=180] ++ (1.2,-0.6);
	\end{tikzpicture}
	\caption{The Legendrian arc $K_3$, together with the axis $x_1=0$, coloured in red}
	\label{figure:K3}
\end{figure}
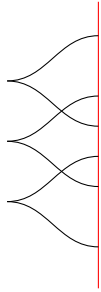
\bigskip

Consider the Legendrian arc $K_r$, which consists of $r$ connected components, and crossings happen only between the nearby strands labelled by $i$ and $i+1$. Without changing the Legendrian isotopy class, we can arrange that all the crossing points in the front of $K_r$ have the same base projection in $\mathbb{R}_{x_1}$. This violation of the genericity of the front projection is justified by the discussions in Section \ref{section:non-gen}. A $K_r$-compatible polygonal decomposition associated to such a non-generic front projection of $K_r$ is depicted in Figure \ref{figure:decomposition1}.
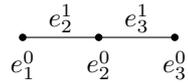
\begin{figure}[htb!]
	\centering
	\begin{tikzpicture}
	\draw (0,0) node[circle,fill,inner sep=1pt,label=below:$e^0_1$](a){} -- (1,0)
	node[circle,fill,inner sep=1pt,label=below:$e^0_2$](b){};
	\draw (1,0) -- (2,0) node[circle,fill,inner sep=1pt,label=below:$e^0_3$](c){};
	\node at (0.5,0.25) {$e^1_2$};
	\node at (1.5,0.25) {$e^1_3$};
	\end{tikzpicture}
	\caption{Cellular decomposition associated to $K_r$}
	\label{figure:decomposition1}
\end{figure}
\bigskip

Denote by $A_1,A_2,A_3,B_2,B_3$ the matrices of generators in $\mathcal{C}(K_r)$ associated to the cells $e^0_1,e^0_2,e_3^0,e_2^1,e^1_3$ in the above cellular decomposition. To start, we want to apply the formula (\ref{eq:diff2}) to the generators in the matrix $B_2$ associated to the 1-cell $e_2^1$. In this case, the role of matrix $A_{\alpha,-}$ in (\ref{eq:diff2}) is played by $A_1$, and the role of the matrix $A_{\alpha,+}$ is played by $A_2$. Since the cusps appear above the 0-cell $e_1^0$ are formed between the strands of $K_r$ labelled by
\begin{equation}
(1,2),(3,4),\dots,(2r-1,2r)
\end{equation}
above the 1-cell $e_2^1$, by (\ref{eq:cusps}) we have
\begin{equation}\label{eq:A1}
A_1=\sum_{i=1}^r\Delta_{2i-1,2i}.
\end{equation}
Since the strands of $K_r$ labelled by
\begin{equation}
(2,3),(4,5),\dots,(2r-2,2r-1)
\end{equation}
above the 0-cell $e_2^0$ cross with each other, it follows by definition of the matrix $A_2$ recalled in Section \ref{section:def-cell} that
\begin{equation}\label{eq:A2=0}
a_2^{2,3}=a_2^{4,5}=\dots=a_2^{2r-2,2r-1}=0.
\end{equation}
For the other non-zero entries $a_2^{m,n}$ in the strictly upper-triangular matrix $A_2$, we can cancel them using the formula of $d_\mathcal{C}B_2$. More precisely, we have by (\ref{eq:diff2}) that
\begin{equation}\label{eq:db2}
d_\mathcal{C}b_2^{m,n}=a_2^{m,n}+a_1^{m,n}+\sum_{m<k<n}a_2^{m,k}b_2^{k,n}+\sum_{m<k<n}b_2^{m,k}a_1^{k,n}.
\end{equation}
By Lemma \ref{lemma:tame}, we can define a filtration $F^\bullet$ on $\mathcal{C}(K_r)$ which respects the increasing order of $n-m$, and cancel the generators $b_2^{m,n}$ with $a_2^{m,n}$ according to the filtration $F^\bullet$ when $a_2^{m,n}\neq0$. After the cancellation process, we get a quotient dg algebra of $\mathcal{C}(K_r)$ which is quasi-isomorphic to $\mathcal{C}(K_r)$. In particular, in the quotient dg algebra $\mathcal{C}'(K_r)$, the non-zero generators in $B_2$ are
\begin{equation}\label{eq:B2remain}
b_2^{2,3},b_2^{4,5},\dots,b_2^{2r-2,2r-1},
\end{equation}
and the identities
\begin{equation}\label{eq:a2mn}
a_2^{m,n}=a_1^{m,n}+\sum_{m<k<n}a_2^{m,k}b_2^{k,n}+\sum_{m<k<n}b_2^{m,k}a_1^{k,n}
\end{equation}
hold. In fact, for generators $b_2^{m,n}$ which has been cancelled with $a_2^{m,n}$, since they vanish in the quotient dg algebra $\mathcal{C}'(K_r)$, it is clear that we have (\ref{eq:a2mn}) in $\mathcal{C}'(K_r)$ for any $(m,n)\neq(k,k+1)$, with $2\leq k\leq 2r-2$. When $(m,n)=(k,k+1)$ for some $k$, it follows from (\ref{eq:A1}) and (\ref{eq:A2=0}) that $a_1^{k,k+1}=a_2^{k,k+1}=0$, therefore by (\ref{eq:db2}) we have $d_\mathcal{C}b_2^{k,k+1}=0$, which shows that (\ref{eq:a2mn}) is still true. Using (\ref{eq:A1}) and (\ref{eq:B2remain}), we can then compute $A_2$ inside the quotient dg algebra $\mathcal{C}'(K_r)$, and deduce that
\begin{equation}\label{eq:a_2q}
A_2=\begin{bmatrix}
0 & 1 & b_2^{2,3} & 0 & 0 & 0 & \dots & 0 & 0 & 0 \\
0 & 0 & 0 & b_2^{2,3} & 0 & 0 & \dots & 0 & 0 & 0 \\
0 & 0 & 0 & 1 & b_2^{4,5} & 0 & \dots & 0 & 0 & 0 \\
0 & 0 & 0 & 0 & 0 & b_2^{4,5} & \dots & 0 & 0 & 0 \\
0 & 0 & 0 & 0 & 0 & 1 & \dots & 0 & 0 & 0 \\
\vdots & \vdots & \vdots & \vdots & \vdots & \vdots & \ddots & \vdots & \vdots & \vdots \\
0 & 0 & 0 & 0 & 0 & 0 & \dots & 1 & b_2^{2r-2,2r-1} & 0 \\
0 & 0 & 0 & 0 & 0 & 0 & \dots & 0 & 0 & b_2^{2r-2,2r-1} \\
0 & 0 & 0 & 0 & 0 & 0 & \dots & 0 & 0 & 1 \\
0 & 0 & 0 & 0 & 0 & 0 & \dots & 0 & 0 & 0
\end{bmatrix}
\end{equation}
in $\mathcal{C}'(K_r)$.

Since there is no crossing above the 1-cell $e_3^1$, none of the entries above the diagonal of $A_3$ is zero in $\mathcal{C}(K_r)$. Arguing similarly as above, we see that all the generators in the matrix $B_3$ can be cancelled with the generators in $A_3$. In $\mathcal{C}'(K_r)$, we have the following identity:
\begin{equation}\label{eq:a_3}
a_3^{m,n}=a_2^{\sigma^{-1}(m),\sigma^{-1}(n)}+\sum_{m<k<n}a_3^{m,k}b_3^{k,n}+\sum_{m<k<n}b_3^{m,k}a_2^{\sigma^{-1}(k),\sigma^{-1}(n)},
\end{equation}
where it follows from our discussions in Section \ref{section:non-gen} that the permutation $\sigma$ is given by
\begin{equation}
\sigma=(2,3)(4,5)\dots(2r-2,2r-1).
\end{equation}
It follows that in $\mathcal{C}'(K_r)$,
\begin{equation}\label{eq:a_3q}
A_3=\begin{bmatrix}
0 & b_2^{2,3} & 1 & 0 & 0 & 0 & 0 & \dots & 0 & 0 & 0 \\
0 & 0 & 0 & b_2^{4,5} & 1 & 0 & 0 & \dots & 0 & 0 & 0 \\
0 & 0 & 0 & 0 & b_2^{2,3} & 0 & 0 & \dots & 0 & 0 & 0 \\
0 & 0 & 0 & 0 & 0 & b_2^{6,7} & 1 & \dots & 0 & 0 & 0 \\
0 & 0 & 0 & 0 & 0 & 0 & b_2^{4,5} & \dots & 0 & 0 & 0 \\
\vdots & \vdots & \vdots & \vdots & \vdots & \vdots & \vdots & \ddots & \vdots & \vdots & \vdots \\
0 & 0 & 0 & 0 & 0 & 0 & 0 & \dots & b_2^{2r-2,2r-1} & 1 & 0 \\
0 & 0 & 0 & 0 & 0 & 0 & 0 & \dots & 0 & b_2^{2r-4,2r-3} & 0 \\
0 & 0 & 0 & 0 & 0 & 0 & 0 & \dots & 0 & 0 & 1 \\
0 & 0 & 0 & 0 & 0 & 0 & 0 & \dots & 0 & 0 & b_2^{2r-2,2r-1} \\
0 & 0 & 0 & 0 & 0 & 0 & 0 & \dots & 0 & 0 & 0
\end{bmatrix}
\end{equation}
By Proposition \ref{proposition:front-spin}, the cellular dg algebra $\mathcal{C}(\Lambda_r)$ is the suspension of $\mathcal{C}(K_r)$ relative to the dg subalgebra $\mathcal{O}(K_r)$ generated by generators in $A_3$. Denote by $D_r:\mathcal{C}(K_r)\rightarrow\mathcal{C}(\Lambda_r)$ the derivation appearing in the definition of the suspension of $\mathcal{C}(K_r)$, which acts on the generators by $D_r(a_i^{m,n})=\check{a}_i^{m,n}$ and $D_r(b_j^{m,n})=\check{b}_j^{m,n}$. Further cancellations among the generators in $\mathcal{C}(\Lambda_r)$ imply
\begin{proposition}\label{proposition:An}
$\mathcal{C}(\Lambda_r)$ is quasi-isomorphic to a dg algebra generated by
\begin{equation}\label{eq:genAn}
\begin{split}
&b_2^{2,3},b_2^{4,5},\dots,b_2^{2r-2,2r-1},\\
&\check{b}_3^{1,3},\check{b}_3^{2,5},\check{b}_3^{4,7},\dots,\check{b}_3^{2r-4,2r-1},\check{b}_3^{2r-2,2r},\\
&\check{b}_3^{2,3},\check{b}_3^{4,5},\dots,\check{b}_3^{2r-2,2r-1}.
\end{split}
\end{equation}
\end{proposition}
\begin{proof}
We have already seen how the generators in $\mathcal{C}(K_r)$ are cancelled, so it suffices to consider which one of the generators $\check{b}_2^{m,n}$ and $\check{b}_3^{m,n}$ remains in the quotient dg algebra $\mathcal{C}(\Lambda_r)$ after cancellation. Applying the derivation $D_r$ to (\ref{eq:db2}) we see that all the generators $\check{b}_2^{m,n}$ can be cancelled with $\check{a}_2^{m,n}$ except for
\begin{equation}\label{eq:b_2}
\check{b}_2^{2,3},\check{b}_2^{4,5},\dots,\check{b}_2^{2r-2,2r-1}.
\end{equation}
Using the fact that $\check{a}_3^{m,n}=0$ and $b_3^{m,n}=0$, $D_r$ applied to (\ref{eq:a_3}) implies that
\begin{equation}\label{eq:differential}
d_\mathcal{C}\check{b}_3^{m,n}=\check{a}_2^{\sigma^{-1}(m),\sigma^{-1}(n)}+\sum_{m<k<n}a_3^{m,k}\check{b}_3^{k,n}+\sum_{m<k<n}\check{b}_3^{m,k}a_2^{\sigma^{-1}(k),\sigma^{-1}(n)}.
\end{equation}
In the quotient dg algebra $\mathcal{C}'(\Lambda_r)$, the values of $a_3^{m,k}$ and $a_2^{\sigma^{-1}(k),\sigma^{-1}(n)}$ in the above formula have been determined in (\ref{eq:a_2q}) and (\ref{eq:a_3q}), from which we see that the generators in (\ref{eq:b_2}) can be cancelled with
\begin{equation}
\check{b}_3^{1,2},\check{b}_3^{2,4},\dots,\check{b}_3^{2r-4,2r-2}.
\end{equation}
On the other hand, using the entries of $A_3$ which are equal to 1 in $\mathcal{C}'(\Lambda_r)$ as specified by (\ref{eq:a_3q}), we can cancel most of the generators of the form $\check{b}_3^{m,n}$ with each other as follows.

Since $a_3^{1,3}=1$, by (\ref{eq:differential}) we can cancel the following generators with each other:
\begin{equation}\label{eq:cancel1}
(\check{b}_3^{1,4},\check{b}_3^{3,4}),(\check{b}_3^{1,5},\check{b}_3^{3,5}),\dots,(\check{b}_3^{1,2r},\check{b}_3^{3,2r}).
\end{equation}
It follows immediately that the only remaining generator in the first row of $\check{B}_3$ in the quotient dg algebra $\mathcal{C}'(\Lambda_r)$ is $\check{b}_3^{1,3}$. 

Similarly, since $a_3^{2,5}=a_3^{4,7}=\dots=a_3^{2r-4,2r-1}=1$, we get the following cancelling pairs of generators in $\mathcal{C}(\Lambda_r)$:
\begin{equation}
\begin{split}
&(\check{b}_3^{1,2j+3},\check{b}_3^{1,2j}),(\check{b}_3^{2,2j+3},\check{b}_3^{2,2j}),\cdots,(\check{b}_3^{2j-1,2j+3},\check{b}_3^{2j-1,2j}),\\
&(\check{b}_3^{2j,2j+4},\check{b}_3^{2j+3,2j+4}),(\check{b}_3^{2j,2j+5},\check{b}_3^{2j+3,2j+5}),\cdots,(\check{b}_3^{2j,2r},\check{b}_3^{2j+3,2r}),
\end{split}
\end{equation}
where $j\leq r-2$. For any integer $k$ with $1\leq k\leq r-1$, by considering the cancellations between $\check{b}_3^{2k,2j+3}$ and $\check{b}_3^{2k,2j}$ for $k+1\leq j\leq r-2$, we see that in the $2k$-th row of $\check{B}_3$, only
\begin{equation}
\check{b}_3^{2k,2k+1},\check{b}_3^{2k,2k+3},\check{b}_3^{2k,2r-2},\check{b}_3^{2k,2r}
\end{equation}
remain in $\mathcal{C}'(\Lambda_r)$. On the other hand, the cancellation pairs $(\check{b}_3^{2k-1,2j+3},\check{b}_3^{2k-1,2j})$ in the above list shows that the only remaining generators in the $(2k-1)$-th row of $\check{B}_3$ are given by
\begin{equation}
\check{b}_3^{2k-1,2k+1},\check{b}_3^{2k-1,2r}.
\end{equation}

Finally, using the fact that $a_3^{2r-2,2r}=1$, it is not hard to see that we can further cancel the generators
\begin{equation}\label{eq:cancel2}
(\check{b}_3^{2k,2r-2},\check{b}_3^{2k,2r}),(\check{b}_3^{2k-1,2k+1},\check{b}_3^{2k-1,2r}),
\end{equation}
which completes the proof.
\end{proof}
To compute the gradings of the generators listed in (\ref{eq:genAn}), we need to specify a Maslov potential on $\Lambda_r$. In our case, it suffices to define a Maslov potential $\mu_r:K_r\rightarrow\mathbb{Z}$ on the Legendrian arc $K_r$, and then apply the grading formula in Definition \ref{definition:susp}. Denote by $S_1,\dots,S_{2r}$ the strands of $K_r$ above the 1-cell $e_3^1$, we can define $\mu_r$ by setting
\begin{equation}\label{eq:MasAn}
\mu_n(S_{2r})=0,\mu_n(S_{2r-1})=\mu_n(S_{2r-2})=1,\dots,\mu_n(S_3)=\mu_n(S_2)=n-1,\mu(S_1)=n.
\end{equation}
This implies that
\begin{equation}
|b_2^{2,3}|=|b_2^{4,5}|=\dots=|b_2^{2r-2,2r-1}|=0,
\end{equation}
\begin{equation}
|\check{b}_3^{2,3}|=|\check{b}_3^{4,5}|=\dots=|\check{b}_3^{2r-2,2r-1}|=-1,
\end{equation}
and
\begin{equation}
|\check{b}_3^{2,3}|=|\check{b}_3^{4,5}|=\dots=|\check{b}_3^{2r-2,2r-1}|=-2.
\end{equation}

It remains to compute the differentials $d_\mathcal{C}$ on the generators of $\mathcal{C}'(\Lambda_r)$. From the cancellation procedure of the generators in $\mathcal{C}(K_r)$ we can extract the following:
\begin{equation}
d_\mathcal{C}b_2^{2,3}=d_\mathcal{C}b_2^{4,5}=\dots=d_\mathcal{C}b_2^{2r-2,2r-1}=0.
\end{equation}
Using (\ref{eq:differential}) it is easy to deduce that
\begin{equation}
d_\mathcal{C}\check{b}_3^{2,3}=d_\mathcal{C}\check{b}_3^{4,5}=\dots=d_\mathcal{C}\check{b}_3^{2r-2,2r-1}=0,
\end{equation}
and
\begin{equation}
d_\mathcal{C}\check{b}_3^{1,3}=b_2^{2,3}\check{b}_3^{2,3},d_\mathcal{C}\check{b}_3^{2r-2,2r}=\check{b}_3^{2r-2,2r-1}b_2^{2r-2,2r-1},
\end{equation}
\begin{equation}
d_\mathcal{C}\check{b}_3^{2j,2j+3}=\check{b}_3^{2j,2j+1}b_2^{2j,2j+1}+b_2^{2j+2,2j+3}\check{b}_3^{2j+2,2j+3},
\end{equation}
for any $j$ with $1\leq j\leq r-2$.
\bigskip

The computations above can be summarized as follows:
\begin{proposition}
The cellular dg algebra $\mathcal{C}(\Lambda_r)$ is quasi-isomorphic to the Ginzburg dg algebra associated to the $A_r$ quiver depicted in Figure \ref{fig:An}.
\end{proposition}

\begin{figure}
	\centering
	\begin{tikzpicture}
	\node[circle,draw, fill, minimum size = 2pt,inner sep=1pt] at (0,0) {};
	\node[circle,draw, fill, minimum size = 2pt,inner sep=1pt] at (8,0) {};
	\node[circle,draw, fill, minimum size = 2pt,inner sep=1pt] at (1.5,0) {};
	\node[circle,draw, fill, minimum size = 2pt,inner sep=1pt] at (3,0) {};
	
	\draw[->,shorten >=8pt, shorten <=8pt] (0,0) to (1.5,0);
	\draw[->,shorten >=8pt, shorten <=8pt] (1.5,0) to (3,0);
	
	\path (3,0) to node {\dots} (8,0);
	\node [shape=circle,minimum size=2pt, inner sep=1pt] at (4.5,0) {};
	\draw[->,shorten >=8pt, shorten <=8pt] (3,0) to (4.5,0);
	
	\node [shape=circle,minimum size=2pt, inner sep=1pt] at (6.5,0) {};
	\draw[->,shorten >=8pt, shorten <=8pt] (6.5,0) to (8,0);
	\end{tikzpicture}
	\caption{$A_r$ quiver}
	\label{fig:An}
\end{figure}
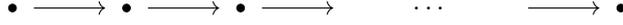

\paragraph{Remark} The orientation data of the holomorphic polygons involved in the definition of the Chekanov-Eliashberg algebra $\mathcal{C}E(\Lambda_r)$ can be recovered by applying directly the Koszul duality result of Ekholm-Lekili $\cite{ekl}$, or by identifying the differentials in $\mathcal{C}'(\Lambda_r)$ with the enumerations of Morse flow trees, and then applying the machinery recalled in Section \ref{section:subo}.

\subsection{The link $\Lambda_{2,2,2}$}\label{section:computation}

This subsection is devoted to the computation of the cellular dg algebra of the link of Legendrian surfaces $\Lambda_{2,2,2}\subset J^1(\mathbb{R}^2)$, whose front projection $p_{x,z}(\Lambda_{2,2,2})$ is described by Figure \ref{fig:front-without-2-handles}. Note that it is not hard to get from Figure \ref{fig:front-without-2-handles} the Legendrian front of the general link $\Lambda_{p,q,r}$, by replacing the standard unknot $\Lambda_P$ with an $A_{p-1}$-chain of parallel unknots, $\Lambda_Q$ with an $A_{q-1}$-chain of unknots, and $\Lambda_R$ with an $A_{r-1}$-chain of unknots. As we have remarked, in order to compute $\mathcal{C}(\Lambda_{p,q,r})$, it suffices to compute $\mathcal{C}(\Lambda_{2,2,2})$, and then combine the computations here with that in Section \ref{section:example}.\bigskip

A $\Lambda_{2,2,2}$-compatible polygonal decomposition of $p_x(\Lambda_{2,2,2})$ associated to the Legendrian front depicted in Figure \ref{fig:front-without-2-handles} is given in Figure \ref{figure:decomposition(2,2,2)}. In that figure, the 1-cells associated to cusp edges are represented by solid curves, and the 1-cells associated to crossing arcs are represented by dashed arcs. As a convention, all the cells $e^0_i,e^1_j$ and $e^2_k$ in the polygonal decomposition are labelled by their associated matrices of generators $A_i,B_j$ and $C_k$, respectively. Note that above the largest solid circle in Figure \ref{figure:decomposition(2,2,2)}, one can find only the cusp edges of the components $\Lambda_A,\Lambda_B$ and $\Lambda_R$, so there is no generator in $\mathcal{C}(\Lambda_{2,2,2})$ associated to it. In particular, the Legendrian front we are using for $\Lambda_{2,2,2}$ is non-generic, and such a choice is justified by our discussions in Section \ref{section:non-gen}. Besides the projection $\Sigma_{2,2,2}\subset\mathbb{R}^2$ of the set of singularities of $p_{x,z}(\Lambda_{2,2,2})$, we have added in the cellular decomposition the 0-cells labelled by the matrices
\begin{equation}
A_1,\dots,A_{12};
\end{equation}
and the 1-cells labelled by the matrices
\begin{equation}
B_2,\dots,B_7,B_9,\dots,B_{12},
\end{equation}
so that the non-simply-connected regions in $p_x(\Lambda_{2,2,2})$ are divided into polygons after these cells are added. All the other 1-cells correspond to singular loci in the Legendrian front of $\Lambda_{2,2,2}$, and their geometric meanings are recorded in the following table.
\begin{center}
	\begin{tabular}{|c|c|} 
		\hline
		Labelling & Front above the 1-cell  \\ 
		\hline
		$B_1,B_8$ & $p_{x,z}(\Lambda_A)\cap p_{x,z}(\Lambda_B)$ \\
		\hline
		$B_{13}$ & $p_{x,z}(\Lambda_B)\cap p_{x,z}(\Lambda_P)$  \\ 
		\hline
		$B_{14}$ & $p_{x,z}(\Lambda_A)\cap p_{x,z}(\Lambda_P)$ \\
		\hline
		$B_{15}$ & cusp edge of $\Lambda_P$ \\
		\hline
		$B_{16}$ & $p_{x,z}(\Lambda_B)\cap p_{x,z}(\Lambda_Q)$ \\
		\hline
		$B_{17}$ & $p_{x,z}(\Lambda_A)\cap p_{x,z}(\Lambda_Q)$ \\
		\hline
		$B_{18}$ & cusp edge of $\Lambda_Q$ \\
		\hline
		$B_{19},B_{20}$ & $p_{x,z}(\Lambda_B)\cap p_{x,z}(\Lambda_R)$ \\
		\hline
		$B_{21}$ & $p_{x,z}(\Lambda_A)\cap p_{x,z}(\Lambda_R)$ \\
		\hline
	\end{tabular}
\end{center}
\bigskip

\begin{figure}[htb!]
	\centering
	\begin{tikzpicture}[scale=2]
	\draw [dash dot] [decoration={markings, mark=at position 0.25 with {\arrow{<}}},postaction={decorate}] (-0.25,0) circle [radius=0.25];
	\draw [dash dot] [decoration={markings, mark=at position 0.25 with {\arrow{>}}},postaction={decorate}] (3,0) circle [radius=0.25];
	\draw [dash dot] [decoration={markings, mark=at position 0.25 with {\arrow{<}}},postaction={decorate}] (-0.25,0) circle [radius=0.5];
	\draw [dash dot] [decoration={markings, mark=at position 0.25 with {\arrow{>}}},postaction={decorate}] (3,0) circle [radius=0.5];
	\draw [dash dot] [decoration={markings, mark=at position 0.25 with {\arrow{<}}},postaction={decorate}] (-0.25,0) circle [radius=0.75];
	\draw [dash dot] [decoration={markings, mark=at position 0.25 with {\arrow{>}}},postaction={decorate}] (3,0) circle [radius=0.75];
	\draw [decoration={markings, mark=at position 0.25 with {\arrow{<}}},postaction={decorate}] (-0.25,0) circle [radius=1];
	\draw [decoration={markings, mark=at position 0.25 with {\arrow{>}}},postaction={decorate}] (3,0) circle [radius=1];
	\draw [dash dot] (1.375,0)
	[decoration={markings, mark=at position 0.25 with {\arrow{<}}},postaction={decorate}]
	[decoration={markings, mark=at position 0.75 with {\arrow{<}}},postaction={decorate}]
	circle [radius=2.875];
	\draw [dash dot] [decoration={markings, mark=at position 0.75 with {\arrow{<}}},postaction={decorate}] (1.375,0) circle [radius=3.125];
	\draw (1.375,0) circle [radius=3.375];
	
	\draw [blue,dash dot] (1.375,0) circle [radius=0.25];
	
	\draw [orange,dash dot]
	[decoration={markings, mark=at position 1/12 with {\arrow{>}}},postaction={decorate}]
	[decoration={markings, mark=at position 3/12 with {\arrow{>}}},postaction={decorate}]
	[decoration={markings, mark=at position 5/12 with {\arrow{>}}},postaction={decorate}]
	[decoration={markings, mark=at position 7/12 with {\arrow{>}}},postaction={decorate}]
	[decoration={markings, mark=at position 9/12 with {\arrow{>}}},postaction={decorate}]
	[decoration={markings, mark=at position 11/12 with {\arrow{>}}},postaction={decorate}]
	(-2,0)--(-0.5,0);
	
	\draw [blue,dash dot] (1.375,-0.25) to (1.375,-2.875);
	
	\draw [orange,dash dot]
	[decoration={markings, mark=at position 1/8 with {\arrow{<}}},postaction={decorate}]
	[decoration={markings, mark=at position 3/8 with {\arrow{<}}},postaction={decorate}]
	[decoration={markings, mark=at position 5/8 with {\arrow{<}}},postaction={decorate}]
	[decoration={markings, mark=at position 7/8 with {\arrow{<}}},postaction={decorate}]
	(3.25,0)--(4.25,0);
	
	\node [orange] at (-2.1,-0.1) {\tiny $A_1$};
	\node [orange] at (-1.85,-0.1) {\tiny $A_2$};
	\node [orange] at (-1.6,-0.1) {\tiny $A_3$};
	\node [orange] at (-1.35,-0.1) {\tiny $A_4$};
	\node [orange] at (-1.1,-0.1) {\tiny $A_5$};
	\node [orange] at (-0.85,-0.1) {\tiny $A_6$};
	\node [orange] at (-0.6,-0.1) {\tiny $A_7$};
	\node [orange] at (4.35,-0.1) {\tiny $A_8$};
	\node [orange] at (4.1,-0.1) {\tiny $A_9$};
	\node [orange] at (3.85,-0.1) {\tiny $A_{10}$};
	\node [orange] at (3.6,-0.1) {\tiny $A_{11}$};
	\node [orange] at (3.35,-0.1) {\tiny $A_{12}$};
	\node [blue] at (1.5,-2.8) {\tiny $A_{13}$};
	\node [blue] at (1.5,-0.35) {\tiny $A_{14}$};
	
	\node at (-0.25,-0.35) {\tiny $B_1$};
	\node [orange] at (-1.9,0.1) {\tiny $B_2$};
	\node [orange] at (-1.65,0.1) {\tiny $B_3$};
	\node [orange] at (-1.4,0.1) {\tiny $B_4$};
	\node [orange] at (-1.15,0.1) {\tiny $B_5$};
	\node [orange] at (-0.9,0.1) {\tiny $B_6$};
	\node [orange] at (-0.65,0.1) {\tiny $B_7$};
	\node at (3,-0.35) {\tiny $B_8$};
	\node [orange] at (4.1,0.1) {\tiny $B_9$};
	\node [orange] at (3.85,0.1) {\tiny $B_{10}$};
	\node [orange] at (3.6,0.1) {\tiny $B_{11}$};
	\node [orange] at (3.35,0.1) {\tiny $B_{12}$};
	\node at (-0.25,-0.6) {\tiny $B_{13}$};
	\node at (-0.25,-0.85) {\tiny $B_{14}$};
	\node at (-0.25,-1.1) {\tiny $B_{15}$};
	\node at (3,-0.6) {\tiny $B_{16}$};
	\node at (3,-0.85) {\tiny $B_{17}$};
	\node at (3,-1.1) {\tiny $B_{18}$};
	\node at (1.375,2.975) {\tiny $B_{19}$};
	\node at (1.375,-2.975) {\tiny $B_{20}$};
	\node at (1.375,-3.225) {\tiny $B_{21}$};
	\node [blue] at (1.5,-1.7) {\tiny $B_{22}$};
	\node [blue] at (1.375,0.35) {\tiny $B_{23}$};
	
	\node at (2.9,2.9) {\tiny $C_1$};
	\node at (2.7,2.7) {\tiny $C_2$};
	\node at (1.375,2) {\tiny $C_3$};
	\node at (-0.25,0.85) {\tiny $C_4$};
	\node at (-0.25,0.6) {\tiny $C_5$};
	\node at (-0.25,0.35) {\tiny $C_6$};
	\node at (-0.25,0) {\tiny $C_7$};
	\node at (3,0.85) {\tiny $C_8$};
	\node at (3,0.6) {\tiny $C_9$};
	\node at (3,0.35) {\tiny $C_{10}$};
	\node at (3,0) {\tiny $C_{11}$};
	\node [blue] at (1.375,0) {\tiny $C_{12}$};
	\end{tikzpicture}
	\caption{Cellular decomposition associated to $\Lambda_{2,2,2}$ (coloured in black and orange) and $\Lambda_{p,q,r}$ (with additional cells coloured in blue)}
	\label{figure:decomposition(2,2,2)}
\end{figure}
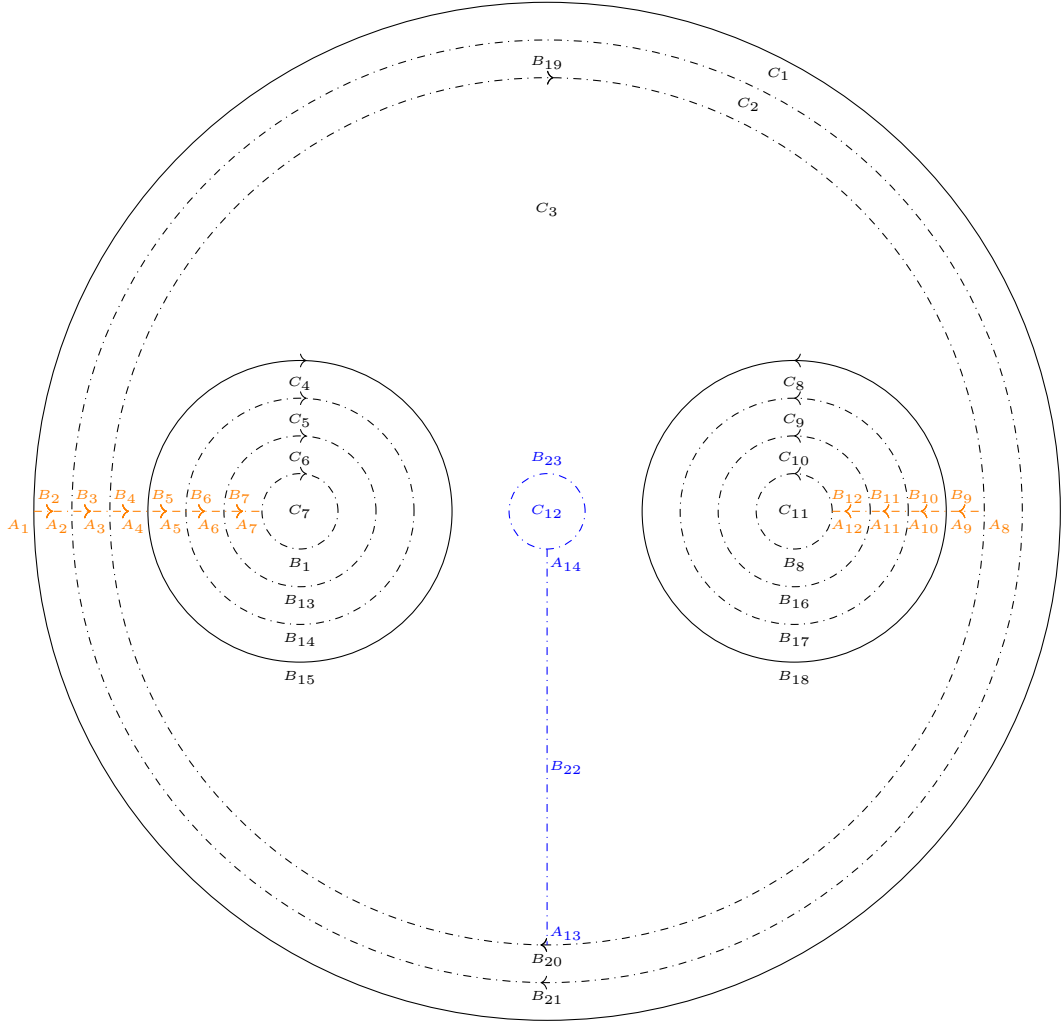

Denote by $J=\{1,\dots,21\}$ the set of subscripts of the matrices $B_j$ associated to the 1-cells in the cellular decomposition of $p_x(\Lambda_{2,2,2})$ specified above. Let
\begin{equation}
J_1=\{2,\dots,7,9,\dots,12,19,20\}\subset J.
\end{equation}
For any 1-cell $e^1_j$ with $j\in J_1$, it has two distinct endpoints, which are labelled as the 0-cells $e_+^0$ and $e_-^0$, by (\ref{eq:diff2}) we have
\begin{equation}\label{eq:cancel3}
d_\mathcal{C}b_j^{m,n}=a_+^{m,n}+a_-^{\sigma_j(m),\sigma_j(n)}+\sum_{m<k<n}a_+^{m,k}b_j^{k,n}+\sum_{m<k<n}b_j^{m,k}a_-^{\sigma_j(k),\sigma_j(n)},
\end{equation}
where the permutation $\sigma_j$ is a transposition determined by the crossing arc of $p_{x,z}(\Lambda_{2,2,2})$ above the 0-cell $e_-^0$.

As we have already seen in Section \ref{section:example}, one can cancel the generators $b_j^{m,n}$ with $a_j^{m,n}$ using the formulas in (\ref{eq:cancel3}). In particular, we see that the remaining generators in $B_j$ with $j\in J_1$ in the quotient dg algebra $\mathcal{C}'(\Lambda_{2,2,2})$ are
\begin{eqnarray}\label{eq:remain}
b_2^{4,5},b_3^{2,3},b_5^{5,6},b_6^{3,4},b_7^{4,5},b_{10}^{5,6},b_{11}^{3,4},b_{12}^{4,5},
\end{eqnarray}
together with all the generators in the strictly upper triangular matrix $B_{20}$ except for $b_{20}^{2,3}$, which is equal to 0 by definition.

Let $J_2=J\setminus J_1$. The 1-cells $e_j^1$ with $j\in J_2$ have the same 0-cell as their initial and terminal point, whose associated matrix of generators will be denoted by $A_{t_j}$. By (\ref{eq:diff2}), the differentials of generators in $B_j$ are given by
\begin{equation}\label{eq:B-circ}
d_\mathcal{C}b_j^{m,n}=\sum_{m<k<n}a_{t_j}^{m,k}b_j^{k,n}+\sum_{m<k<n}b_j^{m,k}a_{t_j}^{k,n}.
\end{equation}
In particular, these generators are not cancelled with any generators associated to 0-cells.
\bigskip

For the index set $I$ of the 2-cells in the cellular decomposition of $\Lambda_{2,2,2}$, denote by $I_1\subset I$ the subset
\begin{equation}
I_1=\{1,4,5,6,8,9,10\}.
\end{equation}
For any $i\in I_1$, the corresponding 2-cell $e_i^2$ is an annulus bounded by two circles, whose inner circle has generators assembled in the matrix $B_{s_i}$ and whose outer circle is labelled by the matrix $B_{t_i}$, together with an additional cutting edge $e_{r_i}^1$, whose endpoints have associated matrices of generators $A_{i,-}$ and $A_{i,+}$, see Figure \ref{fig:annulus}. By (\ref{eq:c}) we can write down the formulas of their differentials:
\begin{equation}\label{eq:ann}
d_\mathcal{C}C_i=A_{i,+}C_i+C_iQ_{\sigma_{r_i}}A_{i,-}Q_{\sigma_{r_i}}+B_{s_i}(E+B_{r_i})+(E+B_{r_i})Q_{\sigma_{r_i}}B_{t_i}Q_{\sigma_{r_i}},
\end{equation}
where
\begin{equation}
Q_{\sigma_{r_i}}=\sum_m\Delta_{\sigma_{r_i}(m),m}
\end{equation}
is the permutation matrix. Note that when $i=1$, $B_{t_1}=0$ in the above formula. In particular, (\ref{eq:ann}) shows that all the generators in $C_i$ with $i\in I_1$ can be cancelled with that in $B_{s_i}$, except for
\begin{equation}
c_1^{4,5},c_4^{5,6},c_5^{3,4},c_6^{4,5},c_8^{5,6},c_9^{3,4},c_{10}^{4,5}.
\end{equation}

\begin{figure}[htb!]
	\centering
	\begin{tikzpicture}
	\draw [decoration={markings, mark=at position 3/4 with {\arrow{<}}},postaction={decorate}] (0,0) circle [radius=0.8];
	\draw [decoration={markings, mark=at position 3/4 with {\arrow{<}}},postaction={decorate}] (0,0) circle [radius=2];
	\draw [decoration={markings, mark=at position 1/2 with {\arrow{>}}},postaction={decorate}] (-2,0) to (-0.8,0);
	\node at (1.4,0) {\footnotesize $C_i$};
	\node at (-1.4,0.15) {\footnotesize $B_{r_i}$};
	\node at (-2.3,0) {\footnotesize $A_{i,-}$};
	\node at (-0.5,0) {\footnotesize $A_{i,+}$};
	\node at (0,0.95) {\footnotesize $B_{s_i}$};
	\node at (0,2.15) {\footnotesize $B_{t_i}$};
	\end{tikzpicture}
	\caption{A type $I_1$ 2-cell labelled by $C_i$}
	\label{fig:annulus}
\end{figure}
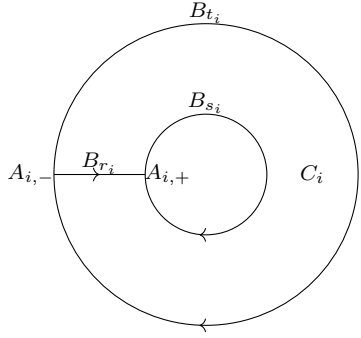

Similarly, the 2-cell $e_2^2$ is also an annulus, with the cutting edge $e_3^1$. The only difference is that the small circle which bounds the annulus $e_2^2$ now consists of two 1-cells, namely $e_{19}^1$ and $e_{20}^1$, and the large circle has the associated 1-cell $e_{21}^1$. As a consequence,
\begin{equation}
d_\mathcal{C}C_2=A_3C_2+C_2Q_{\sigma_3}A_2Q_{\sigma_3}+(B_{19}+B_{20}+B_{20}B_{19})(E+B_3)+(E+B_3)Q_{\sigma_3}B_{21}Q_{\sigma_3}.
\end{equation}
Since $B_{19}=0$ after cancelling it with $A_8$, the above formula simplifies to
\begin{equation}\label{eq:C2}
d_\mathcal{C}C_2=A_3C_2+C_2Q_{\sigma_3}A_2Q_{\sigma_3}+B_{20}(E+B_3)+(E+B_3)Q_{\sigma_3}B_{21}Q_{\sigma_3}
\end{equation}
in $\mathcal{C}'(\Lambda_{2,2,2})$. Since $b_{20}^{2,3}=0$, we see that the generator $c_2^{2,3}$ is the only generator in the matrix $C_2$ that cannot be cancelled via the above formula.

On the generators in $C_7$ and $C_{11}$, differential $d_\mathcal{C}$ takes the form
\begin{equation}\label{eq:C7}
d_\mathcal{C}C_7=Q_0A_7Q_0C_7+C_7Q_0A_7Q_0+Q_0B_1Q_0
\end{equation}
and
\begin{equation}
d_\mathcal{C}C_{11}=Q_0A_{12}Q_0C_{11}+C_{11}Q_0A_{12}Q_0+Q_0B_8Q_0,
\end{equation}
where $Q_0$ is the permutation matrix associated to the transposition $\sigma_0=(4,5)$.

For the remaining 2-cell labelled by $e_3^2$, its boundary consists of the 1-cells
\begin{equation}
e_4^1,e_9^1,e_{15}^1,e_{18}^1,e_{19}^1,e_{20}^1.
\end{equation}
Choosing $e_4^0$ as the initial vertex, and $e_9^0$ as the terminal vertex, we have
\begin{equation}
\begin{split}
d_\mathcal{C}C_3&=A_9C_3+C_3A_4+(E+B_9)(E+Q_{\sigma_4}B_{19}Q_{\sigma_4})(E+B_4)^{-1}(E+B_{15})\\
&+(E+B_{18})(E+B_9)(E+Q_{\sigma_4}B_{20}Q_{\sigma_4})^{-1}(E+B_4)^{-1}.
\end{split}
\end{equation}
Since $B_4=B_9=B_{19}=0$ after cancelling their generators with $A_4,A_9$ and $A_8$, the above formula simplifies to
\begin{equation}\label{eq:C3}
d_\mathcal{C}C_3=A_9C_3+C_3A_4+E+B_{15}+(E+B_{18})(E+Q_{\sigma_4}B_{20}Q_{\sigma_4})^{-1}.
\end{equation}
In particular, all the generators in $C_3$ are cancelled with that of $B_{15}$.
\bigskip

\begin{proposition}\label{proposition:222}
	The cellular dg algebra $\mathcal{C}(\Lambda_{2,2,2})$ is quasi-isomorphic to a dg algebra $\mathcal{C}'(\Lambda_{2,2,2})$ generated by
	\begin{equation}
	\begin{split}
	&b_2^{4,5},b_3^{2,3},b_5^{5,6},b_6^{3,4},b_7^{4,5},b_{10}^{5,6},b_{11}^{3,4},b_{12}^{4,5};\\
	&c_7^{4,5},c_{11}^{4,5},c_7^{3,5},c_{11}^{3,5},c_7^{2,5}+c_{11}^{2,5},c_7^{4,6},c_{11}^{4,6},c_7^{4,7}+c_{11}^{4,7};\\
	&c_7^{4,8}+c_{11}^{4,8},c_7^{1,5}+c_{11}^{1,5},c_7^{3,6},c_{11}^{3,6},c_7^{2,7}+c_{11}^{2,7},
	\end{split}
	\end{equation}
	with gradings
	\begin{equation}\label{eq:grading}
	\begin{split}
	&|b_2^{4,5}|=|b_3^{2,3}|=|b_5^{5,6}|=|b_6^{3,4}|=|c_7^{4,5}|=|b_{10}^{5,6}|=|b_{11}^{3,4}|=|c_{11}^{4,5}|=0,\\
	&|b_7^{4,5}|=|b_{12}^{4,5}|=|c_7^{3,5}|=|c_{11}^{3,5}|=|c_7^{2,5}+c_{11}^{2,5}|=|c_7^{4,6}|=|c_{11}^{4,6}|=|c_7^{4,7}+c_{11}^{4,7}|=-1,\\
	&|c_7^{4,8}+c_{11}^{4,8}|=|c_7^{1,5}+c_{11}^{1,5}|=|c_7^{3,6}|=|c_{11}^{3,6}|=|c_7^{2,7}+c_{11}^{2,7}|=-2,
	\end{split}
	\end{equation}
	and differentials
	\begin{equation}
	d_\mathcal{C}b_2^{4,5}=d_\mathcal{C}b_3^{2,3}=d_\mathcal{C}b_5^{5,6}=d_\mathcal{C}b_6^{3,4}=d_\mathcal{C}c_7^{4,5}=d_\mathcal{C}b_{10}^{5,6}=d_\mathcal{C}b_{11}^{3,4}=d_\mathcal{C}c_{11}^{4,5}=0,
	\end{equation}
	\begin{equation}
	d_\mathcal{C}b_7^{4,5}=b_6^{3,4}b_5^{5,6}+b_3^{2,3}b_2^{4,5},
	\end{equation}
	\begin{equation}
	d_\mathcal{C}b_{12}^{4,5}=b_{11}^{3,4}b_{10}^{5,6}+b_3^{2,3}b_2^{4,5},
	\end{equation}
	\begin{equation}
	d_\mathcal{C}c_7^{3,5}=b_5^{5,6}c_7^{4,5},
	\end{equation}
	\begin{equation}
	d_\mathcal{C}c_{11}^{3,5}=b_{10}^{5,6}c_{11}^{4,5},
	\end{equation}
	\begin{equation}
	d_\mathcal{C}(c_7^{2,5}+c_{11}^{2,5})=b_2^{4,5}c_{11}^{4,5}+b_2^{4,5}c_7^{4,5},
	\end{equation}
	\begin{equation}
	d_\mathcal{C}c_7^{4,6}=c_7^{4,5}b_6^{3,4},
	\end{equation}
	\begin{equation}
	d_\mathcal{C}c_{11}^{4,6}=c_{11}^{4,5}b_{11}^{3,4},
	\end{equation}
	\begin{equation}
	d_\mathcal{C}(c_7^{4,7}+c_{11}^{4,7})=c_{11}^{4,5}b_3^{2,3}+c_7^{4,5}b_3^{2,3},
	\end{equation}
	\begin{equation}
	d_\mathcal{C}(c_7^{4,8}+c_{11}^{4,8})=c_{11}^{4,5}b_{12}^{4,5}+c_7^{4,5}b_7^{4,5}+c_7^{4,6}b_5^{5,6}+c_{11}^{4,6}b_{10}^{5,6}+(c_7^{4,7}+c_{11}^{4,7})b_2^{4,5},
	\end{equation}
	\begin{equation}
	d_\mathcal{C}(c_7^{1,5}+c_{11}^{1,5})=b_6^{3,4}c_7^{3,5}+b_{11}^{3,4}c_{11}^{3,5}+b_3^{2,3}(c_7^{2,5}+c_{11}^{2,5})+b_{12}^{4,5}c_{11}^{4,5}+b_7^{4,5}c_7^{4,5},
	\end{equation}
	\begin{equation}
	d_\mathcal{C}c_7^{3,6}=b_5^{5,6}c_7^{4,6}+c_7^{3,5}b_6^{3,4},
	\end{equation}
	\begin{equation}
	d_\mathcal{C}c_{11}^{3,6}=b_{10}^{5,6}c_{11}^{4,6}+c_{11}^{3,5}b_{11}^{3,4},
	\end{equation}
	\begin{equation}
	d_\mathcal{C}(c_7^{2,7}+c_{11}^{2,7})=b_2^{4,5}(c_7^{4,7}+c_{11}^{4,7})+(c_7^{2,5}+c_{11}^{2,5})b_3^{2,3}.
	\end{equation}
\end{proposition}
\begin{proof}
	Turning back to the differentials $d_\mathcal{C}b_j^{m,n}$ for $j\in J_1$, we need to record the values of $A_i=(a_i^{m,n})$ in the quotient dg algebra $\mathcal{C}'(\Lambda_{2,2,2})$ after the cancellation process between generators.
	
	Since above the 0-cell $e_1^0$ are the cusp edges of the Legendrian fronts of the components $\Lambda_A$, $\Lambda_B$ and $\Lambda_R$ of the link $\Lambda_{2,2,2}$, by (\ref{eq:cusps}), $A_1$ is a block diagonal matrix with three $N$ blocks on its diagonal, i.e.
	\begin{equation}
	A_1=\begin{bmatrix}
	0 & 1 & 0 & 0 & 0 & 0 \\
	0 & 0 & 0 & 0 & 0 & 0 \\
	0 & 0 & 0 & 1 & 0 & 0 \\
	0 & 0 & 0 & 0 & 0 & 0 \\
	0 & 0 & 0 & 0 & 0 & 1 \\
	0 & 0 & 0 & 0 & 0 & 0
	\end{bmatrix}
	\end{equation}
	After cancellation, $b_2^{4,5}$ is the only remaining generator in $B_2$, from this fact and (\ref{eq:cancel3}) we get
	\begin{equation}\label{eq:A2}
	A_2=\begin{bmatrix}
	0 & 1 & 0 & 0 & 0 & 0 \\
	0 & 0 & 0 & 0 & 0 & 0 \\
	0 & 0 & 0 & 1 & b_2^{4,5} & 0 \\
	0 & 0 & 0 & 0 & 0 & b_2^{4,5} \\
	0 & 0 & 0 & 0 & 0 & 1 \\
	0 & 0 & 0 & 0 & 0 & 0
	\end{bmatrix}
	\end{equation}
	The only remaining generator in $B_3$ after cancelling it with $A_3$ is $b_3^{2,3}$, combining with $\sigma_3=(4,5)$ we obtain
	\begin{equation}\label{eq:A3}
	A_3=\begin{bmatrix}
	0 & 1 & b_3^{2,3} & 0 & 0 & 0 \\
	0 & 0 & 0 & b_3^{2,3}b_2^{4,5} & b_3^{2,3} & 0 \\
	0 & 0 & 0 & b_2^{4,5} & 1 & 0 \\
	0 & 0 & 0 & 0 & 0 & 1 \\
	0 & 0 & 0 & 0 & 0 & b_2^{4,5} \\
	0 & 0 & 0 & 0 & 0 & 0
	\end{bmatrix}
	\end{equation}
	$B_4=0$ after cancelling it with the generators in $A_4$. Since there is a cusp edge of $p_{x,z}(\Lambda_P)$ above the 0-cell $e^0_5$, in addition to applying the transposition $\sigma_4=(2,3)$, we need to add an $N$ block in the fourth and fifth rows in order to obtain $A_4$ in $\mathcal{C}'(\Lambda_{2,2,2})$. This gives:
	\begin{equation}
	A_4=\begin{bmatrix}
	0 & b_3^{2,3} & 1 & 0 & 0 & 0 & 0 & 0 \\
	0 & 0 & 0 & 0 & 0 & b_2^{4,5} & 1 & 0 \\
	0 & 0 & 0 & 0 & 0 & b_3^{2,3}b_2^{4,5} & b_3^{2,3} & 0 \\
	0 & 0 & 0 & 0 & 1 & 0 & 0 & 0 \\
	0 & 0 & 0 & 0 & 0 & 0 & 0 & 0 \\
	0 & 0 & 0 & 0 & 0 & 0 & 0 & 1 \\
	0 & 0 & 0 & 0 & 0 & 0 & 0 & b_2^{4,5} \\
	0 & 0 & 0 & 0 & 0 & 0 & 0 & 0
	\end{bmatrix}
	\end{equation}
	To determine the value of $A_5$ in $\mathcal{C}'(\Lambda_{2,2,2})$, we notice that the $b_5^{5,6}$ remains after cancelling $B_5$ with $A_5$:
	\begin{equation}
	A_5=\begin{bmatrix}
	0 & b_3^{2,3} & 1 & 0 & 0 & 0 & 0 & 0 \\
	0 & 0 & 0 & 0 & 0 & b_2^{4,5} & 1 & 0 \\
	0 & 0 & 0 & 0 & 0 & b_3^{2,3}b_2^{4,5} & b_3^{2,3} & 0 \\
	0 & 0 & 0 & 0 & 1 & b_5^{5,6} & 0 & 0 \\
	0 & 0 & 0 & 0 & 0 & 0 & 0 & b_5^{5,6} \\
	0 & 0 & 0 & 0 & 0 & 0 & 0 & 1 \\
	0 & 0 & 0 & 0 & 0 & 0 & 0 & b_2^{4,5} \\
	0 & 0 & 0 & 0 & 0 & 0 & 0 & 0
	\end{bmatrix}
	\end{equation}
	$b_6^{3,4}$ is the only generator in the matrix $B_6$ that is not cancelled with the generators in $A_6$. As $\sigma_6=(5,6)$, it enables us to compute:
	\begin{equation}
	A_6=\begin{bmatrix}
	0 & b_3^{2,3} & 1 & b_6^{3,4} & 0 & 0 & 0 & 0 \\
	0 & 0 & 0 & 0 & b_2^{4,5} & 0 & 1 & 0 \\
	0 & 0 & 0 & 0 & b_3^{2,3}b_2^{4,5}+b_6^{3,4}b_5^{5,6} & b_6^{3,4} & b_3^{2,3} & 0 \\
	0 & 0 & 0 & 0 & b_5^{5,6} & 1 & 0 & 0 \\
	0 & 0 & 0 & 0 & 0 & 0 & 0 & 1 \\
	0 & 0 & 0 & 0 & 0 & 0 & 0 & b_5^{5,6} \\
	0 & 0 & 0 & 0 & 0 & 0 & 0 & b_2^{4,5} \\
	0 & 0 & 0 & 0 & 0 & 0 & 0 & 0
	\end{bmatrix}
	\end{equation}
	Finally, note that $b_7^{4,5}$ can not be cancelled with any of the generators in $A_7$. Since $\sigma_7=(3,4)$, it follows that
	\begin{equation}
	A_7=\begin{bmatrix}
	0 & b_3^{2,3} & b_6^{3,4} & 1 & b_7^{4,5} & 0 & 0 & 0 \\
	0 & 0 & 0 & 0 & b_2^{4,5} & 0 & 1 & 0 \\
	0 & 0 & 0 & 0 & b_5^{5,6} & 1 & 0 & 0 \\
	0 & 0 & 0 & 0 & 0 & b_6^{3,4} & b_3^{2,3} & b_7^{4,5} \\
	0 & 0 & 0 & 0 & 0 & 0 & 0 & 1 \\
	0 & 0 & 0 & 0 & 0 & 0 & 0 & b_5^{5,6} \\
	0 & 0 & 0 & 0 & 0 & 0 & 0 & b_2^{4,5} \\
	0 & 0 & 0 & 0 & 0 & 0 & 0 & 0
	\end{bmatrix}
	\end{equation}
	By (\ref{eq:cancel3}) applied to the case when $j=7$, we deduce from the computations for $A_6$ and $A_7$ above that
	\begin{equation}
	\boxed{d_\mathcal{C}b_7^{4,5}=b_3^{2,3}b_2^{4,5}+b_6^{3,4}b_5^{5,6}.}
	\end{equation}
	\bigskip
	
	\textit{As a convention, we will box the formulas which contribute to the non-trivial differentials of the generators in the dg algebra} $\mathcal{C}'(\Lambda_{2,2,2})$.
	\bigskip
	
	Applying (\ref{eq:ann}) when $i=1$ implies that
	\begin{equation}
	B_{21}=A_2C_1+C_1A_1+B_{21}B_2
	\end{equation}
	holds in $\mathcal{C}'(\Lambda_{2,2,2})$. From this one deduces that
	\begin{equation}\label{eq:B21}
	B_{21}=\begin{bmatrix}
	0 & 0 & 0 & 0 & 0 & 0 \\
	0 & 0 & 0 & 0 & 0 & 0 \\
	0 & 0 & 0 & 0 & b_2^{4,5}c_1^{4,5} & 0 \\
	0 & 0 & 0 & 0 & 0 & c_1^{4,5} \\
	0 & 0 & 0 & 0 & 0 & 0 \\
	0 & 0 & 0 & 0 & 0 & 0
	\end{bmatrix}
	\end{equation}
	Since all the generators in $C_2$ except for $c_2^{2,3}$ have been cancelled in $\mathcal{C}'(\Lambda_{2,2,2})$, it follows from (\ref{eq:C2}) that
	\begin{equation}
	B_{20}=Q_{\sigma_3}B_{21}Q_{\sigma_3}+A_3C_2+C_2Q_{\sigma_3}A_2Q_{\sigma_3}+B_{20}B_3+B_3Q_{\sigma_3}B_{21}Q_{\sigma_3}.
	\end{equation}
	Combining with (\ref{eq:B21}), (\ref{eq:A2}) and (\ref{eq:A3}), one can deduce that
	\begin{equation}\label{eq:B20}
	B_{20}=\begin{bmatrix}
	0 & 0 & c_2^{2,3} & 0 & 0 & 0 \\
	0 & 0 & 0 & c_2^{2,3}b_2^{4,5}+b_3^{2,3}b_2^{4,5}c_1^{4,5} & c_2^{2,3} & 0 \\
	0 & 0 & 0 & b_2^{4,5}c_1^{4,5} & 0 & 0 \\
	0 & 0 & 0 & 0 & 0 & 0 \\
	0 & 0 & 0 & 0 & 0 & c_1^{4,5} \\
	0 & 0 & 0 & 0 & 0 & 0
	\end{bmatrix}
	\end{equation}
	It is straightforward to verify that
	\begin{equation}
	(E+Q_{\sigma_4}B_{20}Q_{\sigma_4})^{-1}=E+Q_{\sigma_4}B_{20}Q_{\sigma_4}.
	\end{equation}
	Since $C_3=0$ in $\mathcal{C}'(\Lambda_{2,2,2})$, we get from (\ref{eq:C3}) that
	\begin{equation}
	B_{15}+B_{18}=Q_{\sigma_4}B_{20}Q_{\sigma_4}+B_{18}Q_{\sigma_4}B_{20}Q_{\sigma_4}.
	\end{equation}
	When computing $d_\mathcal{C}C_4$, since the cusp edge of $\Lambda_P$ lies above the 1-cell with labelling $B_{15}$, by (\ref{eq:block}) we need to replace $B_{15}$ with an $8\times 8$ matrix $\widetilde{B}_{15}$ by adding a $2\times2$ zero block in the fourth and fifth rows and columns of $B_{15}$. By (\ref{eq:ann}) it follows that
	\begin{equation}\label{eq:B14}
	B_{14}=\widetilde{B}_{15}+A_5C_4+C_4A_4+B_{14}B_5+B_5\widetilde{B}_{15},
	\end{equation}
	Applying (\ref{eq:ann}) when $i=5,6$ implies that
	\begin{equation}\label{eq:B13}
	B_{13}=Q_{\sigma_6}B_{14}Q_{\sigma_6}+A_6C_5+C_5Q_{\sigma_6}A_5Q_{\sigma_6}+B_{13}B_6+B_6Q_{\sigma_6}B_{14}Q_{\sigma_6}
	\end{equation}
	and
	\begin{equation}\label{eq:B1}
	B_1=Q_{\sigma_7}B_{13}Q_{\sigma_7}+A_7C_6+C_6Q_{\sigma_7}A_6Q_{\sigma_7}+B_1B_7+B_7Q_{\sigma_7}B_{13}Q_{\sigma_7}
	\end{equation}
	hold in the quotient dg algebra $\mathcal{C}'(\Lambda_{2,2,2})$.
	
	From now on we compute the differentials of the generators in $C_7$ according to the filtration $F^\bullet$ on the cellular dg algebra $\mathcal{C}(\Lambda_{2,2,2})$ defined according to the increasing order of $n-m$. 
	
	For $n=m+1$, by (\ref{eq:C7}) and the cancellations made above we see that the differentials in $\mathcal{C}'(\Lambda_{2,2,2})$ are
	\begin{equation}
	d_\mathcal{C}c_7^{1,2}=b_{18}^{1,2},
	\end{equation}
	\begin{equation}
	d_\mathcal{C}c_7^{2,3}=b_{18}^{2,3}b_6^{3,4},
	\end{equation}
	\begin{equation}\label{eq:c734}
	d_\mathcal{C}c_7^{3,4}=c_4^{5,6},
	\end{equation}
	\begin{equation}
	\boxed{d_\mathcal{C}c_7^{4,5}=0,}
	\end{equation}
	\begin{equation}\label{eq:c756}
	d_\mathcal{C}c_7^{5,6}=c_5^{3,4},
	\end{equation}
	\begin{equation}
	d_\mathcal{C}c_7^{6,7}=b_5^{5,6}b_{18}^{4,5},
	\end{equation}
	\begin{equation}\label{eq:c778}
	d_\mathcal{C}c_7^{7,8}=c_1^{4,5}+b_{18}^{5,6}.
	\end{equation}
	In particular, by (\ref{eq:c734}), (\ref{eq:c756}) and (\ref{eq:c778}) the generators
	\begin{equation}
	(c_7^{3,4},c_4^{5,6}),(c_7^{5,6},c_5^{3,4}),(c_7^{7,8},c_1^{4,5})
	\end{equation}
	can be cancelled with each other, and $C_4=C_5=0$ in $\mathcal{C}'(\Lambda_{2,2,2})$. Thus one can further simplify (\ref{eq:B14}) and (\ref{eq:B13}) respectively to
	\begin{equation}
	B_{14}=\widetilde{B}_{15}+B_{14}B_5+B_5\widetilde{B}_{15}
	\end{equation}
	and
	\begin{equation}
	B_{13}=Q_{\sigma_6}B_{14}Q_{\sigma_6}+B_{13}B_6+B_6Q_{\sigma_6}B_{14}Q_{\sigma_6}.
	\end{equation}
	
	For $n=m+2$, taking into account of the above computations, one can deduce
	\begin{equation}
	d_\mathcal{C}c_7^{1,3}=b_3^{2,3}c_7^{2,3}+b_{18}^{1,3}b_6^{3,4},
	\end{equation}
	\begin{equation}
	d_\mathcal{C}c_7^{2,4}=c_7^{2,3}b_5^{5,6}+b_{18}^{2,4}+b_2^{4,5}b_{18}^{5,6}+b_{18}^{2,3}b_7^{4,5}+b_{18}^{2,3}c_2^{2,3}b_2^{4,5}+b_{18}^{2,3}b_3^{2,3}b_2^{4,5}b_{18}^{5,6},
	\end{equation}
	\begin{equation}
	\boxed{d_\mathcal{C}c_7^{3,5}=b_5^{5,6}c_7^{4,5},}
	\end{equation}
	\begin{equation}
	\boxed{d_\mathcal{C}c_7^{4,6}=c_7^{4,5}b_6^{3,4},}
	\end{equation}
	\begin{equation}\label{eq:c757}
	d_\mathcal{C}c_7^{5,7}=b_6^{3,4}c_7^{6,7}+c_2^{2,3}+b_{18}^{3,5},
	\end{equation}
	\begin{equation}
	d_\mathcal{C}c_7^{6,8}=c_7^{6,7}b_2^{4,5}+b_5^{5,6}b_{18}^{4,6}+b_5^{5,6}b_{18}^{4,5}b_{18}^{5,6}.
	\end{equation}
	Among the above formulas, (\ref{eq:c757}) implies that the generators
	\begin{equation}
	(c_7^{5,7},c_2^{2,3})
	\end{equation}
	can be cancelled with each other.
	
	For $n=m+3$, based on the above computations we have
	\begin{equation}\label{eq:c714}
	d_\mathcal{C}c_7^{1,4}=b_3^{2,3}c_7^{2,4}+c_7^{1,2}b_2^{4,5}+c_7^{1,3}b_5^{5,6}+c_6^{4,5}+\dots,
	\end{equation}
	where the ellipsis on the right-hand side above stands for some additional terms which do not play essential roles in our argument,
	\begin{equation}\label{eq:c725}
	\boxed{d_\mathcal{C}c_7^{2,5}=b_2^{4,5}c_7^{4,5}+b_{18}^{2,3},}
	\end{equation}
	\begin{equation}
	\boxed{d_\mathcal{C}c_7^{3,6}=b_5^{5,6}c_7^{4,6}+c_7^{3,5}b_6^{3,4},}
	\end{equation}
	\begin{equation}\label{eq:c747}
	\boxed{d_\mathcal{C}c_7^{4,7}=c_7^{4,5}b_3^{2,3}+b_{18}^{4,5},}
	\end{equation}
	\begin{equation}
	d_\mathcal{C}c_7^{5,8}=b_6^{3,4}c_7^{6,8}+c_7^{5,7}b_2^{4,5}+c_6^{4,5}+\dots.
	\end{equation}
	From (\ref{eq:c714}) we see that the generators
	\begin{equation}
	(c_7^{1,4},c_6^{4,5})
	\end{equation}
	can be cancelled with each other.
	
	For $n=m+4$, based on the simplifications made above we deduce
	\begin{equation}\label{eq:c715}
	\boxed{d_\mathcal{C}c_7^{1,5}=b_3^{2,3}c_7^{2,5}+b_6^{3,4}c_7^{3,5}+b_7^{4,5}c_7^{4,5}+b_{18}^{1,3},}
	\end{equation}
	\begin{equation}\label{eq:c726}
	d_\mathcal{C}c_7^{2,6}=b_2^{4,5}c_7^{4,6}+c_7^{2,3}+c_7^{2,5}b_6^{3,4}+\dots,
	\end{equation}
	\begin{equation}\label{eq:c737}
	d_\mathcal{C}c_7^{3,7}=b_5^{5,6}c_7^{4,7}+c_7^{6,7}+c_7^{3,5}b_3^{2,3}+\dots,
	\end{equation}
	\begin{equation}\label{eq:c748}
	\boxed{d_\mathcal{C}c_7^{4,8}=c_7^{4,5}b_7^{4,5}+c_7^{4,6}b_5^{5,6}+c_7^{4,7}b_2^{4,5}+b_{18}^{4,6}+b_{18}^{4,5}b_{18}^{5,6}.}
	\end{equation}
	(\ref{eq:c726}) and (\ref{eq:c737}) imply that the generators
	\begin{equation}
	(c_7^{2,6},c_7^{2,3}),(c_7^{3,7},c_7^{6,7})
	\end{equation}
	can be cancelled with each other.
	
	For $n=m+5$, we have
	\begin{equation}\label{eq:c716}
	d_\mathcal{C}c_7^{1,6}=b_6^{3,4}c_7^{3,6}+b_7^{4,5}c_7^{4,6}+c_7^{1,3}+c_7^{1,5}b_6^{3,4}+\dots,
	\end{equation}
	\begin{equation}\label{eq:c727}
	\boxed{d_\mathcal{C}c_7^{2,7}=b_2^{4,5}c_7^{4,7}+c_7^{2,5}b_3^{2,3}+b_{18}^{2,5}+b_{18}^{2,3}b_6^{3,4}c_7^{6,7}+b_{18}^{2,3}b_{18}^{3,5},}
	\end{equation}
	\begin{equation}\label{eq:c738}
	d_\mathcal{C}c_7^{3,8}=b_5^{5,6}c_7^{4,8}+c_7^{6,8}+c_7^{3,5}b_7^{4,5}+c_7^{3,6}b_5^{5,6}.
	\end{equation}
	From (\ref{eq:c716}) and (\ref{eq:c738}) we see that the generators
	\begin{equation}
	(c_7^{1,6},c_7^{1,3}),(c_7^{3,8},c_7^{6,8})
	\end{equation}
	can be cancelled in pair.
	
	For $n=m+6$, we get from the above that
	\begin{equation}
	d_\mathcal{C}c_7^{1,7}=b_3^{2,3}c_7^{2,7}+b_7^{4,5}c_7^{4,7}+c_7^{1,2}+\dots,
	\end{equation}
	\begin{equation}
	d_\mathcal{C}c_7^{2,8}=b_2^{4,5}c_7^{4,8}+c_7^{2,4}+\dots.
	\end{equation}
	This enables us to cancel the generators
	\begin{equation}
	(c_7^{1,7},c_7^{1,2}),(c_7^{2,8},c_7^{2,4})
	\end{equation}
	with each other.
	
	Finally, we have
	\begin{equation}
	d_\mathcal{C}c_7^{1,8}=b_7^{4,5}c_7^{4,8}+c_7^{5,8}+c_7^{1,5}b_7^{4,5}+\dots,
	\end{equation}
	which implies that
	\begin{equation}
	(c_7^{1,8},c_7^{5,8})
	\end{equation}
	can be cancelled with each other. We conclude from the above computations that the remaining generators of $C_7$ in the quotient dg algebra $\mathcal{C}'(\Lambda_{2,2,2})$ are
	\begin{equation}
	c_7^{1,5},c_7^{2,5},c_7^{2,7},c_7^{3,5},c_7^{3,6},c_7^{4,5},c_7^{4,6},c_7^{4,7},c_7^{4,8}.
	\end{equation}
	
	Since $B_{19}=0$, by (\ref{eq:cancel3}) applied to $j=19$, we deduce that $A_3=A_8$ in in $\mathcal{C}'(\Lambda_{2,2,2})$. This enables us to compute
	\begin{equation}
	\boxed{d_\mathcal{C}b_{11}^{4,5}=b_3^{2,3}b_2^{4,5}+b_{11}^{3,4}b_{10}^{5,6},}
	\end{equation}
	together the differentials of the generators in $C_{11}$. The latter computation is in some sense symmetric to the computation of $d_\mathcal{C}c_7^{m,n}$. More explicitly, we have
	\begin{equation}\label{eq:c1134}
	d_\mathcal{C}c_{11}^{3,4}=c_8^{5,6},
	\end{equation}
	\begin{equation}
	\boxed{d_\mathcal{C}c_{11}^{4,5}=0,}
	\end{equation}
	\begin{equation}\label{eq:c1156}
	d_\mathcal{C}c_{11}^{5,6}=c_9^{3,4},
	\end{equation}
	\begin{equation}\label{eq:c1178}
	d_\mathcal{C}c_{11}^{7,8}=b_{18}^{5,6},
	\end{equation}
	\begin{equation}
	\boxed{d_\mathcal{C}c_{11}^{3,5}=b_{10}^{5,6}c_{11}^{4,5},}
	\end{equation}
	\begin{equation}
	\boxed{d_\mathcal{C}c_{11}^{4,6}=c_{11}^{4,5}b_{11}^{3,4},}
	\end{equation}
	\begin{equation}\label{eq:c1157}
	d_\mathcal{C}c_{11}^{5,7}=b_6^{3,4}c_7^{6,7}+b_{18}^{3,5},
	\end{equation}
	\begin{equation}\label{eq:c1114}
	d_\mathcal{C}c_{11}^{1,4}=b_3^{2,3}c_{11}^{2,4}+c_{11}^{1,2}b_2^{4,5}+c_{11}^{1,3}b_{10}^{5,6}+c_{10}^{4,5}+\dots,
	\end{equation}
	\begin{equation}\label{eq:c1125}
	\boxed{d_\mathcal{C}c_{11}^{2,5}=b_2^{4,5}c_{11}^{4,5}+b_{18}^{2,3},}
	\end{equation}
	\begin{equation}
	\boxed{d_\mathcal{C}c_{11}^{3,6}=b_{10}^{5,6}c_{11}^{4,6}+c_{11}^{3,5}b_{11}^{3,4},}
	\end{equation}
	\begin{equation}\label{eq:c1147}
	\boxed{d_\mathcal{C}c_{11}^{4,7}=c_{11}^{4,5}b_3^{2,3}+b_{18}^{4,5},}
	\end{equation}
	\begin{equation}\label{eq:c1115}
	\boxed{d_\mathcal{C}c_{11}^{1,5}=b_3^{2,3}c_{11}^{2,5}+b_{11}^{3,4}c_{11}^{3,5}+b_{12}^{4,5}c_{11}^{4,5}+b_{18}^{1,3},}
	\end{equation}
	\begin{equation}\label{eq:c1126}
	d_\mathcal{C}c_{11}^{2,6}=b_2^{4,5}c_{11}^{4,6}+c_{11}^{2,3}+c_{11}^{2,5}b_{11}^{3,4}+\dots,
	\end{equation}
	\begin{equation}\label{eq:c1137}
	d_\mathcal{C}c_{11}^{3,7}=b_{10}^{5,6}c_{11}^{4,7}+c_{11}^{6,7}+c_{11}^{3,5}b_3^{2,3}+\dots,
	\end{equation}
	\begin{equation}\label{eq:c1148}
	\boxed{d_\mathcal{C}c_{11}^{4,8}=c_{11}^{4,5}b_{12}^{4,5}+c_{11}^{4,6}b_{10}^{5,6}+c_{11}^{4,7}b_2^{4,5}+b_{18}^{4,6},}
	\end{equation}
	\begin{equation}\label{eq:c1116}
	d_\mathcal{C}c_{11}^{1,6}=b_{11}^{3,4}c_{11}^{3,6}+b_{12}^{4,5}c_{11}^{4,6}+c_{11}^{1,3}+c_{11}^{1,5}b_{11}^{3,4}+\dots,
	\end{equation}
	\begin{equation}\label{eq:c1127}
	\boxed{d_\mathcal{C}c_{11}^{2,7}=b_2^{4,5}c_{11}^{4,7}+c_{11}^{2,5}b_3^{2,3}+b_{18}^{2,5},}
	\end{equation}
	\begin{equation}\label{eq:c1138}
	d_\mathcal{C}c_{11}^{3,8}=b_{10}^{5,6}c_{11}^{4,8}+c_{11}^{6,8}+c_{11}^{3,5}b_{12}^{4,5}+c_{11}^{3,6}b_{10}^{5,6},
	\end{equation}
	\begin{equation}\label{eq:c1117}
	d_\mathcal{C}c_{11}^{1,7}=b_3^{2,3}c_{11}^{2,7}+b_{12}^{4,5}c_{11}^{4,7}+c_{11}^{1,2}+\dots,
	\end{equation}
	\begin{equation}\label{eq:c1124}
	d_\mathcal{C}c_{11}^{2,8}=b_2^{4,5}c_{11}^{4,8}+c_{11}^{2,4}+\dots,
	\end{equation}
	\begin{equation}\label{eq:c1118}
	d_\mathcal{C}c_{11}^{1,8}=b_{12}^{4,5}c_{11}^{4,8}+c_{11}^{5,8}+c_{11}^{1,5}b_{12}^{4,5}+\dots.
	\end{equation}
	Arguing similarly as in the case of $C_7$, from (\ref{eq:c1134}), (\ref{eq:c1156}), (\ref{eq:c1178}), (\ref{eq:c1157}), (\ref{eq:c1114}), (\ref{eq:c1126}), (\ref{eq:c1137}), (\ref{eq:c1116}), (\ref{eq:c1138}), (\ref{eq:c1117}), (\ref{eq:c1124}), (\ref{eq:c1118}), we see that the following pairings of generators can be cancelled with each other:
	\begin{equation}
	\begin{aligned}
	&(c_{11}^{3,4},c_8^{5,6}),(c_{11}^{5,6},c_9^{3,4}),(c_{11}^{7,8},b_{18}^{5,6}),(c_{11}^{5,7},b_{18}^{3,5}), \\
	&(c_{11}^{1,4},c_{10}^{4,5}),(c_{11}^{2,6},c_{11}^{2,3}),(c_{11}^{3,7},c_{11}^{6,7}),(c_{11}^{1,6},c_{11}^{1,3}), \\
	&(c_{11}^{3,8},c_{11}^{6,8}),(c_{11}^{1,7},c_{11}^{1,2}),(c_{11}^{2,8},c_{11}^{2,4}),(c_{11}^{1,8},c_{11}^{5,8}).
	\end{aligned}
	\end{equation}
	We conclude that the remaining generators in $C_{11}$ are
	\begin{equation}
	c_{11}^{1,5},c_{11}^{2,5},c_{11}^{2,7},c_{11}^{3,5},c_{11}^{3,6},c_{11}^{4,5},c_{11}^{4,6},c_{11}^{4,7},c_{11}^{4,8}.
	\end{equation}
	Combining (\ref{eq:c725}) and (\ref{eq:c1125}) we see that one of the generators $c_7^{2,5}$ and $c_{11}^{2,5}$ can be cancelled with $b_{18}^{2,3}$. To be symmetric, we will denote the remaining generator in $\mathcal{C}'(\Lambda_{2,2,2})$ by $c_7^{2,5}+c_{11}^{2,5}$. It follows that
	\begin{equation}
	\boxed{d_\mathcal{C}(c_7^{2,5}+c_{11}^{2,5})=b_2^{4,5}c_7^{4,5}+b_2^{4,5}c_{11}^{4,5}.}
	\end{equation}
	Similarly, from (\ref{eq:c747}) and (\ref{eq:c1147}), we obtain
	\begin{equation}
	\boxed{d_\mathcal{C}(c_7^{4,7}+c_{11}^{4,7})=c_7^{4,5}b_3^{2,3}+c_{11}^{4,5}b_3^{2,3}.}
	\end{equation}
	Meanwhile, $b_{18}^{4,5}$ is cancelled. \\
	Also, (\ref{eq:c715}) together with (\ref{eq:c1115}) implies that one can cancel one of $c_7^{1,5}$ and $c_{11}^{1,5}$ with $b_{18}^{1,3}$. As before, use $c_7^{1,5}+c_{11}^{1,5}$ to stand for the remaining generator in $\mathcal{C}'(\Lambda_{2,2,2})$, one has
	\begin{equation}
	\boxed{d_\mathcal{C}(c_7^{1,5}+c_{11}^{1,5})=b_3^{2,3}(c_7^{2,5}+c_{11}^{2,5})+b_6^{3,4}c_7^{3,5}+b_7^{4,5}c_7^{4,5}+b_{11}^{3,4}c_{11}^{4,5}+b_{12}^{4,5}c_{11}^{4,5}.}
	\end{equation}
	We know from (\ref{eq:c1178}) that $b_{18}^{5,6}=0$ in the dg algebra $\mathcal{C}'(\Lambda_{2,2,2})$. Combining this fact with (\ref{eq:c748}) and (\ref{eq:c1148}) we deduce that
	\begin{equation}
	\boxed{d_\mathcal{C}(c_7^{4,8}+c_{11}^{4,8})=c_7^{4,5}b_7^{4,5}+c_7^{4,6}b_5^{5,6}+(c_7^{4,7}+c_{11}^{4,7})b_2^{4,5}+c_{11}^{4,5}b_{12}^{4,5}+c_{11}^{4,6}b_{10}^{5,6},}
	\end{equation}
	and $b_{18}^{4,6}$ has been cancelled. \\
	By (\ref{eq:c1157}), we know that
	\begin{equation}
	b_{18}^{3,5}=b_6^{3,4}c_7^{6,7}
	\end{equation}
	holds in $\mathcal{C}'(\Lambda_{2,2,2})$. This combined with (\ref{eq:c727}) and (\ref{eq:c1127}) implies that
	\begin{equation}
	\boxed{d_\mathcal{C}(c_7^{2,7}+c_{11}^{2,7})=b_2^{4,5}(c_7^{4,7}+c_{11}^{4,7})+(c_7^{2,5}+c_{11}^{2,5})b_3^{2,3},}
	\end{equation}
	and $b_{18}^{2,5}$ is cancelled.
	
	We have obtained all formulas of the differentials of the relevant generators in $\mathcal{C}'(\Lambda_{2,2,2})$. It remains to cancel all of the generators in $B_{18}$. Note that in the above, we have already cancelled $b_{18}^{5,6}$, $b_{18}^{3,5}$, $b_{18}^{2,3}$, $b_{18}^{4,5}$, $b_{18}^{1,3}$, $b_{18}^{4,6}$ and $b_{18}^{2,5}$. To cancel the remaining ones, apply (\ref{eq:B-circ}) to the case when $j=18$. Using the fact that $A_8=A_3$ in $\mathcal{C}'(\Lambda_{2,2,2})$, we have
	\begin{equation}\label{eq:b18}
	d_\mathcal{C}b_{18}^{m,n}=\sum_{m<k<n}a_3^{\sigma_4(m),\sigma_4(k)}b_{18}^{k,n}+\sum_{m<k<n}b_{18}^{m,k}a_3^{\sigma(k),\sigma(n)}.
	\end{equation}
	Since $a_3^{1,2}=a_3^{3,5}=a_3^{4,6}=1$, we deduce from (\ref{eq:b18}) that the generators
	\begin{equation}
	(b_{18}^{1,4},b_{18}^{3,4}),(b_{18}^{1,6},b_{18}^{3,6}),(b_{18}^{1,5},b_{18}^{1,2}),(b_{18}^{2,6},b_{18}^{2,4})
	\end{equation}
	can be cancelled with each other.
	
	The final step is to compute the gradings of the generators in $\mathcal{C}'(\Lambda_{2,2,2})$. To do this, we equip the sheets in the Legendrian surface $\Lambda_{2,2,2}$ with the Maslov potential $\mu_{2,2,2}:\Lambda_{2,2,2}\rightarrow\mathbb{Z}$ as specified in Figure \ref{fig:front-without-2-handles}. It is then straightforward to check that the gradings of the remaining generators in $\mathcal{C}'(\Lambda_{2,2,2})$ are exactly as in (\ref{eq:grading}).
\end{proof}

\subsection{Calabi-Yau completions as wrapped Fukaya categories}\label{section:comparison}

In this subsection, we complete the proof of Theorem \ref{theorem:main}. We begin by recalling the definition of the Chekanov-Eliashberg dg algebra $\mathcal{C}E(\Lambda)$ over $\mathbb{K}=\mathbb{Z}/2$ of a Legendrian surface $\Lambda\subset J^1(\mathbb{R}^2)$. More details can be found in $\cite{ees}$. As in Section \ref{section:def-cell}, assume that $\Lambda$ has vanishing Maslov class.
\bigskip

Denote by $\mathcal{R}$ the set of transverse double points in the Lagrangian projection $p_{xy}(\Lambda)$. Without loss of generality, we can assume that $\Lambda$ is \textit{chord generic}, namely $\mathcal{R}$ is a finite set. As a graded $\mathbb{K}$-algebra,
\begin{equation}
\mathcal{C}E(\Lambda):=\bigoplus_{i=0}^\infty\mathbb{K}\langle\mathcal{R}\rangle^{\otimes_\mathbb{K} i}.
\end{equation}
As explained in Section \ref{section:def-cell}, when $\Lambda=\bigsqcup\Lambda_v$ is a Legendrian link, $\mathcal{C}E(\Lambda)$ can be regarded as a dg algebra over $\Bbbk$, such that $e_w\mathcal{R}e_v$ consists of the Reeb chords from $\Lambda_w$ to $\Lambda_v$. With respect to the differential $\partial$ defined below, this endows $\mathcal{C}E(\Lambda)$ with the structure of a dg algebra over $\Bbbk$.

For any transverse double point $c\in\mathcal{R}$, its pre-image consists of two points, $c_+$ and $c_-$, where $z(c_+)>z(c_-)$. By slight abuse of notations, their images under the front projection $p_{xz}$ will still be denoted by $c_+$ and $c_-$. By assumption, they project to the same point $x_c\in\mathbb{R}^2$. Let $f_+$ and $f_-$ be the local defining functions of the sheets of $\Lambda$ which contain the points $c_+$ and $c_-$ respectively. $x_c$ is a non-degenerate critical point of the local difference function $f:=f_+-f_-$, so it has an associated Morse index $\mathit{ind}(x_c)$. Choose a path $\gamma$ from $c_+$ to $c_-$, which is transverse to the singular set of $p_{xz}(\Lambda)$. It follows that $\gamma$ intersects the cusp edges of $p_{xz}(\Lambda)$ in a finite number of points. Denote by $d(\gamma)$ the number of times that $\gamma$ crosses from the upper sheet to the lower sheet, and by $u(\gamma)$ the number of times that $\gamma$ crosses from the lower sheet to the upper sheet, the grading of $c$ in the dg algebra $\mathcal{C}E(\Lambda)$ is defined to be
\begin{equation}
|c|=u(\gamma)+1-\mathit{ind}(x_c)-d(\gamma).
\end{equation}

For generators $a;b_1,\dots,b_l\in\mathcal{R}$, we can define a moduli space $\mathcal{M}_\Lambda(a;b_1,\dots,b_l)$ which parametrizes holomorphic maps
\begin{equation}
u:(\Delta_{l+1},\partial\Delta_{l+1})\rightarrow(\mathbb{C}^2,p_{xy}(\Lambda)),
\end{equation}
where $\Delta_{l+1}$ is an $(l+1)$-punctured disc with punctures labelled counterclockwisely by $q_0,\dots,q_l$ on the boundary. As in the case of a Legendrian link in $(\mathbb{R}^3,\xi_\mathit{std})$, we need to introduce the \textit{Reeb sign} for the punctures. For a Reeb chord $c\in\mathcal{R}$, pick small neighborhoods $S_\pm\subset\Lambda$ of $c_\pm$ that are mapped injectively by $p_{xy}$ into $\mathbb{C}^2$. If $u(q_i)=c$, we say that $q_i$ has \textit{positive} (resp. \textit{negative}) Reeb sign if $u$ maps points clockwise of $p_i$ on $\partial\Delta_{l+1}$ to the lower (resp. upper) sheet of $p_{xy}(\Lambda)$, and points counterclockwise of $p_i$ on $\partial\Delta_{l+1}$ to the upper (resp. lower) sheet of $p_{xy}(\Lambda)$. See Figure \ref{fig:Reeb}. Furthermore, $u$ is required to satisfy the following boundary and asymptotic conditions:
\begin{itemize}
	\item $u$ maps the boundary components of the punctured disc $\Delta_{l+1}$ to $p_{xy}(\Lambda)\subset\mathbb{C}^2$;
	\item $u(q_0)=a$ has positive Reeb sign and $u(q_i)=b_i$ has negative Reeb sign for $i=1,\dots,l$.
\end{itemize}

\begin{figure}[htb!]
	\centering
	\begin{tikzpicture}
	\filldraw[draw=black,color={black!15}] (-2.4,-2) rectangle (0,0);
	\draw (0,0) ellipse (2.4cm and 0.8cm);
	\draw (0,-2) ellipse (2.4cm and 0.8cm);
	\draw [orange] [decoration={markings, mark=at position 1/2 with {\arrow{<}}},postaction={decorate}] (0,0) to (0,-2);
	\draw (-2.4,0) to (0,0);
	\draw (-2.4,-2) to (0,-2);
	\node [orange] at (0.2,-1) {$c$};
	\node at (0.2,0.1) {$c_+$};
	\node at (0.2,-2.1) {$c_-$};
	\node at (1.2,0) {$S_+$};
	\node at (1.2,-2) {$S_-$};
	\draw [scale=0.5] [red] [->] (-2.4,-2) ++(140:5mm) arc (-220:40:5mm);
	\end{tikzpicture}
	\caption{A positive puncture lifted to $\mathbb{R}^5$, where the shaded region is the image of a holomorphic disc}
	\label{fig:Reeb}
\end{figure}
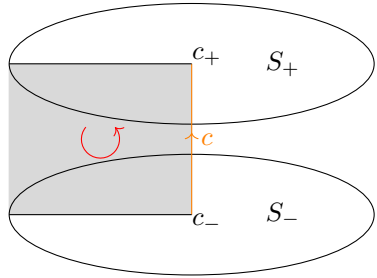

The differential $\partial$ in $\mathcal{C}E(\Lambda)$ is defined as
\begin{equation}
\partial a=\sum(\#_2\mathcal{M}_\Lambda(a;b_1,\dots,b_l))b_1\dots b_l,
\end{equation}
where the sum on the right-hand side above is taken over all words $b_1\dots b_l$ of Reeb chords for which $\dim(\mathcal{M}_\Lambda(a;b_1,\dots,b_l))=0$ and $\#_2$ denotes the mod 2 count of the rigid elements in $\mathcal{M}_\Lambda(a;b_1,\dots,b_l)$. In order to define $\mathcal{C}E(\Lambda)$ over an arbitrary field $\mathbb{K}$, we need to take into consideration the orientations of the Morse flow trees which correspond to rigid holomorphic discs in $\mathcal{M}(a;b_1,\dots,b_l)$, see Section \ref{section:subo}.
\bigskip

Recall that it is proved by Rutherford-Sullivan in $\cite{rs2}$ that the cellular dg algebra $\mathcal{C}(\Lambda)$ is quasi-isomorphic to its Chekanov-Eliashberg algebra $\mathcal{C}E(\Lambda)$ defined over $\mathbb{K}=\mathbb{Z}/2$. In particular, we have a quasi-isomorphism
\begin{equation}\label{eq:RS2}
\mathcal{C}(\Lambda_{p,q,r})\cong\mathcal{C}E(\Lambda_{p,q,r}).
\end{equation}
Note that this quasi-isomorphism preserves the $\Bbbk$-bimodule structures on both sides.

It remains to identify the cellular dg algebra $\mathcal{C}(\Lambda_{p,q,r})$ with the Ginzburg algebra $\mathcal{G}_{p,q,r}$. This follows essentially from our computations in Sections \ref{section:example} and \ref{section:computation}.

\begin{proposition}\label{proposition:pqr}
Let $p,q$ and $r$ be integers satisfying $p\geq2,q\geq2,r\geq2$. We have a quasi-isomorphism
\begin{equation}
\mathcal{C}(\Lambda_{p,q,r})\cong\mathcal{G}_{p,q,r}
\end{equation}
between the cellular dg algebra and the Ginzburg dg algebra defined over $\mathbb{K}=\mathbb{Z}/2$.
\end{proposition}
\begin{proof}
By definition, the Ginzburg algebra $\mathcal{G}_{p,q,r}$ is the semi-free dg algebra generated by the arrows (which have degree 0)
\begin{equation}
a_1,a_2,b_1,b_2,b_3,c_1,c_2,c_3,x_i,y_j,z_k;
\end{equation}
the reversed arrows (which have degree -1)
\begin{equation}
a_1^\ast,a_2^\ast,b_1^\ast,b_2^\ast,b_3^\ast,c_1^\ast,c_2^\ast,c_3^\ast,x_i^\ast,y_j^\ast,z_k^\ast;
\end{equation}
together with the loops (which have degree -2)
\begin{equation}
z_A,z_B,z_{P_1},z_{Q_1},z_{R_1},z_{P_{i+1}},z_{Q_{j+1}},z_{R_{k+1}},
\end{equation}
The differential $d:\mathcal{G}_{p,q,r}\rightarrow\mathcal{G}_{p,q,r}[1]$ is determined by its action on the generators listed above. More explicitly,
\begin{equation}
da_1=da_2=db_1=db_2=db_3=dc_1=dc_2=dc_3=dx_i=dy_j=dz_k=0,
\end{equation}
\begin{equation}
dx_1=\dots=dx_{p-1}=dy_1=\dots=dy_{q-1}=dz_1=\dots=dz_{r-1}=0,
\end{equation}
\begin{equation}
da_1^\ast=b_2c_2+b_3c_3,da_2^\ast=b_1c_1+b_3c_3,
\end{equation}
\begin{equation}
db_1^\ast=c_1a_2,db_2^\ast=c_2a_1,db_3^\ast=c_3a_1+c_3a_2,
\end{equation}
\begin{equation}
dc_1^\ast=a_2b_1,dc_2^\ast=a_1b_2,dc_3^\ast=a_1b_3+a_2b_3,
\end{equation}
\begin{equation}
dx_1^\ast=\dots=dx^\ast_{p-1}=dy_1^\ast=\dots=dy_{q-1}^\ast=dz_1^\ast=\dots=dz_{r-1}^\ast=0,
\end{equation}
\begin{equation}
dz_A=a_1a_1^\ast+a_2a_2^\ast+c_1^\ast c_1+c_2^\ast c_2+c_3^\ast c_3,
\end{equation}
\begin{equation}
dz_B=a_1^\ast a_1+a_2^\ast a_2+b_1b_1^\ast+b_2b_2^\ast+b_3b_3^\ast,
\end{equation}
\begin{equation}
dz_{P_1}=c_1c_1^\ast+b_1^\ast b_1+x_1x_1^\ast,dz_{Q_1}=c_2c_2^\ast+b_2b_2^\ast+y_1y_1^\ast,dz_{R_1}=c_3c_3^\ast+b_3^\ast b_3+z_1z_1^\ast,
\end{equation}
\begin{equation}\label{eq:add1}
dz_{P_i}=x_{i-1}^\ast x_{i-1}+x_ix_i^\ast,dz_{Q_j}=y_{j-1}^\ast y_{j-1}+y_jy_j^\ast,dz_{R_k}=z_{k-1}^\ast z_{k-1}+z_kz_k^\ast,
\end{equation}
where $2\leq i\leq p-2$, $2\leq j\leq q-2$, $2\leq k\leq r-2$, and finally
\begin{equation}\label{eq:add2}
dz_{P_{p-1}}=x_{p-2}^\ast x_{p-2},dz_{Q_{q-1}}=y_{q-2}^\ast y_{q-2},dz_{R_{r-1}}=z_{r-2}^\ast z_{r-2}.
\end{equation}

On the other hand, by Proposition \ref{proposition:222}, it is clear that the map $\Phi_{2,2,2}:\mathcal{C}(\Lambda_{2,2,2})\rightarrow\mathcal{G}_{2,2,2}$ defined by
\begin{equation}\label{eq:222}
\begin{split}
&c_{12}^{4,5}\mapsto a_1,c_7^{4,5}\mapsto a_2,b_6^{3,4}\mapsto b_1,b_{11}^{3,4}\mapsto b_2,b_3^{2,3}\mapsto b_3,b_5^{5,6}\mapsto c_1,b_{10}^{5,6}\mapsto c_2,b_2^{4,5}\mapsto c_3;\\
&b_{12}^{4,5}\mapsto a_1^\ast,b_7^{4,5}\mapsto a_2^\ast,c_7^{3,5}\mapsto b_1^\ast,c_{11}^{3,5}\mapsto b_2^\ast,\\
&c_7^{2,5}+c_{11}^{2,5}\mapsto b_3^\ast,c_7^{4,6}\mapsto c_1^\ast,c_{11}^{4,6}\mapsto c_2^\ast,c_7^{4,7}+c_{11}^{4,7}\mapsto c_3^\ast;\\
&c_7^{1,5}+c_{11}^{1,5}\mapsto z_A,c_7^{4,8}+c_{11}^{4,8}\mapsto z_B,c_7^{3,6}\mapsto z_P,c_{11}^{3,6}\mapsto z_Q,c_7^{2,7}+c_{11}^{2,7}\mapsto z_R
\end{split}
\end{equation}
defines an identification between the generators of the dg algebras, and it is straightforward to check that this map is compatible with the differentials. In particular, $\Phi_{2,2,2}$ is a quasi-isomorphism.

After replacing the components $\Lambda_P,\Lambda_Q$ and $\Lambda_R$ of $\Lambda_{2,2,2}$ with the corresponding $A_{p-1}$, $A_{q-1}$ and $A_{r-1}$ chains of unknots $\{\Lambda_{P_i}\}$, $\{\Lambda_{Q_j}\}$ and $\{\Lambda_{R_k}\}$, one can arrange so that the base projection $p_x(\Lambda_{p,q,r})$ is as shown in Figure \ref{figure:decomposition(2,2,2)}, where the additional cells in the cellular decomposition of $p_x(\Lambda_{p,q,r})$ are coloured in blue. Note that we have arranged, by using a non-generic front projection of $\Lambda_{p,q,r}$, that the base projections of the crossing arcs in the $A_{p-1}$ and $A_{q-1}$ chains of unknots $\{\Lambda_{P_i}\}$ and $\{\Lambda_{Q_j}\}$ coincide precisely with the 1-cells $e_1^1$ and $e_8^1$, while the crossing arcs of the $A_{r-1}$-chain of unknots lie over $e_{14}^1$. The base projections of the cusp edges of $p_{x,z}(\Lambda_{P_i})$, $p_{x,z}(\Lambda_{Q_j})$ and $p_{x,z}(\Lambda_{R_k})$ are given respectively by the 1-cells $e_{15}^1$, $e_{18}^1$ and the largest circle, so there are no newly created 1-cells for these additional cusp edges in the Legendrian front of $\Lambda_{p,q,r}$. This arrangement of the Legendrian front of $\Lambda_{p,q,r}$ is justified by our discussions in Section \ref{section:non-gen}. Note that since we have to add the additional cells labelled by $A_{13},A_{14},B_{22}$ to divide the 2-cell labelled by $C_3$ into a polygon, the original 1-cell $e_{20}^1$ in the cellular decomposition of $p_x(\Lambda_{2,2,2})$ is divided into two 1-cells, say $e_{20,+}^1$ and $e_{20,-}^1$. However, since we can cancel the generators in one of the matrices $B_{20,+}$ and $B_{20,-}$ with the generators in $A_{13}$, it causes no additional complexity in our computations. Using the same cancellation arguments as in the proof of Proposition \ref{proposition:222}, it is not hard to see that passing from $\mathcal{C}'(\Lambda_{2,2,2})$ to $\mathcal{C}'(\Lambda_{p,q,r})$ adds new generators in the matrices $B_7,C_7,B_{12},C_{11},B_{22}$ and $C_{12}$. All the other new generators involved in the original definition of the cellular dg algebra $\mathcal{C}(\Lambda_{p,q,r})$, including those coming from the newly added cells in the base projection of $\Lambda_{p,q,r}$, can be cancelled out. 

For the new generators created by the parallel copies of $\Lambda_P$, $\Lambda_Q$ and $\Lambda_R$, namely $\Lambda_{P_i}$, $\Lambda_{Q_j}$ and $\Lambda_{R_k}$ for $i,j,k\geq 2$, we can apply exactly the same cancellation procedure as in the proof of Proposition \ref{proposition:An}. The upshot is that up to quasi-isomorphism, the additional generators in $\mathcal{C}'(\Lambda_{p,q,r})$ can be identified to be
\begin{equation}\label{eq:newgen}
x_i,x_i^\ast,y_j,y_j^\ast,z_k,z_k^\ast,z_{P_2},\dots,z_{P_{p-1}},
\end{equation}
where $1\leq i\leq p-1$, $1\leq j\leq q-1$, $1\leq k\leq r-1$, with the differentials given exactly as in (\ref{eq:add1}) and (\ref{eq:add2}) above. Furthermore, in addition to the new generators in (\ref{eq:newgen}), there are additional terms $x_1x_1^\ast,y_1y_1^\ast$ and $z_1z_1^\ast$ appearing in the differentials of $z_{P_1},z_{Q_1}$ and $z_{R_1}$ respectively. 

Finally, notice that the Maslov potential $\mu_{2,2,2}$ on the Legendrian link $\Lambda_{2,2,2}$ extends naturally to a Maslov potential $\mu_{p,q,r}:\Lambda_{p,q,r}\rightarrow\mathbb{Z}$, by equipping the additional unknots in $\Lambda_{p,q,r}$ with a Maslov potential as in (\ref{eq:MasAn}). This shows that the grading on $\mathcal{C}(\Lambda_{p,q,r})$ matches with that on $\mathcal{G}_{p,q,r}$, which completes the proof.
\end{proof}

By (\ref{eq:RS2}) we get a quasi-isomorphism
\begin{equation}
\mathcal{C}E(\Lambda_{p,q,r})\cong\Pi_3(\mathcal{A}_{p,q,r})
\end{equation}
between the Chekanov-Eliashberg algebra and the 3-Calabi-Yau completion of the directed $A_\infty$-algebra $\mathcal{A}_{p,q,r}$. With our definitions, it is straightforward to see that this quasi-isomorphism is compatible with the $\Bbbk$-bimodule structures on both sides. This proves Theorem \ref{theorem:main} for $\mathbb{K}=\mathbb{Z}/2$. The general case is proved by combining the computations of the signs of the relevant Morse flow trees, which is carried out in Section \ref{section:sign}.
\bigskip

For related results which identify certain (partially) wrapped Fukaya categories with relative Calabi-Yau completions, see $\cite{etl2}$ and $\cite{wy}$.

\subsection{Degenerate triples}\label{section:exceptional}

Let the polynomial $t_{p,q,r}(x,y,z)$ be as in (\ref{eq:4DMilnor}), it gives rise to a symplectic Landau-Ginzburg model $(\mathbb{C}^3,t_{p,q,r})$. Without loss of generality, we assume that $p\geq p\geq r\geq0$. It turned out that these Landau-Ginzburg models have 1-dimensional mirrors.
\bigskip

When $r\geq2$, the Morsification $\tilde{t}_{p,q,r}$ defines a Lefschetz fibration on $\mathbb{C}^3$, and the mirror of $(\mathbb{C}^3,\tilde{t}_{p,q,r})$ is the weighted projective line $\mathbb{P}^1_{p,q,r}$. The Fukaya categories $\mathcal{F}(M_{p,q,r})$ and $\mathcal{W}(M_{p,q,r})$ of the corresponding Wesintein manifold $M_{p,q,r}$ have been studied in the above. As we have proved, they are $A_\infty$-Koszul dual to each other as $\mathbb{Z}$-graded $A_\infty$-categories.
\bigskip

When $r=1$, the situation is a simplification of the previous case. The mirror of $(\mathbb{C}^3,\tilde{t}_{p,q,1})$ is $\mathbb{P}^1_{p,q}$. When $q\geq2$, the Fukaya categories of $M_{p,q,1}$ are described by the quiver $Q_{p,q,1}$
\begin{equation}\label{eq:simquiver}
\begin{tikzcd}
& & \bullet_{P_1} \arrow[r,"x_1"] \arrow[dll,orange,bend right,"c_1"]
& \bullet_{P_2} \arrow[r,"x_2"] & \bullet\dots\bullet \arrow[r,"x_{p-2}"] & \bullet_{P_{p-1}} \\
\bullet_A \arrow[r,bend left,"a_1"] \arrow[r,bend right,"a_2"] & \bullet_B
\arrow[ur,"b_1"]
\arrow[dr,"b_2"] \\
& & \bullet_{Q_1} \arrow[r,"y_1"] \arrow[ull,orange,bend left,"c_2"]
& \bullet_{Q_2} \arrow[r,"y_2"] & \bullet\dots\bullet \arrow[r,"y_{q-2}"] & \bullet_{Q_{q-1}}
\end{tikzcd}
\end{equation}
with potential
\begin{equation}
w_{p,q,1}=a_1b_2c_2+a_2b_1c_1.
\end{equation}
In the simplest case when $p=q=r=1$, $Q_{1,1,1}$ is just the Kronecker quiver, and the potential $w_{1,1,1}=0$. For $M_{p,q,1}$, one can still find a Lefschetz fibration $\pi_{p,q,1}:M_{p,q,1}\rightarrow\mathbb{C}$ (although the construction of $\pi_{p,q,1}$ does not follow from the general method described in Section \ref{section:Lefschetz 6-fold}) and use the Casals-Murphy recipe to draw its Legendrian front. It turns out that one can cancel all the subcritical handles in the original frontal description of $M_{p,q,1}$ to get a Legendrian surface $\Lambda_{p,q,1}\subset(\mathbb{R}^5,\xi_\mathit{std})$ and the Chekanov-Eliashberg algebra $\mathcal{C}E(\Lambda_{p,q,1})$ can still be computed through its cellular model $\mathcal{C}(\Lambda_{p,q,1})$. The proof of Koszul duality between the endomorphism algebras of $\mathcal{F}(M_{p,q,1})$ and $\mathcal{W}(M_{p,q,1})$ is completely analogous to the previous case.
\bigskip

When $r=0$ and $q\geq1$, the map $t_{p,q,0}:\mathbb{C}^3\rightarrow\mathbb{C}$ cannot be Morsified to produce a Lefschetz fibration, and the mirror of $(\mathbb{C}^3,t_{p,q,0})$ is $\mathbb{A}^1_{p,q}:=\mathbb{P}^1_{p,q}\setminus\{\infty\}$. When $p=q=1$, the map $t_{1,1,0}:\mathbb{C}^3\rightarrow\mathbb{C}$ does not have any critical point. Instead, there is a special fiber over the origin, which is topologically different from all the other fibers. Correspondingly, its mirror $\mathbb{A}^1$ is Floer theoretically trivial over $\mathbb{K}=\mathbb{C}$. The Landau-Ginzburg model $(\mathbb{C}^3,t_{1,1,0})$ has already been studied in Example 2.4 of $\cite{aak}$, where it is equivalently interpreted as a 1-dimensional Landau-Ginzburg model $(\mathbb{C}^\ast,x)$. In fact, let $L_\infty\subset\mathbb{C}^\ast$ be a properly embedded arc which connects $+\infty$ to itself by passing around the origin, then
\begin{equation}
\mathcal{A}_{1,1,0}\cong\mathit{CF}^\ast(L_\infty,L_\infty)\cong H^\ast(S^1;\mathbb{K}),
\end{equation}
which shows that $L_\infty$ is mirror to the skyscraper sheaf at the origin of $\mathbb{A}^1$. $\mathcal{A}_{1,1,0}$ should be regarded as the endomorphism algebra of the \textit{infinitesimally wrapped Fukaya category} (cf. $\cite{nz}$) $\mathcal{A}(t_{1,1,0})$. In this case, it still makes sense to consider the trivial extension $\mathcal{A}_{1,1,0}\oplus\mathcal{A}_{1,1,0}^\vee[-3]$, since one can construct an unobstructed (but non-exact) Lagrangian submanifold $L_\sigma\subset M$, which is diffeomorphic to $S^1\times S^2$ and has trivial Maslov class. Now the trivial extension $\mathcal{A}_{1,1,0}\oplus\mathcal{A}_{1,1,0}^\vee[-3]$ can be regarded as the endomorphism algebra $\mathit{CF}^\ast(L_\sigma,L_\sigma)$ of $L_\sigma$ in the Fukaya category $\mathcal{F}(M_{1,1,0})$ (which extends the usual definition by allowing all the closed Lagrangian submanifolds $L\subset M_{1,1,0}$ with trivial obstructions $\mathfrak{m}_0(L)=0$ as its objects). This suggests the existence of a generalization of the suspension construction of Lefschetz fibrations explained in Section \ref{section:suspension} to more general symplectic Landau-Ginzburg models.

From Section 4.1 of $\cite{cm}$, we see that $M_{1,1,0}$ is obtained by attaching a Weinstein 3-handle to $D^6$ along the Legendrian surface $\Lambda_{1,1,0}$ depicted in Figure \ref{fig:trefoil}. Although in $\cite{cm}$, the Legendrian front $\Lambda_{1,1,0}$ of $M_{1,1,0}$ is obtained using a different Lefschetz fibration, one will end up with the same front by starting from our general framework in Section \ref{section:surgery}, and cancelling the 2-handle $R$ with the 3-handle corresponding to the Legendrian sphere $\Lambda_A$, see Figure \ref{fig:A}. In particular, this shows that for a general Weinstein manifold $M_{p,q,0}$, one can cancel all the subcritical handles in its frontal description, and the endomorphism algebra of its wrapped Fukaya category $\mathcal{W}(M_{p,q,0})$ is quasi-isomorphic to the Chekanov-Eliashberg algebra of a Legendrian surface $\Lambda_{p,q,0}\subset(\mathbb{R}^5,\xi_\mathit{std})$, which is of course computable in terms of the cellular dg algebra $\mathcal{C}(\Lambda_{p,q,0})$. Concretely, $\Lambda_{p,q,0}$ is a link of Legendrian 2-spheres obtained by attaching standard unknots to the surface $\Lambda_{1,1,0}$.

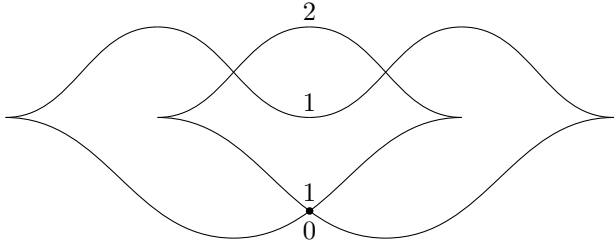
\begin{figure}
	\centering
	\begin{tikzpicture}[scale=2,auto=left,every node/.style={circle}]
	\draw (0,0) to[in=180,out=0] (1,0.6);
	\draw (1,0.6) to[in=180,out=0] (2,0);
	\draw (2,0) to[in=180,out=0] (3,0.6);
	\draw (3,0.6) to[in=180,out=0] (4,0);
	\draw (1,0) to[in=180,out=0] (2,0.6);
	\draw (2,0.6) to[in=180,out=0] (3,0);
	\draw (0,0) to[in=180,out=0] (1.5,-0.8);
	\draw (1.5,-0.8) to[in=180,out=0] (3,0);
	\draw (1,0) to[in=180,out=0] (2.5,-0.8);
	\draw (2.5,-0.8) to[in=180,out=0] (4,0);
	
	\draw (2,-0.62) node[circle,fill,inner sep=1pt]{};
	\node at (2,0.7)   {2};
	\node at (2,0.1)   {1};
	\node at (2,-0.5)   {1};
	\node at (2,-0.75)   {0};
	\end{tikzpicture}
	\caption{Legendrian front of $\Lambda_{1,1,0}$, which is the $S^1$-symmetric rotation of a right-handed trefoil knot with respect to the vertical axis of symmetry. The numbers above strands are values taken by the Maslov potential $\mu_{1,1,0}:\Lambda_{1,1,0}\rightarrow\mathbb{Z}$.}
	\label{fig:trefoil}
\end{figure}

Since the Legendrian surface $\Lambda_{1,1,0}\subset J^1(\mathbb{R}^2)$ can be constructed by spinning the Legendrian front of a right-handed Legendrian trefoil knot in $(\mathbb{R}^3,\xi_\mathit{std})$ along an axis which passes through one of the crossing points, its front projection contains a cone singularity, which is indicated by the thick dot in Figure \ref{fig:trefoil}. Its Chekanov-Eliashberg dg algebra $\mathcal{C}E(\Lambda_{1,1,0})$ can be computed by applying Proposition \ref{proposition:front-spin}, and there is a quasi-isomorphism
\begin{equation}
\mathcal{C}E(\Lambda_{1,1,0})\cong\mathbb{Z}/2[x_1,x_2],|x_1|=1,|x_2|=-2
\end{equation}
over $\mathbb{Z}/2$. In particular, this implies that $\mathcal{C}E(\Lambda_{1,1,0})$ is not Koszul dual to $\mathit{CF}^\ast(L_\sigma,L_\sigma)$ over $\mathbb{K}=\mathbb{Z}/2$.
\bigskip

When $q=r=0$, and $p\geq1$, the Landau-Ginzburg model $(\mathbb{C}^3,t_{p,0,0})$ is mirror to $\mathbb{K}_p^\times:=\mathbb{P}^1_p\setminus\{0,\infty\}$. In the simplest case when $p=1$, the map $t_{1,0,0}:\mathbb{C}^3\rightarrow\mathbb{C}$ defines a Morse-Bott fibration with critical locus isomorphic to $\mathbb{C}^\ast$. The main difficulty in answering whether the Fukaya categories $\mathcal{F}(M_{p,0,0})$ and $\mathcal{W}(M_{p,0,0})$ are Koszul dual comes from the fact that the Legendrian frontal description of $M_{p,0,0}$ involves a 2-handle which cannot be cancelled with any of the critical handles, therefore the calculation of the wrapped Fukaya category $\mathcal{W}(M_{p,0,0})$ requires a generalization of the work of Rutherford-Sullivan $\cite{rs1,rs2}$ for Legendrian surfaces in the contact connected sums $\#_n S^2\times S^3$. In dimension 4, the corresponding generalization has been obtained by Ekholm-Ng $\cite{en}$.
\bigskip

The Landau-Ginzburg model $(\mathbb{C}^3,t_{0,0,0})$ has been studied prominently in the literature, see for example $\cite{aak,dn}$. Its mirror is given by the pair-of-pants $\mathbb{P}^1\setminus\{0,1,\infty\}$. The Fukaya category of $(\mathbb{C}^3,t_{0,0,0})$ has been calculated by Nadler in $\cite{dn}$ in terms of microlocal sheaves, which turns out to be quasi-equivalent to $\mathcal{A}(t_{0,0,0}):=\mathit{Coh}_\mathit{tor}(\mathbb{P}^1\setminus\{0,1,\infty\})$,  the bounded dg category of finitely-generated torsion complexes on $\mathbb{P}^1\setminus\{0,1,\infty\}$. The full $A_\infty$-subcategory $\mathcal{V}(M_{0,0,0})\subset\mathcal{F}(M_{0,0,0})$ which is relevant for Koszul duality should be the trivial extension $\mathcal{A}(t_{0,0,0})\oplus\mathcal{A}(t_{0,0,0})^\vee[-3]$. However, as in the case of $M_{p,0,0}$, the Legendrian front of $M_{0,0,0}$ necessarily involves 2-handles.

\section{Orientations}\label{section:o}

In this section we take into account the orientation issue in the definition of Legendrian contact homology and obtain a computation of the Chekanov-Eliashberg algebra $\mathcal{C}E(\Lambda_{p,q,r})$ over any field $\mathbb{K}$.

\subsection{Orientations of Morse flow trees}\label{section:subo}

This subsection reviews the works $\cite{ck1,ck2}$ of Karlsson, which allow us to define the Chekanov-Eliashberg algebra $\mathcal{C}E(\Lambda)$ over $\mathbb{Z}$ for a Legendrian submanifold $\Lambda\subset J^1(S)$. Given a rigid Morse flow tree $\Gamma$, the definition of the sign $\varepsilon(\Gamma)$ involves four independent signs, namely the sign $\nu_\mathit{triv}(\Gamma)$ which depends on the choice of $\mathit{Spin}$ structures on $\Lambda$; the sign $\nu_\mathit{int}(\Gamma)$ which records the intersection orientation of the flow-outs of the sub flow trees of $\Gamma$; the sign $\nu_\mathit{end}(\Gamma)$ which encodes the information of $e$-vertices; and the sign $\nu_\mathit{stab}(\Gamma)$ which comes from capping orientations of pseudoholomorphic discs associated to Morse flow trees. Since we are only interested in the case when $\Lambda\subset\mathbb{R}^5$ is a disjoint union of Legendrian 2-spheres, the first sign $\nu_\mathit{triv}(\Gamma)$ is irrelevant for us, therefore we only recall here the definitions of $\nu_\mathit{int}(\Gamma)$, $\nu_\mathit{end}(\Gamma)$ and $\nu_\mathit{stab}(\Gamma)$.
\bigskip

As in Section \ref{section:cellular}, let $\Lambda\subset J^1(S)$ be a Legendrian surface which is $\mathit{Spin}$, where $S$ is an orientable surface equipped with a Riemannian metric $g$. Let $S_i$ and $S_j$ be two sheets of $\Lambda$ over a small open subset $U\subset S$, and assume that $S_i$ lies above $S_j$, namely $z(S_i)>z(S_j)$. Denote by $f_i$ and $f_j$ the local defining functions of $S_i$ and $S_j$ respectively, and set $f_{ij}:=f_i-f_j$. We refer to the fundamental paper $\cite{te}$ for background materials concerning Morse flow trees.

\begin{definition}
Let $\ell$ be a flow line of the local difference function $f_{ij}$, and let $K\subset S$ be a subset with $K\cap\ell\neq\emptyset$. The flow-out of $K$ along $\ell$ is the union of all maximal flow lines of $-\nabla f_{ij}$ that intersect $K$.
\end{definition}
Let $c$ be a puncture of the rigid Morse flow tree $\Gamma$ with $f_i>f_j$, denote by $\mathcal{U}(c)$ and $\mathcal{S}(c)$ respectively the unstable and stable manifold of $-\nabla f_{ij}$. Given a sub flow tree $\Gamma'\subset\Gamma$, denote by $s$ its special puncture. There is an edge $\ell\subset\Gamma'$ ending at $s$, with the other end point given by some true vertex $t$ of $\Gamma$. Denote by $\mathcal{P}^+,\mathcal{P}^-,\mathcal{E},\mathcal{P}^2$ and $\mathcal{Y}_0$ the set of 1-valent positive punctures, 1-valent negative punctures, $e$-vertices, negative 2-valent punctures, and $Y_0$-vertices respectively. The \textit{flow-out of} $\Gamma'$ \textit{at} $s$, denoted by $\mathit{FO}_s(\Gamma')$, is defined as follows.
\begin{itemize}
	\item When $t\in\mathcal{P}^+$, $\mathit{FO}_s(\Gamma')$ is the flow-out of $t$ along $\ell$. In particular, we have an identification between the tangent spaces $T_s\mathit{FO}_s(\Gamma')\cong T_s\mathcal{U}(t)$.
	\item When $t\in\mathcal{P}^-$, $\mathit{FO}_s(\Gamma')$ is the flow-out of $t$ along $\ell$. In particular, $T_s\mathit{FO}_s(\Gamma')\cong T_s\mathcal{S}(t)$.
	\item When $t\in\mathcal{E}$, let $I_t$ be an open interval centred at $t$ which is transverse to $\ell$, $\mathit{FO}_s(\Gamma')$ is the flow out of $I_t$ along $\ell$.
	\item When $t\in\mathcal{Y}_0$, and $s$ is a special positive puncture of the sub flow tree $\Gamma'$, the definition of $\mathit{FO}_s(\Gamma')$ is given inductively. Define the \textit{intersection manifold}
	\begin{equation}
	\mathit{IM}_t(\Gamma'):=\mathit{FO}_t(\Gamma_1')\cap\mathit{FO}_t(\Gamma_2'),
	\end{equation}
	where $\Gamma_1'$ and $\Gamma_2'$ are the sub flow trees of $\Gamma$ with $t$ as their common special positive puncture. $\mathit{FO}_s(\Gamma')$ is defined to be the flow-out of $\mathit{IM}_t(\Gamma')\cap I_t$ along $\ell$;
	\item When $t\in\mathcal{P}^2$, the intersection manifold $\mathit{IM}_t(\Gamma'):=\{t\}$ and $\mathit{FO}_s(\Gamma')=\ell$.
\end{itemize}
The cases when $t$ is a switch or a $Y_1$-vertex are not recalled here, as we will not encounter rigid Morse flow trees with such kind of internal vertices in the computation of $\mathcal{C}E(\Lambda_{p,q,r})$ in Section \ref{section:sign}. To simplify the discussions, assume from now on that there is \textit{no} switch or $Y_1$-vertex in $\Gamma$.

\paragraph{Convention} We should emphasize that the sub flow trees $\Gamma_1'$ and $\Gamma_2'$ are numbered so that the standard domain of $\Gamma_1'$ corresponds to the \textit{lower} part in the standard domain of $\Gamma'$.
\bigskip

To define the intersection orientation sign $\nu_\mathit{int}$, we start from some basic facts in linear algebra. Let $V_1,V_2\subset\mathbb{R}^n$ be subspaces, and $V=V_1\cap V_2$. There is a short exact sequence
\begin{equation}
0\rightarrow V\xrightarrow{\delta} V_1\oplus V_2\xrightarrow{\eta}\mathbb{R}^n\rightarrow 0,
\end{equation}
where
\begin{equation}
\delta(v)=(v,v),\eta(u,v)=v-u.
\end{equation}
For fixed choices of orientations on $V_1$ and $V_2$, there is an orientation $o(V)$ on their intersections such that the orientation induced by $V_i$ on its quotient $V_i/V$, together with $o(V)$, coincides with the original orientation on $V_i$. Define $\nu\in\{0,1\}$ to be the sign which satisfies
\begin{equation}
V\oplus(V_1/V)\oplus(V_2/V)\cong(-1)^\nu\cdot\mathbb{R}^n
\end{equation}
as oriented vector spaces.
\begin{definition}
The intersection orientation $o_\mathit{int}(V)$ is given by
\begin{equation}
o_\mathit{int}(V):=o_\mathit{int}(V_1,V_2)=(-1)^{\nu+\dim V_1\cdot(1+\dim V)}o(V).
\end{equation}
\end{definition}
Back to our specific geometric set up, let $o_\mathit{cap}(\mathcal{U}(t))$ denote the initial choice of the orientation of $\mathcal{U}(t)$, given as a wedge product of an oriented basis of $T_s(\mathcal{U}(t))$. It induces an orientation $o_\mathit{cap}(T_s\mathcal{S}(t))$ on the stable manifold such that
\begin{equation}
o_\mathit{cap}(T_t\mathcal{U}(t))\wedge o_\mathit{cap}(T_t\mathcal{S}(t))
\end{equation}
gives the original orientation on $T_tS$.

The orientations are defined inductively as follows:
\begin{itemize}
	\item if $t\in\mathcal{P}^+$, then
	\begin{equation}
	o(\mathit{FO}_s(\Gamma')):=o_\mathit{cap}(T_s\mathcal{U}(t));
	\end{equation}
	\item if $t\in\mathcal{P}^-$, then
	\begin{equation}
	o(\mathit{FO}_s(\Gamma')):=o_\mathit{cap}(T_s\mathcal{S}(t));
	\end{equation}
	\item if $t\in\mathcal{E}$, then
	\begin{equation}
	o(\mathit{FO}_s(\Gamma'))=o(T_sS);
	\end{equation}
	\item if $t\in\mathcal{P}^2$, then
	\begin{equation}
	o(\mathit{IM}_t(\Gamma'))=o_\mathit{int}(o(\mathit{FO}_t(\Gamma_1')),o_\mathit{cap}(\mathcal{U}(t)));
	\end{equation}
	\item if $t\in\mathcal{Y}_0$, $s$ is a special positive puncture of $\Gamma'$, then
	\begin{equation}
	o(\mathit{IM}_t(\Gamma'))=o_\mathit{int}(o(\mathit{FO}_t(\Gamma_1')),o(\mathit{FO}_t(\Gamma_2'))).
	\end{equation}
\end{itemize}
In the above, the orientations of flow-outs when $t$ is a 2-valent puncture or a $Y_0$-vertex are defined in terms of intersection manifolds. To recover the orientation of the flow-out $\mathit{FO}_s(\Gamma')$, let $v_s$ be the tangent vector of $\ell$ at $s$, pointing in the direction against the defining gradient vector field, and define the orientation of the flow-out along $\ell$ as
\begin{equation}
o(\mathit{FO}_s(\Gamma'))=o(\mathit{IM}_t(\Gamma'))\wedge v_s,
\end{equation}
where we have used parallel transport over the elementary regions to identify tangent spaces of $S$.
\bigskip

Assume that the positive puncture $a$ of $\Gamma$ is 1-valent. In this case, the sign $\nu_\mathit{int}$ can be defined as follows. Let $c$ denote the first vertex that we meet when going along $\Gamma$ from $a$, and let $\ell\subset\Gamma$ be an edge which starts at $c$, orient it so that it points toward $a$. Pick a point $s\in\ell$ which is contained in the same elementary region as $c$. Cutting at $s$ we obtain two sub flow trees
\begin{equation}
\Gamma=\Gamma_1'\cup\Gamma_2',
\end{equation}
where $\Gamma_1'$ has $s$ as its positive special puncture, and since we have assumed that there is no switch in $\Gamma$, the unique true vertex of $\Gamma_2'$ is the positive puncture $a$.

If the standard domain $\Delta(\Gamma)$ has no slit, namely when $\Gamma_1'$ has a unique true vertex $t$, assume that the flow orientation of the edge $\ell$ connecting $s$ to $t$ is $\frac{\partial}{\partial x_1}$. The sign $\nu_\mathit{int}(\Gamma)$ is determined by the formula
\begin{equation}\label{eq:int0}
o_\mathit{int}(o(\mathit{FO}_t(\Gamma_1')),o(\mathit{FO}_t(\Gamma_2')))=\nu_\mathit{int}(\Gamma)\cdot\frac{\partial}{\partial x_1}.
\end{equation}

When the standard domain $\Delta(\Gamma)$ has at least one slit, which corresponds to the vertex $c\in\Gamma$. In the above, we have inductively defined the orientations $o(\mathit{IM}_c(\Gamma_1'))$ and $o(\mathit{FO}_s(\Gamma_2'))$. Using flat coordinates along the edge of $\Gamma$ connecting $c$ to $s$, we can identify $T_s\mathit{IM}_c(\Gamma_1')$ with $T_c\mathit{IM}_c(\Gamma_1')$, and $\nu_\mathit{int}(\Gamma)$ is defined via the formula
\begin{equation}\label{eq:int}
T_sS\wedge T_sS=\nu_\mathit{int}(\Gamma)\cdot o(\mathit{IM}_c(\Gamma_1'))\wedge o(\mathit{FO}_s(\Gamma_2')).
\end{equation}
\bigskip

The sign $\nu_\mathit{end}$ records the information of $e$-vertices in $\Gamma$. Since its explicit form will not be needed for our purposes, we are not going to recall its definition, see $\cite{ck2}$ for details.
\bigskip

The sign $\nu_\mathit{stab}$ is not defined explicitly, but can be determined by referring to the capping exact sequence of Morse flow trees, see Section 6.2 of $\cite{ck1}$. Basically, it is defined as a sum of contributions from punctures and internal vertices of different types. More precisely, since we have assumed the non-existence of switches and $Y_1$-vertices, we only need to consider the sets $\mathcal{P}^+,\mathcal{P}^-,\mathcal{P}^2$, and $\mathcal{Y}_0$. For any vertex $c$ in these sets, there is a well-defined sign $\sigma_{\mathcal{P}^+}(c), \sigma_{\mathcal{P}^-}(c), \sigma_{\mathcal{P}^2}(c)$ and $\sigma_{\mathcal{Y}_0}(c)$ respectively. With these notations, the sign $\nu_\mathit{stab}$ is given by
\begin{equation}\label{eq:stab}
\nu_\mathit{stab}(\Gamma)=(-1)^{\sigma_{\mathcal{P}^+}(a)+\sum_{c\in\mathcal{P}^-}\sigma_{\mathcal{P}^-}(c)+\sum_{c\in\mathcal{P}^2}\sigma_{\mathcal{P}^2}(c)+\sum_{c\in\mathcal{Y}_0}\sigma_{\mathcal{Y}_0}(c)}.
\end{equation}
The precise forms of the signs $\sigma_{\mathcal{P}^+}$, $\sigma_{\mathcal{P}^-}$, and $\sigma_{\mathcal{Y}_0}$ will be recalled below.
\bigskip

Fix a Maslov potential $\mu:\Lambda\rightarrow\mathbb{Z}$. For $c\in\Gamma$ a puncture or a special puncture, let $|\mu(c)|\in\mathbb{Z}/2$ denote the parity of its Maslov index. When $c$ is a critical point of some local difference function $f_{ij}$, we have a well-defined Morse index $\mathit{ind}(c)$.

If $\Gamma'\subset\Gamma$ is a sub flow tree with special puncture at $s\in\Gamma$, denote by $\mathit{bm}(\Gamma')$ the number of boundary minima in the standard domain $\Delta(\Gamma')$. If $\mathit{bm}(\Gamma')>0$, let $\mathit{ord}(\Gamma')$ be the order of the boundary minimum of $\Delta(\Gamma')$ with the smallest $\kappa$-value. When $\mathit{bm}(\Gamma')=0$, we set $\mathit{ord}(\Gamma')=0$.

We introduce a subspace $\ker_s(\Gamma')\subset T_s\mathit{FO}_s(\Gamma')$, called the \textit{true kernel} of $\Gamma'$, which is defined inductively as follows:
\begin{itemize}
	\item when $t\in\mathcal{P}^+$, $\ker_s(\Gamma')=T_s\mathit{FO}_s(\Gamma')$;
	\item when $t\in\mathcal{P}^-$, $\ker_s(\Gamma')=T_s\mathit{FO}_s(\Gamma')$;
	\item when $t\in\mathcal{E}$, $\ker_s(\Gamma')=T_s\mathit{FO}_s(\Gamma')$;
	\item when $t\in\mathcal{P}^2$, $\ker_s(\Gamma')=0$;
	\item when $t\in\mathcal{Y}_0$, $\ker_s(\Gamma')=\ker_t(\Gamma_1')\cap\ker_t(\Gamma_2')$.
\end{itemize}
In the above, $t$ is the true vertex of $\Gamma'$ which is connected to the special puncture $s$ by an edge $\ell\subset\Gamma'$, $\Gamma_1',\Gamma_2'\subset\Gamma$ are sub flow trees starting at the $Y_0$-vertex $t$, and we have used parallel transport to identify tangent spaces of flow-outs at $s$ and $t$. As usual, we have only recalled the definitions in the cases that are relevant to us.
\bigskip

Having introduced the notions above, the term $\sigma_{\mathcal{P}^-}(c)$ appeared in (\ref{eq:stab}) is a $\mathbb{Z}/2$-valued function which depends only on the Morse index and the parity of the Maslov index, namely
\begin{equation}\label{eq:p-}
\sigma_{\mathcal{P}^-}(c)=\sigma_{\mathcal{P}^-}(|\mu(c)|,\mathit{ind}(c)).
\end{equation}

For a $Y_0$-vertex $c$, we shall modify the sub flow trees $\Gamma_1'$ and $\Gamma_2'$ above by cutting them a little bit earlier at the special punctures $s_1$ and $s_2$. There is an associated function $\sigma_0:\mathbb{Z}/2\times\mathbb{Z}/2\rightarrow\mathbb{Z}/2$ depending on the parity of Maslov indices of $s_1$ and $s_2$ such that
\begin{equation}\label{eq:y0}
\sigma_{\mathcal{Y}_0}(c)=\sigma_0(|\mu(s_1)|,|\mu(s_2)|)+\eta+\mathit{ord}(\Gamma_1')+\mathit{ord}(\Gamma_2')+\mathit{bm}(\Gamma_1')+\mathit{bm}(\Gamma_2'),
\end{equation}
where
\begin{equation}
\begin{aligned}
\eta={} & e(\Gamma_1')\cdot(\mathit{bm}(\Gamma_2')+e(\Gamma_2')+1)+\dim\mathit{FO}_c(\Gamma_1')\cdot(\mathit{bm}(\Gamma_2')+|\mu(s_2)|+1) \\
        & +\mathit{bm}(\Gamma_1')\cdot(|\mu(s_2)|+\dim\mathit{FO}_c(\Gamma_2')+1)+\dim\ker_a(\Gamma)+\dim\ker_{s_1}(\Gamma_1') \\
        & +\dim\ker_{s_2}(\Gamma_2')+|\mu(s_1)|\cdot(1+|\mu(s_2)|+\mathit{bm}(\Gamma_2')).
\end{aligned}
\end{equation}

The case when $c$ is a 2-valent negative puncture can be regarded as the special case of a $Y_0$-vertex with one of the sub flow trees $\Gamma_1'$ and $\Gamma_2'$ being constant. Without loss of generality, we assume that $\Gamma_2'$ is constant. As before, we cut $\Gamma_1'$ a little bit earlier at the special puncture $s_1$ instead of $c$, so that the negative punctures of $\Gamma_1'$ are given by $b_1,\dots,b_l$. If the punctures of $\Delta(\Gamma')$ are ordered as $s_1,b_1,\dots,b_l$, let $\mathit{tp}(c)=1$; if the punctures of $\Delta(\Gamma_1')$ are ordered as $b_1,\dots,b_l,s_1$, let $\mathit{tp}(c)=2$. There is a function $\sigma_1:\mathcal{P}^2\rightarrow\mathbb{Z}/2$ which depends on $|\mu(c)|$, $|\mu(s_1)|$ and $\mathit{tp}(c)$, with which the sign $\sigma_{\mathcal{P}^2}(c)$ has the form
\begin{equation}
\sigma_{\mathcal{P}^2}(c)=\mathit{bm}(\Gamma_1')\cdot|\mu(c)|+\dim\ker_{s_1}(\Gamma_1')+\sigma_1(c)+\mathit{ord}(\Gamma_1')+\mathit{bm}(\Gamma_1').
\end{equation}

Finally, for the positive puncture $a\in\mathcal{P}^+$, there is a number $\sigma_2(a)\in\mathbb{Z}/2$ which depends only on the Morse index and the parity of the Maslov index of $a$, namely
\begin{equation}\label{eq:2}
\sigma_2(a)=\sigma_2(|\mu(a)|,\mathit{ind}(a)).
\end{equation}
By our assumption, there is no switch in $\Gamma$. In the case when the sub flow tree $\Gamma_1'$ has only one true vertex, the definition of $\sigma_{\mathcal{P}^+}(a)$ simplifies to
\begin{equation}
\sigma_{\mathcal{P}^+}(a)=\sigma_2(a)+(\dim\mathcal{U}(a)+1)|\mu(a)|.
\end{equation}

Having recalled the definitions of the individual signs $\nu_\mathit{int},\nu_\mathit{end}$ and $\nu_\mathit{stab}$, the main result established in $\cite{ck1,ck2}$ implies the following:
\begin{theorem}[Karlsson]\label{theorem:sign}
Let $\Lambda$ be a link of Legendrian spheres in $J^1(S)$, and assume that we have fixed all the initial orientation choices. Let $\mathcal{M}_\Lambda$ be the moduli space of rigid Morse flow trees determined by $\Lambda$. Then there is a coherent orientation on $\mathcal{M}_\Lambda$, with respect to which the sign $\varepsilon(\Gamma)$ of $\Gamma\in\mathcal{M}_\Lambda$ is given by
\begin{equation}
\varepsilon(\Gamma)=\nu_\mathit{int}(\Gamma)\cdot\nu_\mathit{end}(\Gamma)\cdot\nu_\mathit{stab}(\Gamma).
\end{equation}
\end{theorem}
\bigskip

As an illustration, we compute here the signs of two elementary Morse flow trees in $\mathbb{R}^2$. The first example is a rigid Morse flow tree $\Gamma\in\mathcal{M}_\Lambda(a;b_1,b_2)$ which has a unique internal $Y_0$-vertex $c$, see Figure \ref{fig:y0tree}. Furthermore, we assume that $a$ is a saddle point of some difference function $f_{ij}=f_i-f_j$, $b_1$ is a local minimum of some $f_{ik}$, while $b_2$ is a saddle point of some $f_{kj}$, where $f_i>f_k>f_j$.
	
In order to compute $\nu_\mathit{int}(\Gamma)$, we cut $\Gamma$ into two sub flow trees $\Gamma_1'$ and $\Gamma_2'$ at the special puncture $s$, where $s$ is contained in the same elementary region as $c$. Let $\ell_1\subset\Gamma_1'$ be the edge which connects $s$ to $c$, and let $\ell_2\subset\Gamma_2'$ be the edge that connects $a$ to $s$. As indicated in Figure \ref{fig:y0tree}, the flow orientation of the sub flow tree $\Gamma_2'$ is fixed to be $\frac{\partial}{\partial x_1}$ for convenience.

\begin{figure}
	\centering
	\begin{tikzpicture}
	\draw [red] [decoration={markings, mark=at position 1/2 with {\arrow{>}}},postaction={decorate}] (0,0) to (2,0);
	\draw [blue] [decoration={markings, mark=at position 1/2 with {\arrow{>}}},postaction={decorate}] (2,0) to (3,0);
	\draw [blue] [decoration={markings, mark=at position 1/2 with {\arrow{>}}},postaction={decorate}] (3,0) to (4.5,1.5);
	\draw [blue] [decoration={markings, mark=at position 1/2 with {\arrow{>}}},postaction={decorate}] (3,0) to (4.5,-1.5);
	\draw [red] (0,0) node[circle,fill,inner sep=1pt]{};
	\draw (2,0) node[circle,fill,inner sep=1pt]{};
	\draw [blue] (3,0) node[circle,fill,inner sep=1pt]{};
	\draw [blue] (4.5,1.5) node[circle,fill,inner sep=1pt]{};
	\draw [blue] (4.5,-1.5) node[circle,fill,inner sep=1pt]{};
	\node [red] at (0,-0.2) {$a$};
	\node at (2,-0.2) {$s$};
	\node [blue] at (3,-0.2) {$c$};
	\node [blue] at (4.75,1.5) {$b_2$};
	\node [blue] at (4.75,-1.5) {$b_1$};
	\node [red] at (1,-0.25) {$\Gamma_2'$};
	\node [blue] at (4.75,0) {$\Gamma_1'$};
	\node [blue] at (3.75,1.2) {$\Gamma_{12}'$};
	\node [blue] at (3.75,-1.2) {$\Gamma_{11}'$};
	\end{tikzpicture}
	\centering
	\caption{The Morse flow tree $\Gamma$ and its associated sub flow trees $\Gamma_1'$ and $\Gamma_2'$}
	\label{fig:y0tree}
\end{figure}
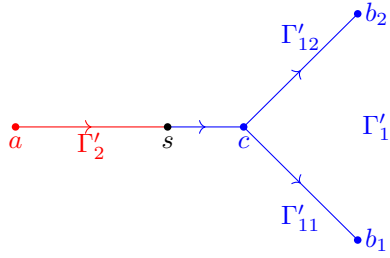

The sub flow tree $\Gamma_1'$ contains two sub flow trees $\Gamma_{11}'$ and $\Gamma_{12}'$, where $\Gamma_{11}'$ is the $-\nabla f_{ik}$ flow line from $c$ to $b_1$, and $\Gamma_{12}'$ is the $-\nabla f_{kj}$ flow line from $c$ to $b_2$, see Figure \ref{fig:y0tree}. Since the true vertex associated to $\Gamma_{11}'$ is $b_1\in\mathcal{P}^-$, the flow-out $\mathit{FO}_c(\Gamma_{11}')$ is by definition the flow out of $b_1$ along the flow line $\Gamma_{11}'$. Similarly, $\mathit{FO}_c(\Gamma_{12}')$ is the flow-out of $b_2$ along $\Gamma_{12}'$. Choosing an interval $I_c$ centering at $c$ which is transverse to $\ell_1$, by our assumption that $b_2$ is a saddle point, it follows that
\begin{equation}
\mathit{IM}_c(\Gamma_1')\cap I_c=\{c\}.
\end{equation}
The flow-out $\mathit{FO}_s(\Gamma_1')$ is then by definition $\ell_1\cup\ell_2$.
On the other hand, the associated true vertex of the sub flow tree $\Gamma_2'$ is $a\in\mathcal{P}^+$. Since $a$ is by assumption a saddle point, we have $\mathit{FO}_s(\Gamma_2')=\ell_1\cup\ell_2$.

Identify the tangent space $T_sS$ with $\mathbb{R}^2$ equipped with its standard orientation. By definition, $o(\mathit{FO}_c(\Gamma_{11}'))$ is the orientation of $T_c\mathcal{S}(b_1)\cong T_cS$, and $o(\mathit{FO}_c(\Gamma_{12}'))$ is the orientation of $T_c\mathcal{S}(b_2)$. This enables us to compute the orientation of the intersection manifold:
\begin{equation}
o(\mathit{IM}_c(\Gamma_1')):=o_\mathit{int}(o(\mathit{FO}_c(\Gamma_{11}')),o(\mathit{FO}_c(\Gamma_{12}')))=o(\Gamma_{12}'),
\end{equation}
where $o(\Gamma_{12}')$ is the flow orientation of $\Gamma_{12}'$ at $c$. Since the true vertex $a$ of $\Gamma_2'$ is a saddle point, by our convention
\begin{equation}
o(\mathit{FO}_s(\Gamma_2'))=o_\mathit{cap}(T_s\mathcal{U}(a))=\frac{\partial}{\partial x_1}.
\end{equation}
Since the standard domain $\Delta(\Gamma)$ has a unique slit, and $o(\mathit{IM}_c(\Gamma_1'))\wedge o(\mathit{FO}_s(\Gamma_2'))$ gives the standard orientation on $T_sS$, we get from (\ref{eq:int}) that
\begin{equation}
\nu_\mathit{int}(\Gamma)=-1.
\end{equation}
Note that if $\Gamma^\mathit{op}\in\mathcal{M}_\Lambda(a;b_2,b_1)$ is a rigid Morse flow tree obtained from $\Gamma$ by exchanging the labellings of $\Gamma_{11}'$ and $\Gamma_{12}'$, then $\nu_\mathit{int}(\Gamma^\mathit{op})=\nu_\mathit{int}(\Gamma)$.

We now compute $\nu_\mathit{stab}(\Gamma)$. It follows immediately from our assumptions that
\begin{equation}
\sigma_{\mathcal{P}^-}(b_1)=\sigma_{\mathcal{P}^-}(|\mu(b_1)|,0),\sigma_{\mathcal{P}^-}(b_2)=\sigma_{\mathcal{P}^-}(|\mu(b_2)|,1),
\end{equation}
and
\begin{equation}
\sigma_{\mathcal{P}^+}(a)=\sigma_2(|\mu(a)|,1)+2|\mu(a)|.
\end{equation}
In order to determine $\sigma_{\mathcal{Y}_0}(c)$, consider the sub flow trees $\Gamma_1''$ and $\Gamma_2''$ of $\Gamma$ depicted as in Figure \ref{fig:y0tree1}. Since the standard domains $\Delta(\Gamma_1'')$ and $\Delta(\Gamma_2'')$ do not contain any boundary minimum, we have
\begin{equation}\label{eq:bm}
\mathit{bm}(\Gamma_1'')=\mathit{bm}(\Gamma_2'')=0,
\end{equation}
and
\begin{equation}\label{eq:ord}
\mathit{ord}(\Gamma_1'')=\mathit{ord}(\Gamma_2'')=0.
\end{equation}
By definition of true kernels, we have
\begin{equation}
\ker_a(\Gamma)=\ker_c(\Gamma_{11}')\cap\ker_c(\Gamma_{12}').
\end{equation}
Since the true vertices associated to $\Gamma_{11}'$ and $\Gamma_{12}'$ are the negative punctures $b_1$ and $b_2$ respectively, we have
\begin{equation}
\ker_c(\Gamma_{11}')=T_c\mathit{FO}_c(b_1),\ker_c(\Gamma_{12}')=T_c\mathit{FO}_c(b_2).
\end{equation}
Similarly,
\begin{equation}
\ker_{s_1}(\Gamma_1'')=T_{s_1}\mathit{FO}_{s_1}(b_1),\ker_{s_2}(\Gamma_2'')=T_{s_2}\mathit{FO}_{s_2}(b_2).
\end{equation}
From the above we deduce that
\begin{equation}
\dim\ker_a(\Gamma)=1,\dim\ker_{s_1}(\Gamma_1'')=2,\dim\ker_{s_2}(\Gamma_2'')=1,
\end{equation}
and as a consequence,
\begin{equation}
\eta=2(|\mu(b_2)|+1)+4+|\mu(b_1)|\cdot(1+|\mu(b_2)|),
\end{equation}
where we have used the fact that
\begin{equation}
\mu(s_1)=\mu(b_1),\mu(s_2)=\mu(b_2).
\end{equation}
Combining with (\ref{eq:bm}) and (\ref{eq:ord}), we get
\begin{equation}
\sigma_{\mathcal{Y}_0}(c)=\sigma_0(|\mu(b_1)|,|\mu(b_2)|)+|\mu(b_1)|\cdot(1+|\mu(b_2)|).
\end{equation}
By (\ref{eq:stab}),
\begin{equation}
\nu_\mathit{stab}(\Gamma)=(-1)^{\sigma_0(|\mu(b_1)|,|\mu(b_2)|)+|\mu(b_1)|\cdot(1+|\mu(b_2)|)+\sigma_2(|\mu(a)|,1)+\sigma_{\mathcal{P}^-}(|\mu(b_1)|,0)+\sigma_{\mathcal{P}^-}(|\mu(b_2)|,1)}.
\end{equation}
\bigskip

\begin{figure}
	\centering
	\begin{tikzpicture}
	\draw [decoration={markings, mark=at position 1/2 with {\arrow{>}}},postaction={decorate}] (0,0) to (2.5,0);
	\draw (2.5,0) to (3,0.5);
	\draw (2.5,0) to (3,-0.5);
	\draw [green] [decoration={markings, mark=at position 1/2 with {\arrow{>}}},postaction={decorate}] (3,0.5) to (4.5,2);
	\draw [orange] [decoration={markings, mark=at position 1/2 with {\arrow{>}}},postaction={decorate}] (3,-0.5) to (4.5,-2);
	\draw (0,0) node[circle,fill,inner sep=1pt]{};
	\draw (2.5,0) node[circle,fill,inner sep=1pt]{};
	\draw [green] (3,0.5) node[circle,fill,inner sep=1pt]{};
	\draw [orange] (3,-0.5) node[circle,fill,inner sep=1pt]{};
	\draw [green] (4.5,2) node[circle,fill,inner sep=1pt]{};
	\draw [orange] (4.5,-2) node[circle,fill,inner sep=1pt]{};
	\node at (0,-0.2) {$a$};
	\node at (2.4,-0.2) {$c$};
	\node [green] at (4.75,2) {$b_2$};
	\node [orange] at (4.75,-2) {$b_1$};
	\node [green] at (3.25,1.2) {$\Gamma_2''$};
	\node [orange] at (3.25,-1.2) {$\Gamma_1''$};
	\node [orange] at (2.75,-0.5) {$s_1$};
	\node [green] at (2.75,0.5) {$s_2$};
	\end{tikzpicture}
	\centering
	\caption{The sub flow trees $\Gamma_1''$ and $\Gamma_2''$}
	\label{fig:y0tree1}
\end{figure}
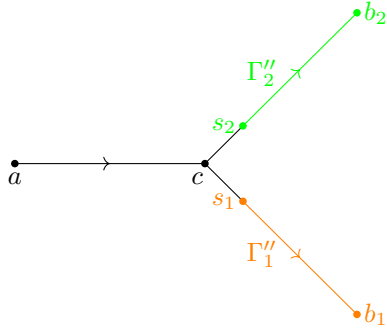

Our second example deals with the case of a rigid Morse flow tree $\Xi\in\mathcal{M}_\Lambda(a;b,b^\ast)$ with a negative 2-valent puncture $b^\ast$, see the left-hand side of Figure \ref{fig:y0tree2}. We assume that the positive puncture $a$ is a maximum of some Morse function $f_{ij}$, the negative 1-valent puncture $b$ is a minimum of $f_{ik}$, and $b^\ast$ is a maximum of $f_{kj}$.

In this case, the sub flow tree $\Xi_2'\subset\Xi$ is a $-\nabla f_{ij}$ flow line from $a$ to $s$, and the sub flow tree $\Xi_1'$ inherits the 2-valent puncture $b^\ast$. Since the standard domain $\Delta(\Xi)$ still has a unique slit, the intersection orientation is determined by requiring that
\begin{equation}
\nu_\mathit{int}(\Xi)\cdot o(\mathit{IM}_{b^\ast}(\Xi_1'))\wedge o(\mathit{FO}_s(\Xi_2'))
\end{equation}
recovers the orientation of $T_sS$. To compute $o(\mathit{IM}_{b^\ast}(\Xi_1'))$, consider the sub flow tree $\Xi_{11}'\subset\Xi_1'$, which is the $-\nabla f_{kj}$ flow line from $b^\ast$ to $b$. By definition, the flow out $\mathit{FO}_{b^\ast}(\Xi_{11}')$ is the flow out of $b$ along $\Xi_{11}'$. Since $b$ is a minimum, we have by definition that
\begin{equation}
o(\mathit{IM}_{b^\ast}(\Xi_1'))=o_\mathit{int}(o_\mathit{cap}(T_{b^\ast}\mathcal{S}(b))),o_\mathit{cap}(T_{b^\ast}\mathcal{U}(b^\ast))=o(T_{b^\ast}S).
\end{equation}
On the other hand, the true vertex of $\Xi_2'$ is the positive puncture $a$, which shows that
\begin{equation}
o(\mathit{FO}_s(\Xi_2'))=o(T_s\mathcal{U}(a))=o(T_s S).
\end{equation}
By (\ref{eq:int}), we have
\begin{equation}
\nu_\mathit{int}(\Xi)=1.
\end{equation}

The computation of $\nu_\mathit{stab}(\Xi)$ involves the determination of the individual signs $\sigma_{\mathcal{P}^+}(a)$, $\sigma_{\mathcal{P}^-}(b)$ and $\sigma_{\mathcal{P}^2}(b^\ast)$. Since $b$ is a minimum, we have
\begin{equation}
\sigma_{\mathcal{P}^-}(b)=\sigma_{\mathcal{P}^-}(|\mu(b)|,0).
\end{equation}
Similarly, since $a$ is a maximum, we get
\begin{equation}
\sigma_{\mathcal{P}^+}(a)=\sigma_2(|\mu(a)|,2)+3|\mu(a)|.
\end{equation}
In order to compute $\sigma_{\mathcal{P}^2}(b^\ast)$, consider the sub flow tree $\Xi_{11}'\subset\Xi_1'$. We cut $\Xi_{11}'$ a little bit earlier so that it starts at a special puncture $s_1$ instead of $b^\ast$, and denote the resulting sub flow tree by $\Xi_1''$. Since the standard domain $\Delta(\Gamma_{11}')$ has no slit, we have
\begin{equation}
\mathit{bm}(\Xi_1'')=\mathit{ord}(\Xi_1'')=0.
\end{equation}
Since $b^\ast\in\mathcal{P}^2$, it follows that $\ker_{s_1}(\Xi_1'')=0$, and
\begin{equation}
\sigma_{\mathcal{P}^2}=\sigma_1(b^\ast)=\sigma_1(|\mu(b^\ast)|,|\mu(b)|,2).
\end{equation}
In conclusion,
\begin{equation}
\nu_\mathit{stab}(\Xi)=(-1)^{\sigma_{\mathcal{P}^-}(|\mu(b)|,0)+\sigma_1(|\mu(b^\ast)|,|\mu(b)|,2)+\sigma_2(|\mu(a)|,2)+|\mu(a)|}.
\end{equation}

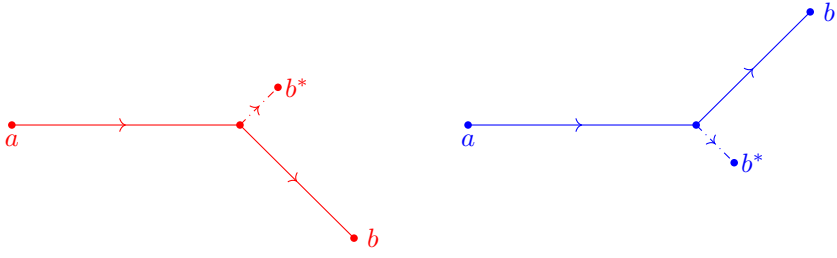
\begin{figure}
	\centering
	\begin{tikzpicture}
	\draw [red] [decoration={markings, mark=at position 1/2 with {\arrow{>}}},postaction={decorate}] (0,0) to (3,0);
	\draw [red,dash dot] [decoration={markings, mark=at position 1/2 with {\arrow{>}}},postaction={decorate}] (3,0) to (3.5,0.5);
	\draw [red] [decoration={markings, mark=at position 1/2 with {\arrow{>}}},postaction={decorate}] (3,0) to (4.5,-1.5);
	\draw [blue] [decoration={markings, mark=at position 1/2 with {\arrow{>}}},postaction={decorate}] (6,0) to (9,0);
	\draw [blue] [decoration={markings, mark=at position 1/2 with {\arrow{>}}},postaction={decorate}] (9,0) to (10.5,1.5);
	\draw [blue,dash dot] [decoration={markings, mark=at position 1/2 with {\arrow{>}}},postaction={decorate}] (9,0) to (9.5,-0.5);
	\node [red] at (0,-0.2) {$a$};
	\node [blue] at (6,-0.2) {$a$};
	\node [red] at (3.75,0.5) {$b^\ast$};
	\node [red] at (4.75,-1.5) {$b$};
	\node [blue] at (10.75,1.5) {$b$};
	\node [blue] at (9.75,-0.5) {$b^\ast$};
	\draw [red] (0,0) node[circle,fill,inner sep=1pt]{};
	\draw [red] (3,0) node[circle,fill,inner sep=1pt]{};
	\draw [red] (3.5,0.5) node[circle,fill,inner sep=1pt]{};
	\draw [red] (4.5,-1.5) node[circle,fill,inner sep=1pt]{};
	\draw [blue] (6,0) node[circle,fill,inner sep=1pt]{};
	\draw [blue] (9,0) node[circle,fill,inner sep=1pt]{};
	\draw [blue] (10.5,1.5) node[circle,fill,inner sep=1pt]{};
	\draw [blue] (9.5,-0.5) node[circle,fill,inner sep=1pt]{};
	\end{tikzpicture}
	\centering
	\caption{The Morse flow trees $\Xi$ and $\Xi^\mathit{op}$, where the dotted edges are mapped to constant}
	\label{fig:y0tree2}
\end{figure}

Let $\Xi^\mathit{op}\in\mathcal{M}_\Lambda(a;b^\ast,b)$ be the rigid Morse flow tree whose fundamental domain $\Delta(\Xi^\mathit{op})$ has $b^\ast$ as its lower puncture, see the right-hand side of Figure \ref{fig:y0tree2}. In this case $\mathit{tp}(b^\ast)=1$ and we conclude that $\varepsilon(\Xi)=-\varepsilon(\Xi^\mathit{op})$.

\subsection{Adaptation to the case of $\Lambda_{p,q,r}$}\label{section:sign}

Recall that in $\cite{rs2}$, the cellular dg algebra $\mathcal{C}(\Lambda)$ of $\Lambda\subset J^1(S)$ is shown to be quasi-isomorphic to the Chekanov-Eliashberg dg algebra $\mathcal{C}E(\Lambda)$ over $\mathbb{Z}/2$ by fixing a \textit{transverse square decomposition} $\mathcal{E}_\pitchfork$ of $S$ and doing local analysis of Morse flow trees for each elementary square. In our specific case when $\Lambda=\Lambda_{p,q,r}$ with $p,q,r\geq2$, it is not hard to see from Figure \ref{figure:decomposition(2,2,2)} that one can find a transverse square decomposition $\mathcal{E}_\pitchfork$ for $\Lambda_{p,q,r}$ such that only Type I, Type II, Type III, Type IV and Type IX squares appear in $\mathcal{E}_\pitchfork$ as elementary squares, see Figure \ref{fig:square}. From now on, fix such a transverse square decomposition.
\bigskip

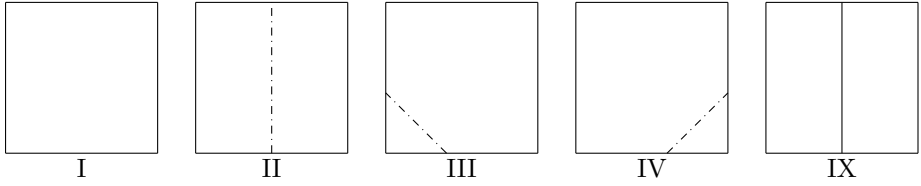
\begin{figure}
	\centering
	\begin{tikzpicture}
	\draw (0,0) to (0,2);
	\draw (0,0) to (2,0);
	\draw (0,2) to (2,2);
	\draw (2,0) to (2,2);
	\node at (1,-0.2) {I};
	
	\draw (2.5,0) to (4.5,0);
	\draw (2.5,0) to (2.5,2);
	\draw (2.5,2) to (4.5,2);
	\draw (4.5,0) to (4.5,2);
	\draw [dash dot] (3.5,0) to (3.5,2);
	\node at (3.5,-0.2) {II};
	
	\draw (5,0) to (7,0);
	\draw (5,0) to (5,2);
	\draw (5,2) to (7,2);
	\draw (7,0) to (7,2);
	\draw [dash dot] (5.8,0) to (5,0.8);
	\node at (6,-0.2) {III};
	
	\draw (7.5,0) to (9.5,0);
	\draw (7.5,0) to (7.5,2);
	\draw (7.5,2) to (9.5,2);
	\draw (9.5,0) to (9.5,2);
	\draw [dash dot] (8.7,0) to (9.5,0.8);
	\node at (8.5,-0.2) {IV};
	
	\draw (10,0) to (12,0);
	\draw (10,0) to (10,2);
	\draw (12,0) to (12,2);
	\draw (10,2) to (12,2);
	\draw (11,0) to (11,2);
	\node at (11,-0.2) {IX};
	\end{tikzpicture}
	\centering
	\caption{The elementary squares of Types I, II, III, IV and IX}
	\label{fig:square}
\end{figure}

Figure \ref{fig:square} depicts the transverse square decomposition in a neighborhood of the 2-cell $e_7^2$ in the $\Lambda_{p,q,r}$-compatible polygonal decomposition of Figure \ref{figure:decomposition(2,2,2)}, which consists of a single Type I square $\square_7\cong[-1,1]^2$, and four Type II squares, two type III squares and two Type IV squares. Denote by $f_m:\square_7\rightarrow\mathbb{R}$, the defining functions of the sheets $S_m$ above $\square_7$, which are labelled to satisfy $f_1>\dots>f_8$. Let $f_{mn}=f_m-f_n$ when $m<n$. Consider defining functions $f_m$ which are of the form
\begin{equation}
f_m(x_1,x_2)=h_m(x_1)+h_m(x_2),
\end{equation}
where the $h_m$ are arranged so that $h_m-h_n$ has local minima at $-1$ and 1, and a single local maximum $\beta_{i,j}\in(-1,1)$, such that
\begin{equation}
\beta_{1,2}<\beta_{1,3}<\dots<\beta_{2,3}<\beta_{2,4}<\dots<\beta_{7,8}.
\end{equation}
The Reeb chords in $\square_7$ correspond to critical points of the functions $f_{mn}$. By abuse of notations, we shall denote these critical points by
\begin{equation}
a^{m,n}_{\pm,\pm},b_U^{m,n},b_R^{m,n},b_D^{m,n},b_L^{m,n},c_7^{m,n},
\end{equation}
where $a^{m,n}_{\pm,\pm}$ are the four corners of $\square_7$, and they are minima of $f_{mn}$, $b_U^{m,n},b_R^{m,n},b_D^{m,n},b_L^{m,n}$ are the saddle points of $f_{mn}$ located on the edges of $\square_7$, and $c_7^{m,n}$ is a local maximum of $f_{mn}$ lying in the interior of $\square_7$. By the analysis of $\cite{rs2}$, except for the $-\nabla f_{m,n}$ flow lines from $c_7^{m,n}$ to $b_U^{m,n},b_R^{m,n},b_D^{m,n}$ and $b_L^{m,n}$, there are four additional rigid Morse flow trees with positive puncture at $c_7^{m,n}$, see Figure \ref{fig:type1}. The first two Morse flow trees have a 2-valent negative puncture at $c_7^{k,n}$ and $c_7^{m,k}$ respectively, while the last two Morse flow trees have a unique $Y_0$-vertex. This implies that

\begin{figure}
	\centering
	\begin{tikzpicture}
	\draw [dash dot] (0,0) circle [radius=1.5];
	\node at (0,0) {\footnotesize $C_7$};
	\draw [orange] (-0.75,-0.75) to (0.75,-0.75);
	\draw [orange] (-0.75,-0.75) to (-0.75,0.75);
	\draw [orange] (-0.75,0.75) to (0.75,0.75);
	\draw [orange] (0.75,-0.75) to (0.75,0.75);
    
    \draw [orange] (-2.25,-0.75) to (-0.75,-0.75);
    \draw [orange] (-2.25,-0.75) to (-2.25,0.75);
    \draw [orange] (-2.25,0.75) to (-0.75,0.75);
    
    \draw [orange] (-2.25,0.75) to (-2.25,2.25);
    \draw [orange] (-2.25,2.25) to (-0.75,2.25);
    \draw [orange] (-0.75,2.25) to (-0.75,0.75);
    
    \draw [orange] (-0.75,2.25) to (0.75,2.25);
    \draw [orange] (0.75,2.25) to (0.75,0.75);
    
    \draw [orange] (0.75,0.75) to (2.25,0.75);
    \draw [orange] (2.25,0.75) to (2.25,-0.75);
    \draw [orange] (0.75,-0.75) to (2.25,-0.75);
    
    \draw [orange] (2.25,0.75) to (2.25,2.25);
    \draw [orange] (2.25,2.25) to (0.75,2.25);
    
    \draw [orange] (-0.75,-0.75) to (-0.75,-2.25);
    \draw [orange] (-0.75,-2.25) to (0.75,-2.25);
    \draw [orange] (0.75,-2.25) to (0.75,-0.75);
    
    \draw [orange] (-0.75,-2.25) to (-2.25,-2.25);
    \draw [orange] (-2.25,-2.25) to (-2.25,-0.75);
    
    \draw [orange] (0.75,-2.25) to (2.25,-2.25);
    \draw [orange] (2.25,-2.25) to (2.25,-0.75);
	\end{tikzpicture}
	\centering
	\caption{The transverse square decomposition $\mathcal{E}_\pitchfork$ in a neighborhood of $e_7^2$}
\end{figure}
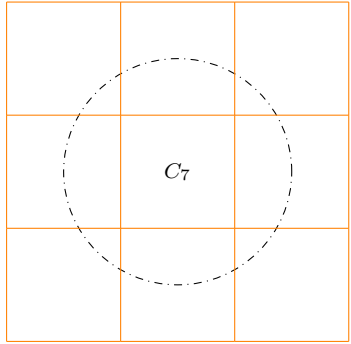

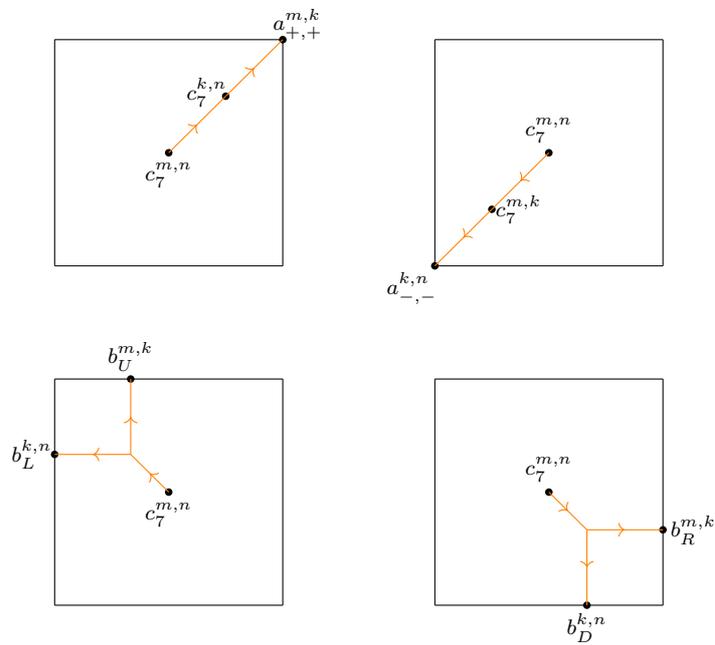
\begin{figure}
	\centering
	\begin{tikzpicture}
	\draw (0,0) to (3,0);
	\draw (0,0) to (0,3);
	\draw (3,0) to (3,3);
	\draw (0,3) to (3,3);
	\draw (1.5,1.5) node[circle,fill,inner sep=1pt]{};
	\draw (2.25,2.25) node[circle,fill,inner sep=1pt]{};
	\draw (3,3) node[circle,fill,inner sep=1pt]{};
	\draw [decoration={markings, mark=at position 1/2 with {\arrow{>}}},postaction={decorate}] [orange] (1.5,1.5) to (2.25,2.25);
	\draw [decoration={markings, mark=at position 1/2 with {\arrow{>}}},postaction={decorate}] [orange] (2.25,2.25) to (3,3);
	\node at (1.5,1.2) {\footnotesize $c_7^{m,n}$};
	\node at (2,2.3) {\footnotesize $c_7^{k,n}$};
	\node at (3.2,3.2) {\footnotesize $a_{+,+}^{m,k}$};
	
	\draw (5,0) to (8,0);
	\draw (5,0) to (5,3);
	\draw (8,0) to (8,3);
	\draw (5,3) to (8,3);
	\draw (5,0) node[circle,fill,inner sep=1pt]{};
	\draw (6.5,1.5) node[circle,fill,inner sep=1pt]{};
	\draw (5.75,0.75) node[circle,fill,inner sep=1pt]{};
	\node at (4.7,-0.3) {\footnotesize $a_{-,-}^{k,n}$};
	\draw [orange] [decoration={markings, mark=at position 1/2 with {\arrow{>}}},postaction={decorate}] (6.5,1.5) to (5.75,0.75);
	\draw [orange] [decoration={markings, mark=at position 1/2 with {\arrow{>}}},postaction={decorate}] (5.75,0.75) to (5,0);
	\node at (6.5,1.8) {\footnotesize $c_7^{m,n}$};
	\node at (6.1,0.75) {\footnotesize $c_7^{m,k}$};
	
	\draw (0,-1.5) to (3,-1.5);
	\draw (0,-1.5) to (0,-4.5);
	\draw (3,-1.5) to (3,-4.5);
	\draw (0,-4.5) to (3,-4.5); 
	\draw (1.5,-3) node[circle,fill,inner sep=1pt]{};
	\draw (1,-1.5) node[circle,fill,inner sep=1pt]{};
	\draw (0,-2.5) node[circle,fill,inner sep=1pt]{};
	\draw [orange] [decoration={markings, mark=at position 1/2 with {\arrow{>}}},postaction={decorate}] (1.5,-3) to (1,-2.5);
	\draw [orange] [decoration={markings, mark=at position 1/2 with {\arrow{>}}},postaction={decorate}] (1,-2.5) to (1,-1.5);
	\draw [orange] [decoration={markings, mark=at position 1/2 with {\arrow{>}}},postaction={decorate}] (1,-2.5) to (0,-2.5);
	\node at (1.5,-3.3) {\footnotesize $c_7^{m,n}$};
	\node at (1,-1.2) {\footnotesize $b_U^{m,k}$};
	\node at (-0.3,-2.5) {\footnotesize $b_L^{k,n}$};
	
	\draw (5,-1.5) to (8,-1.5);
	\draw (8,-1.5) to (8,-4.5);
	\draw (5,-1.5) to (5,-4.5);
	\draw (5,-4.5) to (8,-4.5);
	\draw (6.5,-3) node[circle,fill,inner sep=1pt]{};
	\draw (8,-3.5) node[circle,fill,inner sep=1pt]{};
	\draw (7,-4.5) node[circle,fill,inner sep=1pt]{};
	\draw [orange] [decoration={markings, mark=at position 1/2 with {\arrow{>}}},postaction={decorate}] (6.5,-3) to (7,-3.5);
	\draw [orange] [decoration={markings, mark=at position 1/2 with {\arrow{>}}},postaction={decorate}] (7,-3.5) to (7,-4.5);
	\draw [orange] [decoration={markings, mark=at position 1/2 with {\arrow{>}}},postaction={decorate}] (7,-3.5) to (8,-3.5);
	\node at (6.5,-2.7) {\footnotesize $c_7^{m,n}$};
	\node at (7,-4.8) {\footnotesize $b_D^{k,n}$};
	\node at (8.4,-3.5) {\footnotesize $b_R^{m,k}$};
	\end{tikzpicture}
	\centering
	\caption{Rigid Morse flow trees in a Type I square}
	\label{fig:type1}
\end{figure}

\begin{equation}\label{eq:CE-d}
\begin{aligned}
&\partial c_7^{m,n}=b_U^{m,n}+b_L^{m,n}+b_R^{m,n}+b_D^{m,n} \\
&+\sum_{m<k<n}a_{+,+}^{m,k}c_7^{k,n}+\sum_{m<k<n}c_7^{m,k}a_{-,-}^{k,n}+\sum_{m<k<n}b_U^{m,k}b_L^{k,n}+\sum_{m<k<n}b_R^{m,k}b_D^{k,n}
\end{aligned}
\end{equation}
in the Chekanov-Eliashberg algebra $\mathcal{C}E(\Lambda_{p,q,r})$ over $\mathbb{Z}/2$. Notice that one can recover from (\ref{eq:CE-d}) the formula (\ref{eq:C7}) in the cellular dg algebra $\mathcal{C}(\Lambda_{p,q,r})$. In fact, this follows from the relations
\begin{equation}\label{eq:rel1}
b_L^{m,n}=b_R^{m,n}=b_D^{m,n}=0,
\end{equation}
\begin{equation}\label{eq:rel2}
a_{+,+}^{m,n}=a_{-,-}^{m,n}.
\end{equation}
in the cellular dg algebra $\mathcal{C}_{||}(\Lambda_{p,q,r})$ associated to the cellular decomposition $\mathcal{E}_{||}$ obtained by shifting $p_x(\Sigma)$ into the borders of the elementary squares in the transverse square decomposition $\mathcal{E}_\pitchfork$, see Section 3.6 of $\cite{rs2}$. In particular, the crossing arc corresponding to the 1-cell $e_1^1$ is shifted into the edges of $\square_7$. In fact, the same relations hold in the quotient dg algebra $\mathcal{C}E'(\Lambda_{p,q,r})$ of $\mathcal{C}E(\Lambda_{p,q,r})$ obtained by repeated applications of Lemma \ref{lemma:tame}. This can be checked via the analysis of Morse flow trees with positive punctures at Reeb chords associated to 1-cells in $\mathcal{E}_\pitchfork$, which will be discussed below.

By (\ref{eq:rel1}) and (\ref{eq:rel2}), which also hold over an arbitrary field $\mathbb{K}$, we see that in order to determine $\partial c_7^{m,n}$ in $\mathcal{C}E'(\Lambda_{p,q,r})$ over $\mathbb{K}$, it suffices to determine the signs of the flow line from $c_7^{m,n}$ to $b_U^{m,n}$, and the first two rigid Morse flow trees in Figure \ref{fig:type1}. Denote by $\varepsilon_0(m,n)$, $\varepsilon_1(m,k,n)$ and $\varepsilon_2(m,k,n)$ the signs of these trees, we have in $\mathcal{C}E'(\Lambda_{p,q,r})$ the formula
\begin{equation}\label{eq:grC7}
\begin{aligned}
\partial c_7^{m,n}&=\varepsilon_0(m,n)b_1^{\sigma_0(m),\sigma_0(n)}+\sum_{m<k<n}\varepsilon_1(m,k,n)a_7^{\sigma_0(m),\sigma_0(k)}c_7^{k,n} \\
&+\sum_{m<k<n}\varepsilon_2(m,k,n)c_7^{m,k}a_7^{\sigma_0(k),\sigma_0(n)},
\end{aligned}
\end{equation}
where we have changed our notations from $b_U^{m,n}$ and $a_{+,+}^{m,n}=a_{-,-}^{m,n}$ to $b_1^{\sigma_0(m),\sigma_0(n)}$ and $a_7^{\sigma_0(m),\sigma_0(n)}$, so that it coincides with the ones used in (\ref{eq:C7}) for the cellular dg algebra. However, since we are dealing with the more general case of $\Lambda_{p,q,r}$ instead of $\Lambda_{2,2,2}$, the permutation $\sigma_0$ in the above formula is a composition of $(4,5)$ with a sequence of transpositions associated to the $A_{p-1}$-chain of unknots $\{\Lambda_{P_i}\}$, see Section \ref{section:example}. The signs $\varepsilon_0(m,n)$, $\varepsilon_1(m,k,n)$ and $\varepsilon_2(m,k,n)$ can be explicitly computed using our discussions in Section \ref{section:subo}.

One can do similar analysis for the elementary squares $\square_{11}$ and $\square_{12}$ in $\mathcal{E}_\pitchfork$.
\bigskip

We also need to consider rigid Morse flow trees associated to generators corresponding to 1-cells. Let $e_\alpha^1\cong[-1,1]$ be a 1-cell in the transverse square decomposition $\mathcal{E}_\pitchfork$, with endpoints at the 0-cells $e_-^0$ and $e_+^0$. For each 0-cell $e_\beta^0$ in $\mathcal{E}_\pitchfork$, let $U(e_\beta^0)$ be a closed disc centered at $e_\beta^0$ with radius $\frac{1}{16}$. We can arrange so that there is a unique Reeb chord $a_\beta^{m,n}$ associated to every pair of sheets $S_m$ and $S_n$ with $z(S_m)>z(S_n)$, which corresponds to a minimum of $f_{mn}$, and the gradient $-\nabla f_{mn}$ points inward along $\partial U(e_\beta^0)$.

Let $U(e_\alpha^1)$ be a neighborhood of $e_\alpha^1$, which is depicted in Figure \ref{fig:1-cell} as the region bounded by the brown solid curve. It can be realized as a union $U(e_-^0)\cup\widehat{U}(e_\alpha^1)\cup U(e_+^0)$, where $\widehat{U}(e_\alpha^1)$ consists of a portion of $e_\alpha^1$ away from the endpoints together with parts of two elementary squares in $\mathcal{E}_\pitchfork$ which have $e_\alpha^1$ as their common boundary. We require that the boundary of $\widehat{U}(e_\alpha^1)$, which consists of a union of two paths $\gamma_1$ and $\gamma_2$, is contained within a distance of $\frac{1}{32}$ from $e_\alpha^1$, and both of $\gamma_1$ and $\gamma_2$ are parallel to $e_\alpha^1$.

In order to simplify the analysis of the rigid Morse flow trees in $U(e_\alpha^1)$, one needs to choose carefully the local defining functions $f_m$ of the sheets $S_m\subset\Lambda_{p,q,r}$, see Section 11 of $\cite{rs2}$ for details. For any point $(x_1,x_2)\in\gamma_i$ such that $f_{mn}(x_1,x_2)>0$, we require that $-\nabla f_{mn}(x_1,x_2)$ is transverse to $\gamma_i$ at $(x_1,x_2)$ and is inward pointing, so that the Reeb chords in $U(e_\alpha^1)$ together with their differentials form a dg sub algebra of $\mathcal{C}E(\Lambda_{p,q,r})$. The only Reeb chords in $\widehat{U}(e_\alpha^1)$ are:
\begin{itemize}
	\item a Reeb chord with endpoints on the sheets $S_m$ and $S_n$ above $e_\alpha^1$, which corresponds to a saddle point $b_\alpha^{m,n}$ of $f_{mn}$, where $z(S_m)>z(S_n)$;
	\item if the sheets $S_m$ and $S_n$ cross above $e_\alpha^1$, then there is a Reeb chord $\tilde{b}^{n,m}_\alpha$.
\end{itemize}
With our choices of the defining functions $\{f_m\}$ of sheets $\{S_m\}$, one can further specify the locations of the Reeb chords $a_+^{m,n}, a_-^{m,n}, b_\alpha^{m,n}$ and $\tilde{b}^{n,m}_\alpha$. For details, we refer the readers to Section 5 of $\cite{rs2}$. There are three types of rigid Morse flow trees with positive puncture at $b_\alpha^{m,n}$, namely the $-\nabla f_{mn}$ flow lines from $b_\alpha^{m,n}$ to $a_-^{m,n}$ and $a_+^{m,n}$; the Morse flow trees with one of its negative punctures at $\tilde{b}^{\bullet,\bullet}$; and the Morse flow trees described in Figure \ref{fig:1-cell}, which have unique internal $Y_0$-vertices. This implies that

\begin{figure}
	\centering
\begin{tikzpicture}
\draw [brown] (0,0) arc [start angle=30, end angle=330, radius=1];
\draw [brown] (0,0) to (5,0);
\draw [brown] (0,-1) to (5,-1);
\draw [brown] (5,-1) arc[start angle=-150,end angle=150, radius=1];
\draw [brown, dash dot] (0,-1) arc [start angle=-30, end angle=30, radius=1];
\draw [brown, dash dot] (5,0) arc [start angle=150, end angle=210, radius=1];

\draw [brown] (0,-3) arc [start angle=30, end angle=330, radius=1];
\draw [brown] (0,-3) to (5,-3);
\draw [brown] (0,-4) to (5,-4);
\draw [brown] (5,-4) arc[start angle=-150,end angle=150, radius=1];
\draw [brown, dash dot] (0,-4) arc [start angle=-30, end angle=30, radius=1];
\draw [brown, dash dot] (5,-3) arc [start angle=150, end angle=210, radius=1];

\draw (-0.866,-0.5) node[circle,fill,inner sep=1pt]{};
\draw (5.866,-3.5) node[circle,fill,inner sep=1pt]{};
\draw (2,-0.1) node[circle,fill,inner sep=1pt]{};
\draw (3,-0.8) node[circle,fill,inner sep=1pt]{};
\draw (3,-3.8) node[circle,fill,inner sep=1pt]{};
\draw (4,-3.7) node[circle,fill,inner sep=1pt]{};

\draw [blue] [decoration={markings, mark=at position 1/2 with {\arrow{>}}},postaction={decorate}] (2,-0.7) to (2,-0.1);
\draw [blue] [decoration={markings, mark=at position 1/2 with {\arrow{>}}},postaction={decorate}] (3,-0.8) to [in=-45,out=180] (2,-0.7);
\draw [blue] [decoration={markings, mark=at position 1/2 with {\arrow{>}}},postaction={decorate}] (2,-0.7) to [in=0,out=180] (-0.866,-0.5);

\draw [blue] [decoration={markings, mark=at position 1/2 with {\arrow{>}}},postaction={decorate}] (3,-3.8) to [in=135,out=45] (4,-3.3);
\draw [blue] [decoration={markings, mark=at position 1/2 with {\arrow{>}}},postaction={decorate}] (4,-3.3) to (4,-3.7);
\draw [blue] [decoration={markings, mark=at position 1/2 with {\arrow{>}}},postaction={decorate}] (4,-3.3) to [in=180,out=0] (5.866,-3.5);

\node at (-0.866,-0.8) {\footnotesize $a_-^{k,n}$};
\node at (2,0.2) {\footnotesize $b_\alpha^{m,k}$};
\node at (3,-0.5) {\footnotesize $b_\alpha^{m,n}$};
\node at (5.866,-3.2) {\footnotesize $a_+^{m,k}$};
\node at (2.7,-3.8) {\footnotesize $b_\alpha^{m,n}$};
\node at (4.4,-3.7) {\footnotesize $b_\alpha^{k,n}$};
\end{tikzpicture}
\centering
\caption{Rigid Morse flow trees in neighborhoods of 1-cells}
\label{fig:1-cell}
\end{figure}
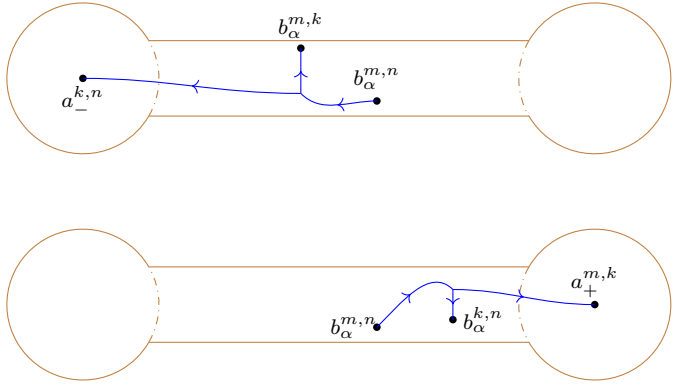

\begin{equation}
\partial b_\alpha^{m,n}=a_{m,n}^++a_{m,n}^-+\sum_{m<k<n}a^{m,k}_+b_\alpha^{k,n}+\sum_{m<k<n}b_\alpha^{m,k}a_-^{k,n}+x
\end{equation}
in $\mathcal{C}E(\Lambda_{p,q,r})$ over $\mathbb{Z}/2$, where the term $x$ corresponds to those flow trees with one of their negative punctures at $\tilde{b}^{\bullet,\bullet}$. Note that if the sheets $S_k$ and $S_{k+1}$ meet at a cusp edge above $e_\alpha^1$, then $a^{k,k+1}_-=1$ in the above formula, which corresponds to an $e$-vertex. By Lemma 8.5 of $\cite{rs2}$, we have $x=0$ in the quotient dg algebra $\mathcal{C}E'(\Lambda_{p,q,r})$. Again, what we have learned in Section \ref{section:subo} enables us to determine the signs of the rigid Morse flow trees in Figure \ref{fig:1-cell}, and we have
\begin{equation}\label{eq:grB}
\begin{aligned}
\partial b_\alpha^{m,n}&=\varepsilon_{\alpha,+}(m,n)a_{m,n}^++\varepsilon_{\alpha,-}(m,n)a_{m,n}^- \\
&+\sum_{m<k<n}\varepsilon_{\alpha,3}(m,k,n)a^{m,k}_+b_\alpha^{k,n}+\sum_{m<k<n}\varepsilon_{\alpha,4}(m,k,n)b_\alpha^{m,k}a_-^{k,n},
\end{aligned}
\end{equation}
where $\varepsilon_{\alpha,+}(m,n),\varepsilon_{\alpha,-}(m,n),\varepsilon_{\alpha,3}(m,k,n),\varepsilon_{\alpha,4}(m,k,n)\in\{\pm1\}$ depend on $m,k$ and $n$.
\bigskip

The considerations above show that our computations of the cellular dg algebra $\mathcal{C}(\Lambda_{p,q,r})$ in Sections \ref{section:computation} and \ref{section:comparison} already determine the Chekanov-Eliashberg algebra $\mathcal{C}E(\Lambda_{p,q,r})$ over any field $\mathbb{K}$ up to signs in front of each of the monomials in the differentials of the generators of $\mathcal{C}E'(\Lambda_{p,q,r})$. More precisely, denote by $\mathcal{G}_{p,q,r}^{\bm{\varepsilon}}$ the dg algebra whose underlying graded associative $\Bbbk$-algebra structure is the same as $\mathcal{G}_{p,q,r}$, but whose differentials of the generators differ from that of $\mathcal{G}_{p,q,r}$ in the sense that there is a sign $(-1)^{\varepsilon_i}$ before every term appearing in the non-trivial differentials of the generators of $\mathcal{G}_{p,q,r}$. For example, the non-trivial differentials of the generators in the dg algebra $\mathcal{G}_{2,2,2}^{\bm{\varepsilon}}$ are given by
\begin{equation}
da_1^\ast=(-1)^{\varepsilon_1}b_2c_2+(-1)^{\varepsilon_2}b_3c_3,da_2^\ast=(-1)^{\varepsilon_3}b_1c_1+(-1)^{\varepsilon_4}b_3c_3,
\end{equation}
\begin{equation}
db_1^\ast=(-1)^{\varepsilon_5}c_1a_2,db_2^\ast=(-1)^{\varepsilon_6}c_2a_1,db_3^\ast=(-1)^{\varepsilon_7}c_3a_1+(-1)^{\varepsilon_8}c_3a_2,
\end{equation}
\begin{equation}
dc_1^\ast=(-1)^{\varepsilon_9}a_2b_1,dc_2^\ast=(-1)^{\varepsilon_{10}}a_1b_2,dc_3^\ast=(-1)^{\varepsilon_{11}}a_1b_3+(-1)^{\varepsilon_{12}}a_2b_3,
\end{equation}
\begin{equation}
dz_A=(-1)^{\varepsilon_{13}}a_1a_1^\ast+(-1)^{\varepsilon_{14}}a_2a_2^\ast+(-1)^{\varepsilon_{15}}c_1^\ast c_1+(-1)^{\varepsilon_{16}}c_2^\ast c_2+(-1)^{\varepsilon_{17}}c_3^\ast c_3,
\end{equation}
\begin{equation}
dz_B=(-1)^{\varepsilon_{18}}a_1^\ast a_1+(-1)^{\varepsilon_{19}}a_2^\ast a_2+(-1)^{\varepsilon_{20}}b_1b_1^\ast+(-1)^{\varepsilon_{21}}b_2b_2^\ast+(-1)^{\varepsilon_{22}}b_3b_3^\ast,
\end{equation}
\begin{equation}
\begin{aligned}
&dz_{P}=(-1)^{\varepsilon_{23}}c_1c_1^\ast+(-1)^{\varepsilon_{24}}b_1^\ast b_1,dz_{Q}=(-1)^{\varepsilon_{25}}c_2c_2^\ast+(-1)^{\varepsilon_{26}}b_2b_2^\ast, \\
&dz_{R}=(-1)^{\varepsilon_{27}}c_3c_3^\ast+(-1)^{\varepsilon_{28}}b_3^\ast b_3.
\end{aligned}
\end{equation}
\begin{lemma}\label{lemma:sign}
Let $\mathbb{K}$ be any field, then there exists a vector $\bm{\varepsilon}=(\varepsilon_1,\dots,\varepsilon_{2(p+q+r)+16})\in(\mathbb{Z}/2)^{2(p+q+r)+16}$ such that $\mathcal{C}E(\Lambda_{p,q,r})$ over $\mathbb{K}$ is quasi-isomorphic to the dg algebra $\mathcal{G}_{p,q,r}^{\bm{\varepsilon}}$.
\end{lemma}
\begin{proof}
Since there is no swallowtail singularities in the Legendrian front of $\Lambda_{p,q,r}$, it follows from the arguments in Sections 5 and 6 of $\cite{rs2}$ that each term in the non-trivial differentials of the generators in $\mathcal{C}E(\Lambda_{p,q,r})$ corresponds to a unique rigid Morse flow tree (instead of an odd number of them), and no two of the rigid Morse flow trees are cancelled in the differentials of $\mathcal{C}E(\Lambda_{p,q,r})$ when $\mathbb{K}=\mathbb{Z}/2$. In particular, since $\mathcal{C}E(\Lambda_{p,q,r})$ and $\mathcal{C}(\Lambda_{p,q,r})$ are related by a stable tame isomorphism, the same argument as in the proof of Proposition \ref{proposition:222} applies when passing from $\mathcal{C}(\Lambda_{p,q,r})$ to $\mathcal{C}E(\Lambda_{p,q,r})$, and from $\mathbb{Z}/2$ to an arbitrary field $\mathbb{K}$, which cancels all of the generators in $\mathcal{C}E(\Lambda_{p,q,r})$ except for those in the quotient dg algebra $\mathcal{C}'(\Lambda_{p,q,r})$. This identifies the quotient dg algebra $\mathcal{C}E'(\Lambda_{p,q,r})$ and $\mathcal{G}_{p,q,r}$ as $\Bbbk$-bimodules.

On the other hand, observe from our computations in Sections \ref{section:computation} and \ref{section:comparison} that all of the generators in $\mathcal{C}(\Lambda_{p,q,r})$ associated to 0-cells are cancelled out in the quotient dg algebra $\mathcal{C}'(\Lambda_{p,q,r})$, therefore they do not contribute to $\mathcal{C}E'(\Lambda_{p,q,r})$. Similarly, since all the generators of $\mathcal{C}(\Lambda_{p,q,r})$ associated to 2-cells can be cancelled out except for those in the matrices $C_7,C_{11}$ and $C_{12}$, in order to determine the dg algebra $\mathcal{C}E'(\Lambda_{p,q,r})$ over $\mathbb{K}$, it suffices to use the formulae (\ref{eq:grC7}) (together with its analogues for the elementary squares $\square_{11}$ and $\square_{12}$) and (\ref{eq:grB}). More precisely, let $A_i$ be the matrix of generators in the cellular dg algebra $\mathcal{C}(\Lambda_{p,q,r})$ associated to the 0-cell $e_i^0$. Since the cellular decomposition $\mathcal{E}_{||}$ is clearly finer than the one we used in Section \ref{section:computation} for computing $\mathcal{C}(\Lambda_{p,q,r})$, one can regard $A_i$ as the matrix of generators associated to a 0-cell $e_i^0$ in $\mathcal{E}_{||}$. Since there is no elementary squares of Types V, VI, VIII and XII in $\mathcal{E}_\pitchfork$, under the identification $\mathcal{C}_{||}(\Lambda_{p,q,r})\cong\mathcal{C}E(\Lambda_{p,q,r})/\mathcal{I}$ proved in Section 8.2 of $\cite{rs2}$, where $\mathcal{I}$ is the ideal generated by the exceptional generators in $\mathcal{C}E(\Lambda_{p,q,r})$, it makes sense to talk about the matrix $A_i$ of generators in $\mathcal{C}E(\Lambda_{p,q,r})$. After taking into account the orientation data of Morse flow trees, denote by $A_i^\#$ the corresponding matrix of generators in $\mathcal{C}E(\Lambda_{p,q,r})$ over $\mathbb{K}$. Applying (\ref{eq:grB}) iteratively, one can determine the value of $A_6^\#$ in the quotient dg algebra $\mathcal{C}E'(\Lambda_{2,2,2})$, from which the differential $\partial b_7^{4,5}$ in $\mathcal{C}E'(\Lambda_{p,q,r})$ can be computed. Use (\ref{eq:grB}) again to determine the signs of rigid Morse flow trees with positive punctures at $b_7^{m,n}$, the value of $A_7^\#$ in $\mathcal{C}E'(\Lambda_{p,q,r})$ can be computed, from which one gets the differentials of the generators in $C_7$ using (\ref{eq:grC7}). The same method can be applied to the Reeb chords above the elementary squares $\square_{11}$ and $\square_{12}$. This enables us to conclude that there is a term by term identification between the differentials of the  generators in $\mathcal{C}E'(\Lambda_{p,q,r})$ and the corresponding generators in $\mathcal{G}_{p,q,r}$, and the quasi-isomorphism $\mathcal{C}E(\Lambda_{p,q,r})\cong\mathcal{G}_{p,q,r}^{\bm{\varepsilon}}$ follows.
\end{proof}

\begin{proposition}
Let $\mathbb{K}$ be any field. There is a quasi-isomorphism
\begin{equation}
\mathcal{C}E(\Lambda_{p,q,r})\cong\mathcal{G}_{p,q,r}
\end{equation}
between the Chekanov-Eliashberg algebra of $\Lambda_{p,q,r}$ and the Ginzburg algebra $\mathcal{G}_{p,q,r}$ over $\mathbb{K}$.
\end{proposition}
\begin{proof}
Let $r=2$. By Lemma \ref{lemma:sign}, there exists a vector $\bm{\varepsilon}\in(\mathbb{Z}/2)^{2(p+q)+20}$ such that $\mathcal{C}E(\Lambda_{p,q,2})\cong\mathcal{G}_{p,q,2}^{\bm{\varepsilon}}$ as augmented dg algebras. On the other hand, it follows from Theorem \ref{theorem:gem} that
\begin{equation}
R\mathrm{Hom}_{\mathcal{G}_{p,q,2}^{\bm{\varepsilon}}}(\Bbbk,\Bbbk)\cong\mathcal{V}_{p,q,2},
\end{equation}
where the $R\mathrm{Hom}$ on the left-hand side is taken with respect to the trivial augmentation on $\mathcal{G}_{p,q,2}^{\bm{\varepsilon}}$. Since the Weinstein manifold $M_{p,q,2}$ can be regarded as the fiber of the double suspension of a Lefschetz fibration on $\mathbb{C}^2$, Corollary \ref{corollary:double-susp} of Section \ref{section:suspension} can be applied to show that $\mathcal{V}_{p,q,2}=\mathcal{A}_{p,q,2}\oplus\mathcal{A}_{p,q,2}^\vee[-3]$. This determines completely the signs in $\bm{\varepsilon}$. For example, since
\begin{equation}
\mu^2_{\mathcal{V}_{p,q,2}}(c_1^\vee,b_1^\vee)=-(a_2^\ast)^\vee
\end{equation}
and
\begin{equation}
\mu^2_{\mathcal{V}_{p,q,2}}(c_3^\vee,b_3^\vee)=(a_2^\ast)^\vee
\end{equation}
in $\mathcal{V}_{p,q,2}$, we get $\varepsilon_3=0$ and $\varepsilon_4=1$, where as in Section \ref{section:CY algebra}, $g^\vee\in\mathcal{V}_{p,q,r}$ denotes the dual generator of $g\in\mathcal{G}_{p,q,r}$ under Koszul duality. The other signs can be similarly determined and it follows that $\mathcal{C}E(\Lambda_{p,q,2})\cong\mathcal{G}_{p,q,2}$ over $\mathbb{K}$.

Passing from $\mathcal{C}E(\Lambda_{p,q,2})$ to the general case of $\mathcal{C}E(\Lambda_{p,q,r})$ essentially requires a computation of the signs of rigid Morse flow trees determined by an $A_{r-2}$ link of unknots $\Lambda_{r-2}$. In fact, by the discussions above and the computations in Section \ref{section:comparison}, the only rigid Morse flow trees that contribute to the differentials of the additional generators in $\mathcal{C}E'(\Lambda_{p,q,r})$ are those whose positive punctures correspond to Reeb chords which start at $\Lambda_{R_{j-1}}$ and end $\Lambda_{R_j}$ for some $j$. Let $c_{12}^{m,n}$ be the generator associated to such a Reeb chord. It follows from our assumptions that $b_{23}^{\sigma_0(m),\sigma_0(n)}=0$ in $\mathcal{C}E'(\Lambda_{p,q,r})$, so
\begin{equation}
\partial c_{12}^{m,n}=\sum_{m<k<n}\varepsilon_1(m,k,n)a_{14}^{\sigma_0(m),\sigma_0(k)}c_{12}^{k,n}+\sum_{m<k<n}\varepsilon_2(m,k,n)c_{12}^{m,k}a_{14}^{\sigma_0(k),\sigma_0(n)}.
\end{equation}
Let $c_{12}^{m',n'}$ be the generator associated to the Reeb chord of the Legendrian unknot $\Lambda_{R_j}$, then from the configuration of the Legendrian front of an $A_{r-2}$ link of unknots, it is clear that $m'=n-1$. Consider the terms $a_{14}^{\sigma_0(m),\sigma_0(n-1)}c_{12}^{n-1,n}$ in $\partial c_{12}^{m,n}$ and $c_{12}^{n-1,n}a_{14}^{\sigma_0(n),\sigma_0(n')}$ in $\partial c_{12}^{n-1,n'}$, denote by $\Xi_1$ and $\Xi_2$ the corresponding rigid Morse flow trees. (\ref{eq:grB}), together with our computations in Section \ref{section:subo}, shows that
\begin{equation}
a_{14}^{\sigma_0(m),\sigma_0(n-1)}=a_{14}^{\sigma_0(n),\sigma_0(n')}=b_{22}^{n-1,n}
\end{equation}
in $\mathcal{C}E'(\Lambda_{p,q,r})$. It follows that $\varepsilon(\Xi_1)=\varepsilon(\Xi)$ and $\varepsilon(\Xi_2)=\varepsilon(\Xi^\mathit{op})$, where $\Xi$ and $\Xi^\mathit{op}$ are the rigid Morse flow trees depicted in Figure \ref{fig:y0tree2}. The computation at the end of Section \ref{section:subo} then allows us to conclude that $\varepsilon(\Xi_1)=-\varepsilon(\Xi_2)$. One can make the choice of capping orientations so that $\varepsilon(\Xi_1)=-1$ and $\varepsilon(\Xi_2)=1$.

Using the notations as in the proof of Proposition \ref{proposition:pqr}, we have proved that
\begin{equation}
\partial z_{R_j}=z_{j-1}^\ast z_j-z_jz_{j+1}^\ast
\end{equation}
in $\mathcal{C}E'(\Lambda_{p,q,r})$. Moreover, the additional term $z_1z_1^\ast$ in the differential of $z_{Q_1}$ has negative sign, since its associated rigid Morse flow tree has the same sign as $\Xi$. This proves the quasi-isomorphism $\mathcal{C}E'(\Lambda_{p,q,r})\cong\mathcal{G}_{p,q,r}$.
\end{proof}

\end{document}